\documentclass[11pt,a4paper]{article}
\usepackage{}
\usepackage{amsfonts}
\usepackage{mathrsfs}
\usepackage{multirow}
\usepackage{color}
\usepackage{url}
\usepackage{booktabs}
\usepackage{listings}
\usepackage{xcolor}
\usepackage{algorithm}  
\usepackage{algpseudocode} 
\usepackage{hyperref}
\usepackage{epsfig}
\usepackage{epstopdf}
\usepackage[T1]{fontenc}
\usepackage{geometry}
\usepackage{float}
\usepackage{amsbsy,amsmath,latexsym,amsfonts, epsfig, color, authblk, amssymb, graphics, bm}
\usepackage{epsf,slidesec,epic,eepic}
\usepackage{fancybox}
\usepackage{fancyhdr}
\usepackage{mathrsfs}
\usepackage{setspace}
\usepackage{nccmath}
\usepackage{cases}
\usepackage{subcaption}
\usepackage{makecell}
\newtheorem{theorem}{Theorem}[section]
\newtheorem{lemma}{Lemma}[section]
\newtheorem{definition}{Definition}[section]
\newtheorem{example}{Example}

\newtheorem{proposition}{Proposition}[section]
\newtheorem{corollary}{Corollary}[section]
\newtheorem{assumption}{Assumption}
\newtheorem{remark}{Remark}[section]

\newtheorem{claim}{Claim}

\newenvironment{proof}{{\noindent \bf Proof:}}{\hfill$\Box$\medskip}

\definecolor{lred}{rgb}{1,0.8,0.8}
\definecolor{lblue}{rgb}{0.8,0.8,1}
\definecolor{dred}{rgb}{0.6,0,0}
\definecolor{dblue}{rgb}{0,0,0.5}
\definecolor{dgreen}{rgb}{0,0.5,0.5}

\title{Proximal-Linearized NEP Method for Composite Optimization with Set-Inclusion and Manifold Constraints}

\author{Hao He\footnote{School of Mathematics, South China University of Technology, Guangzhou, China},\ Ruyu Liu\footnote{School of Mathematics, South China University of Technology}\ \ {\rm and}\ \ Shaohua Pan\footnote{School of Mathematics, South China University of Technology (shhpan@scut.edu.cn)}}

\begin{document}

\maketitle

\begin{abstract}
This paper focuses on difference-of-convex (DC) composite optimization problems with set-inclusion and manifold constraints. We develop an inexact proximal-linearized nonsmooth exact penalty (NEP) method that penalizes the conic constraint while retaining the manifold constraint explicitly. At each outer iteration, the method approximately solves a strongly convex proximal-linearized subproblem over the current tangent space, using computable inexactness criteria, adaptive parameters, and a retraction-based acceptance mechanism. Under boundedness of the iterates and penalty parameters, we establish an $O(\epsilon^{-2})$ worst-case outer iteration complexity for finding an $\epsilon$-approximate KKT point of the original problem. We further develop a complexity-transfer framework relating uniform primal objective-gap and KKT-residual decay, or dual objective-gap decay, to total inner-iteration complexity bounds. This yields an $O(\epsilon^{-4})$ overall oracle complexity with either accelerated semi-proximal ADMM or FISTA applied to the dual subproblem as the inner solver. Using semi-proximal ADMM as the inner solver improves the bound to $O(\epsilon^{-2}\log(1/\epsilon))$ under an additional uniform error-bound condition. We further establish convergence of the entire iterate sequence to a KKT point under the same boundedness assumption and the Kurdyka–Łojasiewicz property of the associated potential function. Numerical experiments demonstrate the competitive performance of the proposed method on row-sparse PCA, clustering, robust SVMs, and nonconvex SOCPs.
\end{abstract}


\section{Introduction}\label{sec1}

Let $\mathbb{X}$ and $\mathbb{Y}$ be real Euclidean spaces endowed with inner products and their induced norms, both denoted by $\langle\cdot,\cdot\rangle$ and $\|\cdot\|$, respectively. Let $\mathcal{M}$ be a closed $\mathcal{C}^2$-embedded submanifold of $\mathbb{X}$, and let $K\subset\mathbb{Y}$ be a simple closed convex set in the sense that projections onto $K$ under a
prescribed norm $\vert\!\|\cdot\|\!\vert$ on $\mathbb{Y}$
can be computed efficiently. We consider the following optimization problem with set-inclusion and manifold constraints:
\begin{equation}\label{Rcprob}    
\min_{x\in g^{-1}(K)\cap\mathcal{M}}\Theta(x):=f(x)+\vartheta(x)-h(x),
\end{equation}
where $f,\vartheta,h\!:\mathbb{X}\to\mathbb{R}$ and $g\!:\mathbb{X}\to\mathbb{Y}$ satisfy the conditions stated in Assumption \ref{ass0} below. Unless otherwise stated, we assume that \eqref{Rcprob} has a nonempty feasible set $g^{-1}(K)\cap\mathcal{M}$. 
\begin{assumption}\label{ass0}
{\bf(i)} $f$ is differentiable with locally Lipschitz gradient, and $g$ is differentiable with locally Lipschitz Jacobian; {\bf(ii)} $\vartheta$ and $h$ are convex functions. 
\end{assumption}

Model \eqref{Rcprob} encompasses a broad class of optimization problems, including problems with convex composite terms and certain extended-valued convex regularizers whose effective domains can be incorporated into the constraint structure through suitable reformulations. In contrast to the formulation in Le Thi et al. \cite[Eq. (46)]{Le2024}, our model explicitly distinguishes the manifold constraint from the set-inclusion constraint. Such structured problems arise naturally in data analysis and scientific computing. 
\begin{example}\label{Example1}
 Sparse and row-sparse principal component analysis provide representative examples of optimization problems with manifold constraints in \cite{Xiao2021,Huang2025}. Consider 
 \begin{equation}\label{sPCA}
 \min_{X\in\mathbb{R}^{n\times p}}\Big\{-\frac{1}{2}{\rm tr}(X^{\top}AX)+\mu\varphi(X)\ \ {\rm s.t.}\ \ XX^{\top}u-u=0,\,X^{\top}X-I_p=0\Big\},
 \end{equation}
 where $A\in\mathbb{R}^{n\times n}$ is a symmetric data matrix, $\mu>0$ is a regularization parameter, $\varphi$ is a convex regularizer promoting sparsity or row sparsity, $u\in\mathbb{R}^n$ is either a positive vector or the zero vector, and $I_p$ is the $p\times p$ identity matrix. Model \eqref{sPCA} is a special case of \eqref{Rcprob} with $\mathcal{M}={\rm St}(n,p):=\big\{X\in\mathbb{R}^{n\times p}\,|\,X^{\top}X=I_p\big\}, h\equiv 0, g(X)=XX^{\top}u-u$ and $K=\{0\}$. Alternatively, both equality constraints can be incorporated into the set-inclusion constraint by setting $g(X)=(X^{\top}X-I_p,XX^{\top}u-u)$, $K=\{0\}\times\{0\}$ and $\mathcal{M}=\mathbb{X}=\mathbb{R}^{n\times p}$.  
\end{example}
\begin{example}\label{Example2}
The well-known $k$-means clustering problem in unsupervised machine learning is formulated in Carson et al. \cite{carson2017} as a special case of the following nonnegative orthogonal model
\begin{equation}\label{RNO}
\min_{X\in\mathbb{R}^{n\times p}}\Big\{f(X)\ \ {\rm s.t.}\ \ X^{\top}X-I_p=0,\,X\in \mathbb{R}_{+}^{n\times p},\,u\in{\rm col}(X)\Big\},
\end{equation}
where $f:\mathbb{R}^{n\times p}\to\mathbb{R}$ is a smooth function, $\mathbb{R}_{+}^{n\times p}$ is the nonnegative cone in $\mathbb{R}^{n\times p}$, ${\rm col}(X)$ denotes the column space of $X$, and $u\in\mathbb{R}^n$ is a  positive vector. Since $u\in{\rm col}(X)$ is equivalent to $XX^\top u=u$ under $X^\top X=I_p$, model \eqref{RNO} is a special case of \eqref{Rcprob} with $\vartheta\equiv h\equiv 0, g(X)=(X,XX^{\top}u-u),K=\mathbb{R}_{+}^{n\times p}\times \{0\}$ and $\mathcal{M}={\rm St}(n,p)$. 
\end{example}
\begin{example}\label{Example3}
Motivated by DC-regularized second-order cone optimization in L\'{o}pez et al. \cite{Lopez2018} and nonsmooth optimization on Riemannian manifolds in Chen et al. \cite{Chen2020b}, we consider the following DC composite problems with second-order cone constraints and a generalized sphere constraint:
 \begin{equation}\label{SOCP}
 \min_{x\in\mathbb{R}^{n}}\Big\{f(x)+\vartheta(x)-h(x)\ \ {\rm s.t.}\ \ g(x)\in K^{m_1}\times\cdots\times K^{m_l},\ x^{\top}Bx=1\Big\}
 \end{equation}
 where $f,\vartheta,h$ and $g$ are as specified in \eqref{Rcprob} with $\mathbb{X}=\mathbb{R}^n$ and $\mathbb{Y}=\mathbb{R}^m$ for $m=\sum_{i=1}^l m_i$, each $K^{m_i}:=\{(t,u)\in\mathbb{R}\times\mathbb{R}^{m_i-1} \mid \|u\|_2\le t\}$ is a second-order cone, and $B\in\mathbb{R}^{n\times n}$ is symmetric positive definite. Model \eqref{SOCP} is a special case of \eqref{Rcprob} with $K=K^{m_1}\times\cdots\times K^{m_l}$ and $\mathcal{M}=\{x\in\mathbb{R}^n\mid x^{\top}Bx=1\}$.
\end{example}
\begin{example}\label{Example4}
Consider the following nonnegative sparse canonical correlation model:
\begin{align}\label{NSCCA-model}
&\min_{X\in\mathbb{R}_{+}^{p\times r}, Y\in\mathbb{R}_{+}^{q\times r}}-\frac{1}{2}{\rm tr}(X^{\top}A^{\top}BY)+\nu_1\|X\|_{2,1}+\nu_2\|Y\|_{2,1}\nonumber\\
&\qquad\ \ {\rm s.t.}\ \ X^{\top}A^{\top}AX=I_{r},Y^{\top}B^{\top}BY=I_r,
\end{align}
where $A\in\mathbb{R}^{n\times p}$ and $B\in\mathbb{R}^{n\times q}$ are data matrices such that $A^{\top}A$ and $B^{\top}B$ are positive definite, $1\le r\le\min\{p,q\}$ is the number of canonical pairs, and $\|X\|_{2,1}:=\sum_{i=1}^p\|X_{i,:}\|_2$ denotes the $\ell_{2,1}$-norm of $X$. Model \eqref{NSCCA-model} is a special case of \eqref{Rcprob} with $\mathbb{X}=\mathbb{R}^{p\times r}\times\mathbb{R}^{q\times r}$, $\mathcal{M}=\!\{(X,Y)\in\mathbb{X} \mid X^{\top}A^{\top}AX=I_r,Y^{\top}B^{\top}BY=I_r\}$, and $K=\mathbb{R}_{+}^{p\times r}\times\mathbb{R}_{+}^{q\times r}$. Removing the nonnegativity constraints on $X$ and $Y$ yields the model considered in \cite{Chen2020}.
\end{example}

These examples illustrate the versatility of model \eqref{Rcprob} for structured optimization with nonsmooth regularization, set-inclusion constraints, and manifold geometry.
 
\subsection{Related research}\label{sec1.1}

Nonsmooth exact penalty (NEP) methods have been extensively studied since the pioneering works of Eremin \cite{EREMIN1966}, Zangwill \cite{Zangwill1967}, and Pietrzykowski \cite{Pietrzykowski1969}, with subsequent developments in general constrained optimization \cite{Burke1991}, mathematical programs with equilibrium constraints \cite{Luo1996}, and rank-constrained optimization \cite{Qian2023}. More recently, the iteration complexity of nonsmooth exact penalty (NEP) methods for nonconvex constrained optimization has also received attention. 

Cartis et al. \cite{Cartis2011} developed a distance-based NEP method for optimization problems with smooth objective and constraint mappings, covering nonlinear equality constraints and their extension to inequality constraints. These problems are special cases of \eqref{Rcprob} with $\vartheta\equiv h\equiv 0$, $\mathcal{M}=\mathbb{X}=\mathbb{R}^n$, and $K={0}$ or $\mathbb{R}_+^m$. Their method uses a steering procedure to update the penalty parameter, and approximately solves each penalized problem using a first-order trust-region or quadratic-regularization method until a prescribed first-order criticality measure falls below a fixed tolerance. They established worst-case problem-evaluation complexity bounds of $O(\epsilon^{-2})$ when the penalty parameters are uniformly bounded independently of $\epsilon$, and $O(\epsilon^{-5})$ when the penalty parameters are unbounded but the outer iterates remain bounded. These bounds guarantee termination at either an approximate KKT point of the original problem or an infeasible critical point of a constraint violation measure. Thus, the infeasible-critical-point alternative cannot be ruled out.

Diouane et al. \cite{Diouane2026-II} investigated NEP methods for nonlinear equality-constrained optimization, where each penalized problem is approximately solved to a prescribed stationarity tolerance using a proximal-type algorithm, implemented via a proximal modified quasi-Newton method. Both the stationarity tolerance and the penalty parameter are updated according to a feasibility measure. Assuming that the penalty parameter sequence satisfies a condition ensured by the Mangasarian-Fromovitz constraint qualification (MFCQ), they established an overall inner-iteration complexity of
\(O(|\log(\epsilon/\epsilon_0)|\epsilon^{-2})\) for finding either an approximate KKT point or an infeasible critical point of the feasibility measure. When the penalty parameters grow unbounded, the bound becomes $O(\epsilon^{-8})$. Since this complexity is expressed in terms of the number of inner iterations, it is also important to examine the computational cost of each inner iteration. For their proximal modified quasi-Newton method, each inner iteration involves evaluating a proximal mapping associated with the affine-composite penalty term $s\mapsto \tau_k\|c(x_{k,j})+c'(x_{k,j})s\|$ and approximately minimizing a quadratic model regularized by the same term, where $x_{k,j}$ is the current inner iterate. 

The above complexity-oriented NEP methods mainly address classical nonlinear programming with smooth objectives and conventional constraints. In contrast, this work considers nonsmooth DC composite optimization with set-inclusion and manifold constraints, for which related works are discussed below. 

DC composite optimization with constraints has been investigated from various perspectives. Ye et al. \cite{Ye2023} proposed inexact proximal-linearized DC algorithms for DC programming and  established subsequential convergence under specific inexactness conditions, which differ from the computable criteria developed in this paper. Le Thi et al. \cite{Le2024} developed a penalty-based DC composite algorithm for optimization problems with set-inclusion constraints, while Mordukhovich et al. \cite{Mordukhovich2026} proposed an extended SQP framework for nonsmooth difference programming and established KL-based whole-sequence convergence. Both works rely on exact solutions of the corresponding subproblems. Liu and Boumal \cite{LiuBoumal2020} studied augmented Lagrangian and smoothed exact penalty methods for optimization on manifolds with additional constraints, establishing conditional subsequential convergence results but no full-sequence convergence or complexity bounds. In a related but distinct direction, Deng et al. \cite{Deng2025} developed inexact augmented Lagrangian methods for nonsmooth composite optimization on compact submanifolds and established oracle complexity guarantees. More recently, Jiang et al. \cite{Jiang2025} developed an inexact Riemannian proximal DC framework for nonsmooth DC optimization over manifolds, while He et al. \cite{He2026} proposed an inexact variable metric proximal linearization method for nonsmooth composite optimization on manifolds. Both methods employ inexact tangent-space subproblems and establish complexity guarantees in their respective settings.  

\subsection{Motivation and penalty reformulation}\label{sec1.2}

Although problem \eqref{Rcprob} captures a variety of structured optimization models, the coexistence of set-inclusion and manifold constraints poses significant challenges. In particular, the analyses of the inexact manifold proximal methods in Jiang et al. \cite{Jiang2025} and He et al. \cite{He2026} do not directly extend to an adaptive NEP framework with a manifold constraint. In this setting, a vanishing exact tangent-space direction may yield only a weak stationary point of the penalized problem, rather than a KKT point of the original problem. Establishing approximate KKT guarantees therefore requires controlling constraint violations and errors arising from inexact tangent-space subproblem solutions, adaptive penalty parameter updates, and manifold retraction. In addition, existing complexity-oriented NEP methods in Cartis el tal. \cite{Cartis2011} and Diouane et al. \cite{Diouane2026-II} typically require each penalized problem to be solved to a prescribed stationarity tolerance, which may entail substantial computational effort. Avoiding this requirement while retaining complexity guarantees for the original constrained problem poses an additional analytical challenge.

To address these challenges, we adopt a partial penalization strategy that penalizes only the set-inclusion constraint while retaining the manifold constraint explicitly.
Specifically, we consider the penalized problem
\begin{equation}\label{penprob}
\min_{x\in\mathcal{M}}\,\Theta_{\rho}(x):=
f(x)+\vartheta(x)-h(x)+\rho{\rm dist}(g(x),K),
\end{equation}
where $\rho>0$ is the penalty parameter, and
${\rm dist}(\cdot,K)$ denotes the distance to $K$ induced by the prescribed norm $\vert\!\|\cdot\|\!\vert$ on $\mathbb{Y}$. As shown in Lemma \ref{GEP-lemma}, the penalty is globally exact under metric subregularity condition and either coercivity of $\Theta$ or compactness of $\mathcal{M}$: there exists a threshold $\widehat{\rho}>0$ such that every penalized problem \eqref{penprob} with $\rho\geq\widehat{\rho}$ has the same set of global minimizers as \eqref{Rcprob}. However, even for a fixed penalty parameter $\rho$, minimizing the resulting nonconvex function remains computationally challenging. At the stationary point level, Proposition \ref{spoint-relation} establishes a relationship between the KKT points of \eqref{Rcprob} and the stationary points of \eqref{penprob}. This motivates the development of an inexact proximal-linearized NEP framework for approximately solving a sequence of penalized problems. Table~\ref{Compar-tab} summarizes the problem settings, subproblem requirements, and theoretical guarantees of representative related methods. 

\begin{table}
 \centering
 \setlength{\abovecaptionskip}{0pt}
 \setlength{\belowcaptionskip}{-20pt}
 \caption{\centering Summary of related methods and their theoretical guarantees}\label{Compar-tab}
 \renewcommand\arraystretch{2}
 \scalebox{0.68}[0.68]{
 \begin{tabular}{c|c|c|c|c|c|c} 
 \hline    
 {\bf Methods}& $K$ & $\mathcal{M}$ & \makecell[c]{\bf Inner problem \\ solve} &{\bf Complexity bounds} &\makecell[c]{\(\epsilon\)-KKT\\ guarantee}&{\bf Converg.}\\
 \hline
 NEPA \cite{Cartis2011} &\makecell[c]{$\{0\}$\\ or $\mathbb{R}_{+}^m$} & $\mathbb{R}^n$ &\makecell[c]{$\min_{x\in\mathbb X}\Theta_{\rho_k}(x)$\\ inexact solve}&Eval. $O(\epsilon^{-2})$ &$\times$&$-$\\
\hline
NEPA \cite{Diouane2026-II}&$\{0\}$&$\mathbb{R}^n$&\makecell[c]{$\min_{x\in\mathbb X}\Theta_{\rho_k}(x)$\\ inexact solve }& Inner $O(|\log(\epsilon/\epsilon_0)|\epsilon^{-2})$&$\times$&$-$\\
\hline
ESQM \cite{Mordukhovich2026}&$\mathbb{R}^m_-$&$\mathbb{R}^n$&\makecell[c]{Strongly convex quadratic\\ exact solve}& $-$ & $-$ & S $\&$ F\\
\hline
DCCA \cite{Le2024}&$K$&$\mathbb{R}^n$&\makecell[c]{Convex \\ exact solve}& $-$ & $-$ & S \\
\hline
RNEPA \cite{LiuBoumal2020}&$\mathbb{R}^m_-$&General&\makecell[c]{$\min_{x\in\mathbb\mathcal{M}}\widetilde{\Theta}_{\rho_k}(x)$\\ inexact solve} &$-$& $-$& S \\
\hline 
{\bf iPLNEP} &$K$&\makecell[c]{Closed \\$\mathcal{C}^2$-submfd} &\makecell[c]{\!\!$\min\limits_{v\in T_{\!x^k}\mathcal{M}}\!\widehat{\Theta}_{\rho_k,\beta_k}(x^k\!+v;x^k)$\\ inexact solve}&\makecell[c]{Outer $O(\epsilon^{-2})$\\ Oracle $O(\epsilon^{-4})$ \\  Oracle $O(\epsilon^{-2}\log\epsilon^{-1})$ under UEB}& $\checkmark$ & S $\&$ F\\
\hline
\end{tabular}}
\parbox{\linewidth}{\raggedright\footnotesize
\textit{Notes:} Here, $\widetilde{\Theta}_{\rho_k}$ denotes a smooth approximation of $\Theta_{\rho_k}$. Complexity bounds are reported in the units used in the cited works (problem evaluations, inner or outer iterations, or oracle calls) and are not directly comparable across rows. A $\checkmark$ indicates a guarantee of an $\epsilon$-KKT point; $\times$ indicates that an infeasible critical-point outcome cannot be excluded; and $-$ indicates that no corresponding guarantee is established. $S$ and $F$ denote subsequential and full convergence, respectively, under the respective assumptions; $S\&F$ indicates both guarantees. UEB denotes a uniform error bound condition.
}
\end{table}
\subsection{Main contributions}\label{sec1.3}

The main contributions of this work are summarized as follows.

{\bf(i)} We propose an implementable two-loop inexact proximal-linearized NEP method (iPLNEP), featuring computable inexactness criteria for strongly convex tangent-space subproblems, and prove that these criteria can be met in finitely many inner iterations. In particular, the descent-based inexactness criterion imposes a second-order bound on the splitting residual, which plays a key role in  controlling the errors arising from proximal linearization and inexact subproblem solutions. The method also incorporates adaptive proximal and penalty parameter updates and a retraction-based acceptance mechanism. Unlike certain existing NEP methods, iPLNEP does not require each penalized problem to be solved to a prescribed stationarity tolerance.

{\bf(ii)} Under boundedness assumptions on the iterates and penalty parameters, we establish an $O(\epsilon^{-2})$ worst-case
outer iteration complexity bound for iPLNEP to compute an
$\epsilon$-approximate KKT point of the original problem.
Our analysis explicitly controls errors arising from
proximal linearization, inexact subproblem solutions,
and retraction through computable inexactness criteria. 
Unlike certain existing NEP complexity results, our
complexity guarantee does not involve an
infeasible-critical-point alternative. Moreover, under the same boundedness assumptions and the additional KL property of an associated potential function, we establish convergence of the entire iterate sequence to a KKT point of the original problem.

{\bf(iii)} We develop a general complexity-transfer framework that, under the same boundedness assumption, translates uniform decay estimates for the subproblem objective gaps and KKT residuals into total inner-iteration complexity bounds. For dual inner solvers, we further show that uniform decay estimates for dual objective gaps yield corresponding square-root decay estimates for the primal subproblem objective gaps and KKT residuals. Applying this framework, we obtain an \(O(\epsilon^{-4})\) overall oracle complexity when accelerated semi-proximal ADMM in \cite{Sun2025} is employed as an inner solver or FISTA in \cite[Section 10.7]{Beck2017} as a dual inner solver. Under an additional uniform error-bound condition, semi-proximal ADMM in \cite[Appendix B]{Fazel2013}
yields an improved $O(\epsilon^{-2}\log\epsilon^{-1})$ overall oracle complexity bound. We also construct an explicit nonsmooth optimization example over the Stiefel manifold with nonnegativity constraints to demonstrate that this condition can hold. To the best of our knowledge, these are the first overall oracle complexity guarantees for an inexact proximal-linearized NEP method equipped with specific inner solvers for DC composite optimization with set-inclusion and manifold constraints.

The theoretical guarantees are complemented by numerical experiments on structured optimization problems involving orthogonality, nonnegativity, and second-order cone constraints. The numerical results in Sections \ref{sec6.3}-\ref{sec6.6} demonstrate the competitive solution quality and computational efficiency of iPLNEP compared with existing state-of-the-art methods. 
\subsection{Notation}\label{sec1.4}

Throughout this paper, $\mathbb{R}^{n\times p}$ denotes the space of $n\times p$ real matrices equipped with the standard trace inner product and the induced Frobenius norm $\|\cdot\|_{F}$. The product space $\mathbb{Z}:=\mathbb{X}\times\mathbb{Y}$ is equipped with the norm $\|z\|:=\!\sqrt{\|x\|^2\!+\!\|y\|^2}$ for $z=(x;y)$. For a linear mapping $\mathcal{A}$, $\mathcal{A}^*$ denotes its adjoint mapping. For $a\in\mathbb{R}$, let $\lceil a\rceil_{+}:=\max\{0,\lceil a\rceil\}$. Let $\mathbb{N}$ and $\mathbb{N}_{+}$ denote the sets of nonnegative integers and positive integers, respectively. For any $k\in\mathbb{N}_{+}$, we write $[k]:=\{0,\ldots,k\}$ and $[k]_{+}:=\{1,\ldots,k\}$. For any $x,x'\in\mathbb{X}$, $[x,x']$ denotes the line segment joining $x$ and $x'$. For a closed set $C\subset\mathbb{X}$, let $\delta_C$ denote its indicator function, and let $\mathcal{P}_C$ denote the projection mapping onto $C$. For $x\in\mathbb{X}$ and $\delta>0$, let $\mathbb{B}(x,\delta)$ denote the closed balls centered at $x$ with radius $\delta$. We also write $\mathbb{B}_{\mathbb{X}}:=\mathbb{B}(0,1)$. For a mapping $g\!:\mathbb{X}\to\mathbb{Y}$ and a point $x\in\mathbb{X}$, if $g$ is differentiable at $x$, $\nabla g(x)$ denotes the adjoint of its differential $g'(x)\!:\mathbb{X}\to\mathbb{Y}$. For any $x\in\mathcal{M}$, $T_x\mathcal{M}$ and $N_x\mathcal{M}$ denote the tangent and normal spaces to $\mathcal{M}$ at $x$, respectively. 

For a proper function $\varphi\!:\mathbb{X}\to\overline{\mathbb{R}}:=(-\infty,\infty]$ and a point $x\in{\rm dom}\,\varphi$, $\partial\varphi(x)$ denotes the (limiting) subdifferential of $\varphi$ at $x$. When $\varphi$ is the indicator function of a closed set $C\subset\mathbb{X}$, $\partial\varphi(x)$ coincides with the normal cone to $C$ at $x$, denoted by $\mathcal{N}_{C}(x)$. In particular, when $C=\mathcal{M}$, we write $\mathcal{N}_{\mathcal{M}}(x)$ as $N_{x}\mathcal{M}$. The proximal mapping of $\varphi$ associated with parameter $\gamma>0$ is defined as $\mathcal{P}_{\gamma\varphi}(\cdot)\!:=\mathop{\arg\min}_{x'\in\mathbb{X}}\big\{\varphi(x')+\frac{1}{2\gamma}\|x'-\cdot\|^2\big\}$. 
\section{Preliminaries}\label{sec2}

We begin with the following characterization of the subdifferential
of ${\rm dist}(\cdot,K)$, which follows from applying
\cite[Theorem 10.13]{RW98} to the function
$(y',y)\mapsto |\!\lVert y'-y\rVert\!|+\delta_K(y')$.
\begin{lemma}\label{subdiff-dist}
 At any $\overline{y}\in\mathbb{Y}$, $\partial{\rm dist}(\overline{y},K)=\bigcup_{x^*\in\mathcal{P}_{K}(\overline{y})}[(-\partial|\!\lVert \cdot\lVert\!|(x^*\!-\!\overline{y}))\cap\mathcal{N}_{K}(x^*)]$ if $\overline{y}\notin K$; otherwise, $\partial{\rm dist}(\overline{y},K)=\{u\in\mathcal{N}_K(\overline{y})\mid |\!\lVert u\lVert\!|_*\le 1\}$, where $\mathcal{P}_K(\cdot)$ denotes the projection mapping onto $K$ induced by $|\!\lVert \cdot\lVert\!|$, and $|\!\lVert \cdot\lVert\!|_*$ is the dual norm of $|\!\lVert \cdot\lVert\!|$.  
\end{lemma}
\subsection{CQs and global exact penalty}\label{sec2.1} 
\begin{definition}\label{def-ECQ}
For problem \eqref{Rcprob}, the extended CQ is said to hold at  $\overline{x}\in\mathcal{M}$ if there is no $\overline{\xi}\in\mathbb{Y}$ satisfying $\overline{\xi}\in\partial{\rm dist}(g(\overline{x}),K)$, $|\!\lVert \overline{\xi}\lVert\!|_*=1$ and $0\in\nabla g(\overline{x})\overline{\xi}+N_{\overline{x}}\mathcal{M}$.
\end{definition}

By Lemma \ref{subdiff-dist}, when $\mathcal{M}=\mathbb{X}=\mathbb{R}^n$, the extended CQ in Definition \ref{def-ECQ} is weaker than the CQ in Le Thi et al. \cite[Eq.~(53)]{Le2024}. When $g(\overline{x})\in K=\mathbb{R}_{+}^m$, the two CQs coincide, and reduce to the standard MFCQ. Moreover, at any $\overline{x}\in\mathcal{M}\cap g^{-1}(K)$, the extended CQ  coincides with the CQ in Rockafellar and Wets \cite[Example 9.44]{RW98}, applied to $F(x)\!:=(g(x);x)$ for $x\in\mathbb{X}$ and $D\!=K\times\mathcal{M}$. The latter CQ is equivalent to the metric regularity of $\mathcal{F}(\cdot):=F(\cdot)-D$ at $(\overline{x},0)\in{\rm gph}\,\mathcal{F}$, which implies its metric subregularity at this point. For $\overline{x}\in\!\mathcal{M}\cap g^{-1}(K)$, the metric subregularity of $\mathcal{F}$ at $(\overline{x},0)$ implies that of the multifunction $\mathcal{G}$ at the same point, where 
\begin{equation}\label{MG-map}
\mathcal{G}(x):=\left\{\begin{array}{cl}
	g(x)-K & {\rm if}\ x\in\mathcal{M},\\
	\emptyset &{\rm otherwise}.
\end{array}\right.
\end{equation}
Consequently, the extended CQ of \eqref{Rcprob} at any feasible point $\overline{x}$ implies the metric subregularity of $\mathcal{G}$ at $(\overline{x},0)$, namely, the MSCQ for \eqref{Rcprob} at $\overline{x}$. By Lemma \ref{GEP-lemma} below, \eqref{penprob} is a global exact penalty of \eqref{Rcprob} under the MSCQ.
\begin{lemma}\label{GEP-lemma}
Let $S^*$ denote the set of global optimal solutions of \eqref{Rcprob}. Suppose that the mapping $\mathcal{G}$ defined in \eqref{MG-map} is metrically subregular at any $(x,0)$ with $x\in S^*$, and that $\Theta$ is coercive or the manifold $\mathcal{M}$ is compact. Then, \eqref{penprob} is a global exact penalty of \eqref{Rcprob}. 
\end{lemma}
\begin{proof}
Let $\mathcal{H}(\tau):=\big\{x\in\mathcal{M}\ |\ {\rm dist}(g(x),K)=\tau\big\}$ for $\tau\ge 0$. Fix any $\overline{x}\in\mathcal{M}\cap g^{-1}(K)=\mathcal{H}(0)$. Since $\mathcal{H}(0)=\mathcal{G}^{-1}(0)$, the definitions of metric subregularity and calmness, together with the equivalence of norms on $\mathbb{Y}$, imply that $\mathcal{G}$ is metrically subregular at $(\overline{x},0)$ if and only if $\mathcal{H}$ is calm at $(0,\overline{x})$ for any $\overline{x}\in\mathcal{H}(0)$; see also \cite[Theorem 3H.3]{Dontchev2009}. Thus, $\mathcal{H}$ is calm at $(0,x)$ for every $x\in S^*$. Since $\Theta$ is locally Lipschitz continuous, the global exactness of \eqref{penprob} follows from \cite[Lemma 2.1 and Proposition 2.1(b)]{Liu2018}.
\end{proof}
\subsection{Stationary points}\label{sec2.2}

Let $\widetilde{\Theta}:=\Theta+\delta_{K}\circ g+\delta_{\mathcal{M}}$ be the extended-real-valued objective function of \eqref{Rcprob}. By \cite[Theorem 10.1]{RW98}, any local minimizer $x^*$ of \eqref{Rcprob} satisfies $0\in\partial\widetilde{\Theta}(x^*)$. Under Assumption \ref{ass0}, the function $\Theta$ is locally Lipschitz continuous, so that \cite[Exercise 10.10]{RW98} yields  $\partial\widetilde{\Theta}(x)\subset\partial\Theta(x)+\partial(\delta_{K}\circ g+\delta_{\mathcal{M}})(x)$ and $\partial\Theta(x)\subset\nabla\!f(x)+\partial \vartheta(x)+\partial(-h)(x)$ for every $x\in \mathcal{M}\cap g^{-1}(K)$. Moreover, if $\mathcal{G}$ in \eqref{MG-map} is metrically subregular at $(x,0)$ for $x\in \mathcal{M}\cap g^{-1}(K)$, by Ioffe and Outrata \cite[Section 3.1]{Ioffe2008}, $\partial(\delta_{K}\circ g+\delta_{\mathcal{M}})(x)\subset \nabla g(x)\mathcal{N}_{K}(g(x))+N_{x}\mathcal{M}$. Thus, if $\mathcal{G}$ is metrically subregular at $(x^*,0)$, then
\[
 0\in\partial\widetilde{\Theta}(x^*)\subset \nabla\!f(x^*)+\partial \vartheta(x^*)+\partial(-h)(x^*)+\nabla g(x^*)\mathcal{N}_{K}(g(x^*))+N_{x^*}\mathcal{M}. 
 \]
 This inclusion motivates the following notions of KKT points for problem \eqref{Rcprob}.
\begin{definition}\label{def-spoint}
\textnormal{(i)} A point $x$ is called a KKT point of \eqref{Rcprob} if $x\in\mathcal{M}\cap g^{-1}(K)$ and there exists $\xi\in \mathcal{N}_K(g(x))$ such that $0\in \mathcal{P}_{T_x\mathcal{M}}\big[\nabla\!f(x) + \partial\vartheta(x)+\partial (-h)(x) + \nabla g(x)\xi\big]$.
	
\noindent
\textnormal{(ii)} For $\epsilon\ge 0$, a point $x$ is called an $\epsilon$-approximate KKT point of \eqref{Rcprob} if $x\in\mathcal{M}$ and there exist $z\in\mathbb{X}$, $y\in K,\,\zeta\in\partial (-h)(x),\,\xi_1\in\partial\vartheta(z)$ and $\xi_2\in\mathcal{N}_K(y)$ such that
\begin{equation*}
\|x-z\|\le\epsilon,\,\|g(x)-y\|\le\epsilon\ \ {\rm and}\ \ \|\mathcal{P}_{T_x\mathcal{M}}\left[\nabla\!f(x)+\zeta+\xi_1+\nabla g(x)\xi_2\right]\|\le\epsilon.
\end{equation*}
\end{definition}	

For any $\rho>0$ and $x\in\mathcal{M}$, by \cite[Exercise 10.10 ]{RW98}, $\partial(\Theta_{\rho}\!+\!\delta_{\mathcal{M}})(x)\!\subset \partial\Theta_{\rho}(x)+N_{x}\mathcal{M}\subset \nabla\! f(x)+\partial\vartheta(x)+\partial(-h)(x)+\rho\partial({\rm dist}(g(\cdot),K))(x)+N_{x}\mathcal{M}$. Since the two inclusions may be strict, we distinguish the following two notions of stationarity for \eqref{penprob}.
\begin{definition}\label{Def-Penspoint}
For $\rho>0$, a point $x$ is called a stationary point of \eqref{penprob} if $x\in\mathcal{M}$ and $0\in\partial\Theta_{\rho}(x)+N_{x}\mathcal{M}$, and a weak stationary point if $x\in\mathcal{M}$ and $0\in\nabla\! f(x)+\partial\vartheta(x)+\rho\partial({\rm dist}(g(\cdot),K))(x)+\partial(-h)(x)+N_{x}\mathcal{M}$. 
\end{definition}

We next establish the relationship between KKT points of \(\eqref{Rcprob}\) and (weak) stationary points of \(\eqref{penprob}\).
\begin{proposition}\label{spoint-relation}
Every KKT point $\overline{x}$ of \eqref{Rcprob} is a weak stationary point of \eqref{penprob} for all large enough $\rho$. Conversely, every weak stationary point $\overline{x}$ of \eqref{penprob} satisfying $g(\overline{x})\in K$ is a KKT point of \eqref{Rcprob}. If $h\equiv 0$, both conclusions remain valid with ``weak stationary point'' replaced by ``stationary point.''
\end{proposition}
\begin{proof}
Since ${\rm dist}(\cdot,K)$ is convex and Lipschitz continuous, applying the chain rule in \cite[Theorem 10.6]{RW98} yields
\begin{equation}\label{equa21-dist}
\partial({\rm dist}(g(\cdot),K))(x)=\nabla g(x)\partial{\rm dist}(g(x),K)\quad{\rm for\ all}\ x\in\mathbb{X}.
\end{equation}
If $\overline{x}$ is a KKT point of \eqref{Rcprob}, by Definition \ref{def-spoint}(i), $\overline{x}\in\mathcal{M}\cap g^{-1}(K)$ and there exists some $\xi\in \mathcal{N}_K(g(\overline{x}))$ such that $0\in \nabla\!f(\overline{x})+\partial\vartheta(\overline{x})+ \nabla g(\overline{x})\xi+\partial(-h)(\overline{x})+N_{\overline{x}}\mathcal{M}$. Choose  $\rho_{\xi}\ge|\!\lVert \xi\lVert\!|_*$. Since $\xi\in \mathcal{N}_K(g(\overline{x}))$, Lemma \ref{subdiff-dist} yields $\rho^{-1}\xi\in \partial{\rm dist}(g(\overline{x}),K)$ for all $\rho\ge\rho_{\xi}$. Together with \eqref{equa21-dist}, $\nabla g(\overline{x})\xi\in\rho\partial({\rm dist}(g(\cdot),K))(\overline{x})$ for all $\rho\ge\rho_{\xi}$. Consequently, $0\in \nabla\! f(\overline{x})+\partial\vartheta(\overline{x})+\rho\partial({\rm dist}(g(\cdot),K))(\overline{x})+\partial(-h)(\overline{x})+N_{\overline{x}}\mathcal{M}$. This shows that $\overline{x}$ is a  weak stationary point of \eqref{penprob} for all $\rho\ge\rho_{\xi}$. Conversely, if $\overline{x}\in\mathcal{M}$ is a weak stationary point of \eqref{penprob} for $\rho>0$ satisfying $g(\overline{x})\in K$, by Lemma \ref{subdiff-dist} and \eqref{equa21-dist}, we have $0\in \nabla\!f(\overline{x})+\partial\vartheta(\overline{x})+ \nabla g(\overline{x})\mathcal{N}_K(g(\overline{x}))+\partial(-h)(\overline{x})+N_{\overline{x}}\mathcal{M}$. Then $\overline{x}$ is a KKT point of \eqref{Rcprob}. If $h\equiv0$, the subdifferential regularity of $\vartheta$ and ${\rm dist}(g(\cdot),K)$, together with the sum rule, yields $\partial\Theta_{\rho}(x)=\nabla\! f(x)+\partial\vartheta(x)+\rho\partial({\rm dist}(g(\cdot),K))(x)$ for all $x\in\mathbb{X}$. The last assertion holds.
\end{proof}


\section{An inexact proximal-linearized NEP method}\label{sec3}

Let $x^k\in\!\mathcal{M}$ be the current iterate and $\rho_k$ the associated penalty parameter. Under Assumption \ref{ass0}(i), define the first-order local models of \(f\) and \(g\) at \(x^k\) by
\[  
\ell_f(\cdot;x^k):=f(x^k)+\langle\nabla\!f(x^k),\cdot-x^k\rangle\ \ {\rm and}\ \ \ell_g(\cdot;x^k):=g(x^k)+g'(x^k)(\cdot-x^k).
\]
Choose any \(\zeta^k\in\partial(-h)(x^k)\). It follows from the convexity of \(h\) that, for any $x\in\mathbb{X}$, 
\[
 -h(x)\le -h(x^k)+\langle \zeta^k,x-x^k\rangle.
\] 
For a proximal parameter $\beta_k>0$, we define the proximal-linearized model of $\Theta_{\rho_k}$ at $x^k$ by
\begin{equation}\label{Thetahat-rhobetak}
\widehat{\Theta}_{\rho_k,\beta_k}(x;x^k):=		\ell_f(x;x^k)+\vartheta(x)+\rho_k\mathrm{dist}(\ell_g(x;x^k),K)-h(x^k)+\langle\zeta^k,x-x^k\rangle+\frac{\beta_k}{2}\|x-x^k\|^2.
\end{equation}
The proximal term makes $\widehat{\Theta}_{\rho_k,\beta_k}(\cdot;x^k)$ strongly convex, but this model need not locally majorize $\Theta_{\rho_k}$. The proposed proximal-linearized NEP method computes a search direction $v^k\in T_{x^k}\mathcal{M}$ by approximately minimizing $\widehat{\Theta}_{\rho_k,\beta_k}(x^k+v;x^k)$ or solving the following strongly convex subproblem 
\begin{equation}\label{subprob}
\min_{v\in T_{\!x^k}\mathcal{M}}\Big\{\ell_k(v)+\vartheta(x^k\!+v)+\rho_k{\rm dist}(\ell_g(x^k\!+v;x^k),K)+\frac{\beta_k}{2}\|v\|^2\Big\},
\end{equation}
where $\ell_k(v):=\langle\nabla\!f(x^k)+\zeta^k, v \rangle+f(x^k)-h(x^k)$ for $v\in\mathbb{X}$. To formulate computable inexactness criteria, we introduce the equivalent linearly constrained convex subproblem
\begin{equation}\label{Esubprob}
 \begin{aligned}
 &\min_{v\in\mathbb{X},z\in\mathbb{Z}}\,\ell_k(v)+\psi_{\rho_k}(z)+\delta_{T_{\!x^k}\mathcal{M}}(v)+\frac{\beta_k}{2}\|v\|^2\\
 &\quad{\rm s.t.}\ \ G(x^k,v)-z=0\quad\ {\rm with}\ \ G(x,v):=\begin{pmatrix}
		x+v\\ \ell_g(x+v;x)
 \end{pmatrix},
   \end{aligned}
\end{equation}	
where $\psi_{\rho_k}(z):=\vartheta(z_1)+\rho_k{\rm dist}(z_2,K)$ for $z=(z_1,z_2)\in\mathbb{Z}$. A pair $(v^{k,*},z^{k,*})\in T_{\!x^k}\mathcal{M}\times\mathbb{Z}$ is an optimal solution of \eqref{Esubprob} if and only if there exists $\xi^{k,*}\in\partial\psi_{\rho_k}(z^{k,*})$ such that 
\[
G(x^k,v^{k,*})-z^{k,*}=0\ \ {\rm and}\ \ \mathcal{P}_{T_{\!x^k}\mathcal{M}}\big[\nabla \ell_k(v^{k,*})+\beta_kv^{k,*}+\nabla_{\!v}G(x^k,v^{k,*})\xi^{k,*}\big]=0.
\]

Let $\overline{v}^k$ denote the unique optimal solution of \eqref{subprob}. We consider two alternative sets of inexactness criteria, motivated by whether $\overline v^k$ vanishes. If $\overline{v}^k = 0$, the optimality condition of \eqref{subprob} implies that $x^k$ is a critical point of $\Theta_{\rho_k}$. This motivates an inexactness criterion that controls both the KKT residuals of \eqref{Esubprob} and the size of the approximate direction. Specifically, we require a triple $(v^k,z^k,\xi^k)\in T_{x^k}\mathcal{M}\times\mathbb{Z}\times\partial\psi_{\rho_k}(z^k)$ to satisfy
\begin{equation}\label{inexact-cond1}
\max\big\{\|G(x^k,v^k)-z^k\|,\big\|\mathcal{P}_{T_{\!x^k}\mathcal{M}}\big[\nabla \ell_k(v^k)+\beta_kv^k+\nabla_{\!v}G(x^k,v^k)\xi^k\big]\big\|, \|\beta_kv^k\|\big\}\le \varepsilon_k,
\end{equation}
where $\varepsilon_k > 0$ is a prescribed tolerance. The last term controls the size of $v^k$. If $\overline{v}^k\ne 0$, the strong convexity of the subproblem objective ensures that its minimum is strictly smaller than its value at $v=0$. Accordingly, we require the approximate direction to yield a strict decrease in the proximal-linearized model: 
\begin{equation}\label{model-descent}
\widehat{\Theta}_{\rho_k,\beta_k}(x^k+v^k;x^k)< \widehat{\Theta}_{\rho_k,\beta_k}(x^k;x^k)=\Theta_{\rho_k}(x^k).
\end{equation}  
In addition, we impose the following relative KKT residual bounds:
\begin{equation}\label{inexact-cond2}
\|G(x^k,v^k)-z^k\|\le a_k\|v^k\|^2\ \ {\rm and}\ \ \big\|\mathcal{P}_{T_{\!x^k}\mathcal{M}}\big[\nabla \ell_k(v^k)+\beta_kv^k+\nabla_{\!v}G(x^k,v^k)\xi^k\big]\big\|\le \!b_k\|v^k\|,
\end{equation}
where $a_k,b_k>0$ are prescribed parameters. In particular, the equality-constraint residual is required to be of order $O(\|v^k\|^2)$ rather than merely $O(\|v^k\|)$. This second-order control is essential for Lemma \ref{Lemma1-Thetak} and the subsequent iteration complexity analysis. Thus, at each iteration, we compute a triple
$(v^k,z^k,\xi^k)\in T_{x^k}\mathcal{M}\times\mathbb{Z}
\times\partial\psi_{\rho_k}(z^k)$
satisfying either the inexactness condition \eqref{inexact-cond1} or the alternative conditions
\eqref{model-descent}-\eqref{inexact-cond2}.
The resulting trial point is $R_{x^k}(v^k)$, where $R_{x^k}$ is a retraction onto $\mathcal{M}$.

The proximal parameter $\beta_k$ determines the curvature of the proximal linearization \(\widehat{\Theta}_{\rho_k,\beta_k}(\cdot; x^k)\), and consequently affecting its local approximation accuracy and the quality of the trial point $R_{x^k}(v^k)$. Following a strategy similar to that in \cite[Algorithm 2.2]{Cartis2011}, we adaptively update \(\beta_k\). Specifically, define the actual and predicted decreases as 
\begin{equation}\label{ared-pred}
{\rm ared}_k:=\Theta_{\rho_k}(x^k)-\Theta_{\rho_k}(R_{x^k}(v^k))\ \ {\rm and}\ \ {\rm pred}_k:=\Theta_{\rho_k}(x^k)-\widehat{\Theta}_{\rho_k,\beta_k}(x^k+v^k;x^k).
\end{equation}
For a nonzero trial direction, we reject the trial point if 
${\rm pred}_k<\overline{\gamma}\|v^k\|^2$
for a prescribed $\overline{\gamma}>0$, or if the ratio
${\rm ared}_k/{\rm pred}_k$ falls below a prescribed threshold. In either case, $\beta_k$ is increased. Otherwise, the trial point is accepted, and $\beta_k$ is either kept unchanged or reduced. 
The penalty parameter $\rho_k$ is updated based on the surrogate feasibility violation ${\rm dist}(z_2^k,K)$. Under the inexactness criterion \eqref{inexact-cond1}, we increase $\rho_k$ whenever $\text{dist}(z_2^k, K)>\varepsilon_k$. Under \eqref{model-descent}-\eqref{inexact-cond2}, the bound $\Vert{}g(x^k)+g'(x^k)v^k-z_2^k\Vert{} \le a_k\|v^k\|^2$ motivates increasing $\rho_{k}$ if $\text{dist}(z_2^k, K)>\beta_k \Vert{}v^k\Vert{}$. These considerations lead to the following algorithm.
\begin{algorithm}[h]
 \caption{\label{PenAl}{\bf (Inexact nonsmooth exact penalty method)}}
 \begin{algorithmic}[1]
 \State\label{alg:step1} Input: a retraction $R$, parameters $\rho_0>0,0<\beta_{\rm min}<\beta_{0},a_{\max},b_{\max}\!>0,\tau>1$,  \hspace*{1.3cm} $\varepsilon_0\in(0,1),\varsigma\in(0,1)$, $0<\eta_1<\eta_2<1,\,\overline{\gamma}>0,\,0<\varsigma_2<1<\varsigma_1$, and $x^0\in \mathcal{M}$. 
		
 \For{$k=0,1,2,\ldots$}\label{alg:step2} 
		
 \State Choose $\zeta^k\in\partial (-h)(x^k)$. 
		
 \State\label{alg:step4} Select $a_k\in (0,a_{\max}]$ and $b_k\in(0,b_{\max}]$. Solve \eqref{Esubprob} inexactly to obtain \hspace*{0.43cm} $(v^k,z^k,\xi^k)\in T_{x^k}\mathcal{M}\times \mathbb{Z}\times\partial\psi_{\rho_k}(z^k)$ satisfying \eqref{inexact-cond1} or both \eqref{model-descent} and \eqref{inexact-cond2}. If \hspace*{0.43cm} \eqref{inexact-cond1} holds, go to Step \ref{alg:step5}; otherwise, go to Step \ref{alg:step6}. 
		
 \State\label{alg:step5} Set $x^{k+1}:=x^k$, $\beta_{k+1}:=\beta_k$, $\varepsilon_{k+1}:=\varsigma\varepsilon_k$, and  $\rho_{k+1}:=\rho_k$ if ${\rm dist}(z_2^k,K)\le\varepsilon_k$; \hspace*{0.45cm} otherwise, $\rho_{k+1}:=\tau\rho_k$. Return to Step \ref{alg:step2}.
		
 \State\label{alg:step6} Compute $r_k:={{\rm ared}_k}/{{\rm pred}_k}$. If $r_k<\eta_1$ or ${\rm pred}_k<\overline{\gamma}\|v^k\|^2$, set $x^{k+1}:=x^k$, \hspace*{0.45cm} $\beta_{k+1}:=\varsigma_1\beta_k,\,\varepsilon_{k+1}:=\varepsilon_k$, $\rho_{k+1}:=\rho_k$, and return to Step \ref{alg:step2}. Otherwise, set \hspace*{0.47cm} $\varepsilon_{k+1}:=\varepsilon_k,\,x^{k+1}:=R_{x^k}(v^k)$, and update $\beta_{k+1}$ as 
 \begin{equation}\label{update-betak}
 \beta_{k+1}:=\left\{\begin{array}{cl}
				\beta_k &\text{if}\ r_k<\eta_2,\qquad {\# successful}\\
				\max\{\varsigma_2\beta_k,\beta_{\min}\}&\text{ otherwise}.\qquad{\# very\ successful}
 \end{array}\right.
\end{equation}
		
\State\label{alg:step7} Update the penalty parameter $\rho_{k+1}$ via the following rule
		\begin{equation}\label{update-rho}
			\rho_{k+1}:=\left\{\begin{array}{cl}
				\tau\rho_k & \text{if}\ \beta_k\|v^k\|
				<\min\{\rho_k^{-1},{\rm dist}(z_2^k,K)\},\\
				\rho_k & \text{otherwise}.
			\end{array}\right.
		\end{equation}	
		\EndFor
	\end{algorithmic}
\end{algorithm}
\begin{remark}\label{remark-alg}
{\bf(a)} As detailed in Section \ref{sec6.1.1}, semi-proximal ADMM and its accelerated variant, when applied to subproblem \eqref{Esubprob}, generate in finitely many inner iterations a triple $(v^k,z^k,\xi^k)\in T_{x^k}\mathcal{M}\times\mathbb{Z}\times\partial\psi_{\rho_k}(z^k)$ satisfying either \eqref{inexact-cond1} or \eqref{model-descent}-\eqref{inexact-cond2}. Likewise, Lemma \ref{dual-inexact} establishes the same finite-termination property for dual solvers, including accelerated proximal gradient methods. Thus, the proposed inexact stopping rule is finitely attainable by both primal-dual and dual subproblem solvers. 

The alternative criterion \eqref{inexact-cond1} is particularly important when the exact search direction $\overline v^k=0$, which implies only weak stationarity of the penalized problem \eqref{penprob}, but not necessarily the KKT conditions of \eqref{Rcprob}. In this case, the strict-decrease condition \eqref{model-descent} is impossible. Even if strict descent is relaxed to nonstrict descent, strong convexity forces $v^k=0$, so the relative residual bounds in \eqref{inexact-cond2} require exact satisfaction of the subproblem KKT conditions, which cannot generally be guaranteed in finitely many inner iterations. Thus, \eqref{inexact-cond1} is crucial for ensuring finite termination and establishing the overall oracle complexity guarantees, though it is invoked only occasionally in our numerical experiments (see Figure \ref{ex_inc_fig}(a)).
 \begin{figure}[htbp]
 \centering
 \includegraphics[width=1.0\textwidth]{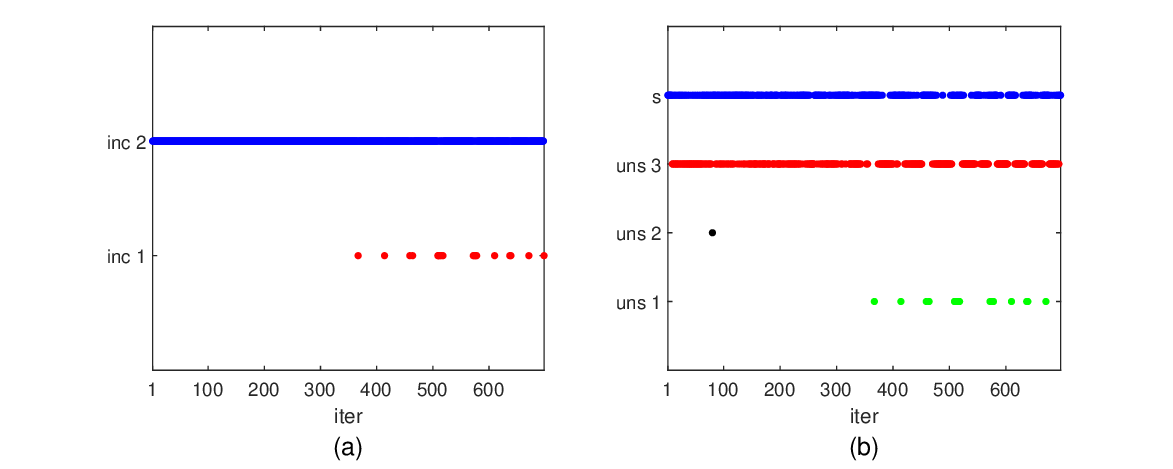}
 \vspace{-1em}
 \caption{\small (a) Satisfaction of the inexactness criteria and (b) types of outer iterations throughout the algorithm. In (a), ``inc 1'' and ``inc 2'' indicate that conditions \eqref{inexact-cond1} and \eqref{model-descent}-\eqref{inexact-cond2}, respectively, are satisfied. In (b), ``uns~1'', ``uns~2'' and ``uns~3'' denote unsuccessful iterations associated with Step \ref{alg:step5}, the condition ${\rm pred}_k < \overline{\gamma}\|v^k\|^2$ in Step \ref{alg:step6}, and the condition $r_k < \eta_1$ in Step \ref{alg:step6}, respectively; “s” denotes a successful or very successful iteration.}
 \label{ex_inc_fig}
\end{figure}
	
 \noindent
 {\bf (b)} Whenever Step \ref{alg:step6} is executed, the triple $(v^k,z^k,\xi^k)\in T_{x^k}\mathcal{M}\times \mathbb{Z}\times\partial\psi_{\rho_k}(z^k)$ satisfies \eqref{model-descent}-\eqref{inexact-cond2}. In particular, \eqref{model-descent} guarantees that ${\rm pred}_k > 0$, so the ratio $r_k$ is well-defined. If $r_k < \eta_1$ or ${\rm pred}_k<\overline{\gamma}\Vert{}v^k\Vert{}^2$, the trial point is rejected, $x^{k+1}=x^k$, and $\beta_{k+1}$ is increased to strengthen the proximal regularization. Otherwise, the trial point is accepted, and $\beta_{k+1}$ is updated by \eqref{update-betak}. In this case, since $r_k\ge\eta_1$ and ${\rm pred}_k\ge\overline{\gamma}\|v^k\|^2$, 
 \begin{equation}\label{S-descent}
 \Theta_{\rho_k}(x^k)-\Theta_{\rho_k}(R_{x^k}(v^k))
 \ge\overline{\gamma}\eta_1\|v^k\|^2.
 \end{equation} 
 The initialization and update rules for $\beta_k$ in Algorithm~\ref{PenAl} ensure that $\beta_k\ge\beta_{\min}$ for all $k\in\mathbb{N}$. 
	
 \noindent
 {\bf(c)} Let $\mathcal{S}:=\{k\in\mathbb{N} \mid \text{the trial point is accepted at iteration}\ k\}$ denote the set of successful iteration indices, sorted in ascending order. By the iterative mechanism of Algorithm \ref{PenAl}, for each $k\in\mathcal{S}$, the triple $(v^k,z^k,\xi^k)\in T_{x^k}\mathcal{M}\times \mathbb{Z}\times\partial\psi_{\rho_k}(z^k)$ satisfies \eqref{model-descent}-\eqref{inexact-cond2}, so that the sufficient-decrease condition \eqref{S-descent} holds. For any $k\notin\mathcal{S}$, the update rule yields $x^{k+1}=x^k$. That is, the unsuccessful iterations do not generate new iterates; instead, they simply retain the iterate obtained at the latest successful iteration. Hence, the cluster point set of the entire sequence $\{x^k\}_{k\in\mathbb{N}}$ is identical to that of $\{x^k\}_{k\in\mathcal{S}}$ when $\mathcal{S}$ is infinite. Figure~\ref{ex_inc_fig}(b) illustrates the successful and unsuccessful iterations, showing that most unsuccessful iterations are caused by $r_k<\eta_1$ in Step~\ref{alg:step6}. 
%
\end{remark}
\begin{remark}\label{remark1-Alg}
 When $\mathcal{M}=\mathbb{X}$, the tangent-space constraint disappears and $\delta_{T_{x^k}\mathcal{M}}(\cdot)$ vanishes. Accordingly, the reformulation \eqref{Esubprob} reduces to a splitting involving only the nonsmooth term $\vartheta$ and the linearized set-inclusion constraint $\ell_g(x^k+v;x^k)\in K$. The inexactness conditions \eqref{inexact-cond1} and \eqref{inexact-cond2} then simplify to 
 \begin{equation*}
 \begin{aligned}
 \max\big\{\|(x^k+v^k-s^k;\ell_g(x^k+v^k;x^k)-z^k)\|,\|x^k+v^k-s^k\|, \|\beta_kv^k\|\big\}\le \varepsilon_k,\\
 \big\|\big(x^k+v^k-s^k;\ell_g(x^k+v^k;x^k)-z^k\big)\big\|\le a_k\|v^k\|^2\ \ {\rm and}\ \ \big\|x^k+v^k-s^k\big\|\le b_k\|v^k\|, 
 \end{aligned}
 \end{equation*}
respectively, where $s^k:=\mathcal{P}_{\vartheta}\big(x^k+v^k-\!\nabla\ell_k(v^k)-\beta_kv^k-\nabla g(x^k)\xi^k\big)$ with $\xi^k\in\rho_k\partial{\rm dist}(z_2^k,K)$. 
\end{remark}

By the equivalence of norms on \(\mathbb Y\) and the Lipschitz continuity of \(\operatorname{dist}(\cdot,K)\) with modulus $1$ with respect to $|\!\lVert\cdot\rVert\!|$, there exist constants \(c_u\ge c_l>0\) such that  
\begin{equation}\label{dist-Lip}
c_{l}\|y\|\le|\!\lVert y\rVert\!|\le c_{u}\|y\|
\ \ {\rm and}\ \ |{\rm dist}(y,K)-{\rm dist}(y',K)|\le c_{u}\|y-y'\|\quad\forall y,y'\in\mathbb{Y}. 
\end{equation}
Combining the second inequality in \eqref{dist-Lip} with the definition of $\psi_{\rho_k}$ yields the following estimate.
\begin{lemma}\label{lemma-psirhok}
 Let $\gamma$ be a Lipschitz constant of $\vartheta$ on $\Omega\subset\mathbb{X}$ with respect to $\|\cdot\|$. Then, for each $k\in\mathbb{N}$, 
 \begin{equation*}
 |\psi_{\rho_k}(z^1)-\psi_{\rho_k}(z^2)|\le \big(\gamma^2\!+\!c_{u}^2\rho_k^2\big)^{1/2}\|z^1-z^2\|\quad\forall z^1,z^2\in \Omega\times\mathbb{Y}.
 \end{equation*}
\end{lemma}

The estimate in Lemma \ref{lemma-psirhok} will be often used in the subsequent analysis. We next establish two estimates for the analysis in Section \ref{sec4}. Lemma \ref{Lemma1-Thetak} provides a lower bound on $\widehat{\Theta}_{\rho_k,\beta_k}(x^k\!+v;x^k)-\widehat{\Theta}_{\rho_k,\beta_k}(x^k\!+v^k;x^k)$ for any $v\in T_{x^k}\mathcal{M}$, while Lemma \ref{Lemma2-Thetak} gives an upper bound for $\Theta_{\rho_k}(R_{x^k}(v^k))-\widehat{\Theta}_{\rho_k,\beta_k}(x^k+v^k;x^k)$.
\begin{lemma}\label{Lemma1-Thetak}
 Fix $\overline{\tau}>0$. For each $k\in\mathbb{N}$, let $\widetilde L_{\vartheta,k}$ denote the Lipschitz modulus of $\vartheta$ on $\mathcal{U}_{k,\overline{\tau}}$, where $\mathcal{U}_{k,\overline{\tau}}:=\{x\in\mathbb{X}\mid {\rm dist}(x,[x^k+v^k,z_1^k])<\overline{\tau}\}$. For any $k\in\mathbb{N}$, if $(v^k,z^k,\xi^k)\!\in T_{x^k}\mathcal{M}\times\mathbb{Z}\times\partial\psi_{\rho_k}(z^k)$ satisfies the conditions \eqref{model-descent}-\eqref{inexact-cond2}, then, for all $v\in T_{x^k}\mathcal{M}$, 
\begin{equation*}
\resizebox{\linewidth}{!}{$
 \widehat{\Theta}_{\rho_k,\beta_k}(x^k\!+\!v;x^k)-\widehat{\Theta}_{\rho_k,\beta_k}(x^k\!+\!v^k;x^k)\ge\frac{\beta_k}{2}\|v\!-\!v^k\|^2-b_k\|v^k\|\|v-v^k\|-2a_k\sqrt{\widetilde{L}_{\vartheta,k}^2\!+\!c_{u}^2\rho_{k}^2}\|v^k\|^2.
$}
\end{equation*}
\end{lemma}
\begin{proof}
Fix any $\epsilon>0$ and $k\in\mathbb{N}$. Suppose that $(v^k,z^k,\xi^k)\!\in T_{x^k}\mathcal{M}\times\mathbb{Z}\times\partial\psi_{\rho_k}(z^k)$ satisfies \eqref{model-descent}-\eqref{inexact-cond2}. Since $\vartheta$ is Lipschitz continuous on $\mathcal{U}_{k,\overline{\tau}}$ with Lipschitz constant $\widetilde{L}_{\vartheta,k}\!+\epsilon$, Lemma \ref{lemma-psirhok} and \cite[Theorem 9.13]{RW98} yield $\|\xi^k\|\le [(\widetilde{L}_{\vartheta,k}\!+\epsilon)^2+c_{u}^2\rho_{k}^2]^{1/2}$. By the first inequality in \eqref{inexact-cond2}, we have 
 \begin{equation}\label{temp-ineq31}
 \big|\psi_{\rho_k}(z^k)\!-\psi_{\rho_k}(G(x^k,v^k))+\langle\xi^k,G(x^k,v^k)\!-z^k\rangle\big|\!\le 2a_k[(\widetilde{L}_{\vartheta,k}\!+\!\epsilon)^2\!+\!c_{u}^2\rho_{k}^2]^{1/2}\|v^k\|^2.
 \end{equation}
 Fix any $v\in T_{x^k}\mathcal{M}$. Since $\ell_k(\cdot)+\frac{\beta_k}{2}\|\cdot\|^2+\delta_{T_{\!x^k}\mathcal{M}}(\cdot)$ is  $\beta_k$-strongly convex and $v^k\in T_{\!x^k}\mathcal{M}$, it follows that
 \begin{equation*}
 \ell_k(v)+\frac{\beta_k}{2}\|v\|^2+\delta_{T_{\!x^k}\mathcal{M}}(v) \ge \ell_k(v^k)+\frac{\beta_k}{2}\|v^k\|^2+\langle\nabla \ell_k(v^k)+\beta_kv^k+s,v-v^k\rangle+\frac{\beta_k}{2}\|v-v^k\|^2
 \end{equation*}
 for any $s\in \mathcal{N}_{T_{x^k}\mathcal{M}}(v^k)$. 
 Moreover, by convexity and $\xi^k\in\partial\psi_{\rho_k}(z^k)$, we have
 \begin{equation*}
  \begin{aligned}
  \psi_{\rho_k}(G(x^k,v))&\ge \psi_{\rho_k}(z^k)+\langle \xi^k,G(x^k,v^k)-z^k\rangle+\langle \xi^k,G(x^k,v)-G(x^k,v^k)\rangle\nonumber\\
  &=\psi_{\rho_k}(z^k)+\langle\xi^k,G(x^k,v^k)-z^k\rangle+\langle \nabla_{\!v} G(x^k,v^k)\xi^k,v-v^k\rangle,
	 \end{aligned}
 \end{equation*}
 where the equality is due to the affinity of $G(x^k,\cdot)$. Adding the above two inequalities together and using the relation $\widehat{\Theta}_{\rho_k,\beta_k}(x^k\!+\!\cdot;x^k)=\ell_k(\cdot)+\frac{\beta_k}{2}\|\cdot\|^2+\psi_{\rho_k}(G(x^k,\cdot))$ gives 
 \begin{equation*}
  \begin{aligned}
  \widehat{\!\Theta}_{\rho_k,\beta_k}(x^k\!+\!v;x^k)&\ge \widehat{\Theta}_{\rho_k,\beta_k}(x^k\!+\!v^k;x^k)+\psi_{\rho_k}(z^k)-\psi_{\rho_k}(G(x^k,v^k))+\langle\xi^k,G(x^k,v^k)-z^k\rangle\nonumber \\
  &\quad +\frac{\beta_k}{2}\|v-v^k\|^2+\langle\nabla \ell_k(v^k)+\beta_kv^k+\nabla_{\!v}G(x^k,v^k)\xi^k\!+s,v-v^k\rangle\\
  &\overset{\eqref{temp-ineq31}}{\ge} \widehat{\Theta}_{\rho_k,\beta_k}(x^k\!+\!v^k;x^k)+\langle\nabla \ell_k(v^k)+\beta_kv^k+\nabla_{\!v} G(x^k,v^k)\xi^k,v-v^k\rangle\\
  &\quad +\frac{\beta_k}{2}\|v-v^k\|^2-2a_k[(\widetilde{L}_{\vartheta,k}\!+\!\epsilon)^2\!+\!c_{u}^2\rho_{k}^2]^{1/2}\,\|v^k\|^2,
 \end{aligned}
 \end{equation*}
 where the second inequality also follows from $v-v^k\in T_{\!x^k}\mathcal{M}$ and $s\in\mathcal{N}_{T_{x^k}\mathcal{M}}(v^k)=(T_{x^k}\mathcal{M})^{\perp}= N_{\!x^k}\mathcal{M}$. Letting $\epsilon\downarrow0$ in the above inequality and applying the Cauchy-Schwarz inequality to the inner product on the right-hand side, together with the second inequality in \eqref{inexact-cond2}, yields the desired conclusion. 
\end{proof}
\begin{lemma}\label{Lemma2-Thetak}
 Fix any $k\in\mathbb{N}$. Let $\widehat{L}_{f,k}, \widehat{L}_{\vartheta,k},\widehat{L}_{h,k}$ and $\widehat{L}_{g,k}$ be the Lipschitz moduli of $f,\vartheta,h$ and $g$, respectively, on the line segment $[x^k\!+\!v^k,R_{x^k}(v^k)]$, and let $\widehat{L}_{\nabla\!f,k}$ and $\widehat{L}_{\nabla g,k}$ be those of $\nabla\!f$ and $\nabla g$, respectively, on the line segment $[x^k,x^k\!+\!v^k]$. Then, 
 \begin{equation*}
  \begin{aligned}
  \Theta_{\rho_k}(R_{x^k}(v^k))&\le\widehat{\Theta}_{\rho_k,\beta_k}(x^k\!+v^k;x^k)+\frac{1}{2}\big(\widehat{L}_{\nabla\!f,k}+c_{u}\rho_k\widehat{L}_{\nabla g,k}-\beta_k\big)\|v^k\|^2\\		&\quad+\big(\widehat{L}_{f,k}+c_{u}\rho_k\widehat{L}_{g,k}+\widehat{L}_{\vartheta,k}+\widehat{L}_{h,k}\big)\|R_{x^k}(v^k)-(x^k\!+\!v^k)\|.
\end{aligned}
\end{equation*}
\end{lemma}
\begin{proof}
 Fix any $\epsilon>0$. By the definitions of $\widehat{L}_{f,k}$ and $\widehat{L}_{\nabla\!f,k}$ and the descent lemma, we have
 \begin{equation*}
\resizebox{\linewidth}{!}{$
 \begin{aligned}
 f(R_{x^k}(v^k))&\le f(x^k+v^k)+(\widehat{L}_{ f,k}+\epsilon)\|R_{x^k}(v^k)-(x^k\!+\!v^k)\|\\
 &\le f(x^k)+\langle\nabla\! f(x^k),v^k\rangle+\frac{1}{2}(\widehat{L}_{\nabla\!f,k}+\epsilon)\|v^k\|^2+(\widehat{L}_{ f,k}+\epsilon)\|R_{x^k}(v^k)-(x^k\!+\!v^k)\|.
 \end{aligned}
$}
\end{equation*}
 Similarly, the definitions of $\widehat{L}_{\vartheta,k}$ and $\widehat{L}_{h,k}$, together with the inclusion $\zeta^k\in\partial(-h)(x^k)$, yield
 \begin{equation*}
 \begin{aligned}
 \vartheta(R_{x^k}(v^k))&\le \vartheta(x^k+v^k)+(\widehat{L}_{\vartheta,k}+\epsilon)\|R_{x^k}(v^k)-(x^k+v^k)\|,\\
 -h(R_{x^k}(v^k))&\le -h(x^k+v^k) +(\widehat{L}_{h,k}+\epsilon)\|R_{x^k}(v^k)-(x^k+v^k)\|\nonumber\\
 &\le -h(x^k)+\langle \zeta^k,v^k\rangle+(\widehat{L}_{h,k}+\epsilon)\|R_{x^k}(v^k)-(x^k+v^k)\|.
 \end{aligned}
 \end{equation*} 
In addition, from the definitions of $\widehat{L}_{g,k}$ and $\widehat{L}_{\nabla g,k}$ and the relation \eqref{dist-Lip}, it follows that
\begin{equation*}
\begin{aligned}
 {\rm dist}(g(R_{x^k}(v^k)),K)&\le {\rm dist}(\ell_g(x^k+v^k;x^k),K)+|\!\lVert g(R_{x^k}(v^k))-\ell_g(x^k\!+\!v^k;x^k)\lVert\!|\nonumber\\
 &\le {\rm dist}(\ell_g(x^k+v^k;x^k),K)+c_{u}\| g(R_{x^k}(v^k))\!-\!g(x^k+v^k)\|\\
 &\quad + c_{u}\| g(x^k+v^k)-\ell_g(x^k+v^k;x^k)\|\\
 &\le {\rm dist}(\ell_g(x^k\!+v^k;x^k),K)+c_{u}(\widehat{L}_{g,k}\!+\epsilon)\|R_{x^k}(v^k)-x^k-v^k\|\\
 &\quad+(c_{u}/2)(\widehat{L}_{\nabla g,k}\!+\epsilon)\|v^k\|^2.
 \end{aligned}
 \end{equation*}
 Combining the above inequalities with the definitions of $\Theta_{\rho_k}$ and $\widehat{\Theta}_{\rho_k,\beta_k}(\cdot;x^k)$ gives
 \begin{equation*}
 \begin{aligned}
 \Theta_{\rho_k}(R_{x^k}(v^k))&\le \widehat{\Theta}_{\rho_k,\beta_k}(x^k\!+\!v^k;x^k)+\frac{1}{2}\big[\widehat{L}_{\nabla\!f,k}+c_{u}\rho_k\widehat{L}_{\nabla g,k}+(1+c_{u}\rho_k)\epsilon-\beta_k\big]\|v^k\|^2\\		
 &\ \ +\big[\widehat{L}_{f,k}\!+c_{u}\rho_k\widehat{L}_{g,k}+\!\widehat{L}_{\vartheta,k}+\!\widehat{L}_{h,k}+(3+c_{u}\rho_k)\epsilon\big]\|R_{x^k}(v^k)\!-\!(x^k\!+\!v^k)\|.
\end{aligned}
\end{equation*}
Letting $\epsilon\downarrow 0$ in this inequality, we obtain the desired conclusion. 
\end{proof}

In the following two sections, we establish iteration complexity bounds for Algorithm \ref{PenAl} and analyze the convergence of $\{x^k\}_{k\in\mathbb{N}}$ under the following boundedness assumption.
\begin{assumption}\label{ass1}
{\bf(i)} The sequence $\{x^k\}_{k\in\mathbb{N}}$ is bounded;
{\bf(ii)} $\{\rho_k\}_{k\in\mathbb{N}}$ is bounded.
\end{assumption}
\begin{remark}\label{remark-ass1}
{\bf(a)} Assumption \ref{ass1}(i) is relatively mild and holds automatically if \(\mathcal{M}\) is compact. This assumption, by the closedness of $\mathcal{M}$, implies that $\varLambda_0:={\rm cl}\{x^k\mid k\in\mathbb{N}\}$ is a compact subset of $\mathcal{M}$. 
	
\noindent
{\bf(b)} Since the penalty parameter sequence is nondecreasing in Algorithm~\ref{PenAl}, Assumption \ref{ass1}(ii) implies that there exists $\overline{k}\in\mathbb{N}$ such that $\rho_{k}=\rho_{\overline{k}}:=\overline{\rho}$ for all $k\ge\overline{k}$. This will be used repeatedly later. 
\end{remark}

The following proposition provides a sufficient condition for the boundedness of the sequence $\{\rho_k\}_{k\in\mathbb{N}}$.
\begin{proposition}\label{rho-bounded}
Suppose that Assumption \ref{ass1}(i) holds, and that
the extended CQ in Definition \ref{def-ECQ} holds at
every cluster point of $\{x^k\}_{k\in\mathbb{N}}$.
Then the sequence $\{\rho_k\}_{k\in\mathbb{N}}$ is bounded.
\end{proposition}
\begin{proof}
 Suppose, to the contrary, that $\{\rho_k\}_{k\in\mathbb{N}}$ is unbounded. By the update rules for $\rho_k$ in Steps \ref{alg:step5} and \ref{alg:step7}, we consider the following two cases.
	
 \noindent
 {\bf Case 1:} there exists an infinite set $\mathcal{K}\subset\mathbb{N}$ such that $\{\rho_k\}_{k\in\mathcal{K}}$ is generated by Step \ref{alg:step5} and $\lim_{\mathcal{K}\ni k\to\infty}\rho_k=\infty$. Now, for each $k\in\!\mathcal{K}$, the triple $(v^k,z^k,\xi^k)\in T_{\!x^k}\mathcal{M}\times\mathbb{Z}\times\partial\psi_{\rho_k}(z^k)$ and satisfies \eqref{inexact-cond1}, and ${\rm dist}(z_2^k,K)>\varepsilon_k$. Further, $\lim_{\mathcal{K}\ni k\to\infty}\varepsilon_k=0$, which together with \eqref{inexact-cond1} yields 
 \begin{equation}\label{temp-limit}
 \lim_{\mathcal{K}\ni k\to\infty}\beta_kv^k=0\ \ {\rm and}\ \ \lim_{\mathcal{K}\ni k\to\infty}G(x^k,v^k)-z^k=0.
 \end{equation}
 Since $\beta_k\ge\beta_{\min}>0$ for all $k\in\mathbb{N}$, the first limit in \eqref{temp-limit} implies that $\lim_{\mathcal{K}\ni k\to\infty}v^k=0$. By Assumption \ref{ass1}(i) and the closedness of $\mathcal{M}$, passing to a subsequence if necessary, we may assume that
 $\lim_{\mathcal{K}\ni k\to\infty} x^k=x^*\in\mathcal{M}$. Together with the second limit in \eqref{temp-limit} and the definition of $G$, this yields $\lim_{\mathcal{K}\ni k\to\infty}z_1^k=x^{*}$ and $\lim_{\mathcal{K}\ni k\to\infty}z_2^k=g(x^{*})$. For each $k\in\mathcal{K}$, since $\xi^k=(\xi_1^k,\xi_2^k)\in\partial\psi_{\rho_k}(z^k)$, we have $\xi_1^k\in\partial\vartheta(z_1^k)$ and $\rho_{k}^{-1}\xi_2^k\in\partial{\rm dist}(z_2^k,K)$. By the local Lipschitz continuity of $\vartheta$ and \cite[Theorem 9.13]{RW98}, the sequence $\{\xi_1^k\}_{k\in\mathcal K}$ is bounded. Moreover, since ${\rm dist}(z_2^k,K)>\varepsilon_k\ge0$, Lemma \ref{subdiff-dist} gives $|\!\lVert\rho_{k}^{-1}\xi_2^k|\!\lVert_*=1$ for all $k\in\mathcal{K}$. Passing to a further subsequence if necessary, we may assume that $\lim_{\mathcal{K}\ni k\to\infty}\rho_k^{-1}\xi_2^k=\xi_2^*$. By the outer semicontinuity of $\partial{\rm dist}(\cdot,K)$ and the convergence $z_2^k\to g(x^*)$, we obtain
$\xi_2^*\in\partial {\rm dist}(g(x^*),K)$ and $|\!\lVert\xi_2^*\lVert\!|_*=1$. Recall that $\nabla\ell_k(v^k)=\nabla\!f(x^k)+\zeta^k$ with $\zeta^k\in\partial(-h)(x^k)$ for each $k\in\mathbb{N}$. By Assumption \ref{ass1}(i), the local Lipschitz continuity of $h$, and \cite[Theorem 9.13]{RW98}, the sequence $\{\nabla\ell_k(v^k)\}_{k\in\mathcal{K}}$ is bounded. Now dividing the first inequality in \eqref{inexact-cond1} by $\rho_k$ gives 
\begin{equation*}
 \|\mathcal{P}_{T_{\!x^k}\mathcal{M}}[\rho_k^{-1}(\nabla \ell_k(v^k)+\beta_kv^k+\xi_1^k+\nabla g(x^k)\xi_2^k)]\|\le\rho_k^{-1}\varepsilon_{k}.   
\end{equation*} 
Recall that $\mathcal{M}$ is a closed $\mathcal{C}^2$-embedded submanifold of $\mathbb{X}$. The local representation of $\mathcal{M}$ implies that the projection mapping $\mathcal{P}_{T_x\mathcal{M}}$ is continuous in $x\in\mathcal{M}$. Passing to the limit $\mathcal{K}\ni k\to\infty$ in the above inequality, we obtain $\mathcal{P}_{T_{\!x^{*}}\mathcal{M}}(\nabla g(x^{*})\xi_2^*)=0$. Equivalently, $0\in\nabla g(x^{*})\xi_2^*+N_{x^*}\mathcal{M}$. Together with $\xi_2^*\in\partial{\rm dist}(g(x^*),K)$ and $|\!\lVert\xi_2^*\lVert\!|_*=1$, this contradicts the extended CQ at $x^*$. 
	
 \noindent
 {\bf Case 2:} there exists an infinite set $\mathcal{K}\subset\mathbb{N}$ such that $\{\rho_k\}_{k\in\mathcal{K}}$ is generated by Step \ref{alg:step7} and
$\lim_{\mathcal K\ni k\to\infty}\rho_k=\infty$. For each $k\in\mathcal{K}$, the triple $(v^k,z^k,\xi^k)\in T_{\!x^k}\mathcal{M}\times\mathbb{Z}\times\partial\psi_{\rho_k}(z^k)$ satisfies \eqref{model-descent}-\eqref{inexact-cond2}, and  the update condition gives 
$\beta_k\|v^k\|<\min\{\rho_k^{-1},{\rm dist}(z_2^k,K)\}$. In particular, $\beta_k\|v^k\|<\rho_k^{-1}$ and ${\rm dist}(z_2^k,K)>0$ for each $k\in\mathcal{K}$. Since $\beta_k\ge\beta_{\min}>0$ for each $k\in\mathcal{K}$ and $\lim_{\mathcal{K}\ni k\to\infty}\rho_k=\infty$,  we obtain  $\lim_{\mathcal{K}\ni k\to\infty}v^k=0$. By Assumption \ref{ass1}(i) and the closedness of $\mathcal{M}$, passing to a subsequence if necessary, we may assume that $\lim_{\mathcal{K}\ni k\to\infty} x^k=x^*\in\mathcal{M}$. The first inequality in \eqref{inexact-cond2}, together with the definition of $G(x^k,\cdot)$, yields $\lim_{\mathcal{K}\ni k\to\infty} z^k=z^*=(x^*,g(x^*))$. As in Case 1, the sequence
$\{\xi_1^k\}_{k\in\mathcal{K}}$ is bounded.
Since $z_2^k\notin K$ and
$\rho_k^{-1}\xi_2^k\in\partial{\rm dist}(z_2^k,K)$,
Lemma \ref{subdiff-dist} implies that $|\!\lVert\rho_k^{-1}\xi_2^k\rVert\!|_*=1$. Passing to a further subsequence if necessary, we obtain $\lim_{\mathcal{K}\ni k\to\infty}
\rho_k^{-1}\xi_2^k=\xi_2^*$ for some $\xi_2^*$ satisfying 
$\xi_2^*\in\partial{\rm dist}(g(x^{*}),K)$ and $|\!\lVert\xi_2^*\lVert\!|_*=1$. Dividing the second inequality of \eqref{inexact-cond2} by $\rho_k$ gives
\begin{equation*}
\big\|\mathcal{P}_{T_{x^k}\mathcal{M}}
\big[\rho_k^{-1}\big(\nabla\ell_k(v^k)+\beta_kv^k+\xi_1^k\big)
+\nabla g(x^k)(\rho_k^{-1}\xi_2^k)\big]\big\|
\le\rho_k^{-1}b_k\|v^k\|.
\end{equation*}
Since $\{\nabla\ell_k(v^k)\}_{k\in\mathcal{K}}$
and $\{\xi_1^k\}_{k\in\mathcal{K}}$ are bounded,
$\beta_kv^k\to0$, and $b_k\le b_{\max}$,
passing to the limit yields $\mathcal{P}_{T_{x^*}\mathcal{M}}
\big[\nabla g(x^*)\xi_2^*\big]=0$. Equivalently, $0\in\nabla g(x^*)\xi_2^*+N_{x^*}\mathcal{M}$. This, together with $\xi_2^*\in\partial{\rm dist}(g(x^{*}),K)$ and $|\!\lVert\xi_2^*\lVert\!|_*=1$, contradicts the extended CQ at $x^*$.

Both cases lead to contradictions. Therefore, the sequence
$\{\rho_k\}_{k\in\mathbb{N}}$ is bounded.
\end{proof}

\section{Iteration complexity analysis}\label{sec4}

In this section, we establish iteration complexity bounds for Algorithm \ref{PenAl} under Assumption \ref{ass1}. To this end, define the index sets $\mathbb{N}_1:=\{k\in\mathbb{N} \mid \text{Step~\ref{alg:step5} is executed at iteration }k\}$ and $\mathbb{N}_2:=\{k \in \mathbb{N} \mid \text{Step~\ref{alg:step6} is executed at iteration }k\}$. We enumerate the elements of \(\mathbb N_1\) and \(\mathbb N_2\) in ascending order. For each $k\in\mathbb{N}$, let \(\mathcal F_k\) denote the block of consecutive unsuccessful iteration indices beginning at \(k\), and define $\mathcal{F}_k^{2}:=\mathcal{F}_k\cap \mathbb{N}_2$. More precisely, $\mathcal{F}_k=\emptyset$ if $k\in\mathcal{S}$; otherwise, $\mathcal{F}_k=\big\{j\in\mathbb{N}\mid j\ge k\big\}$ when $\mathcal{S}$ is finite and \(k\) is larger than the largest element of \(\mathcal S\), or $\mathcal{F}_k=\!\{k,\ldots,k+\!n_k\}$ for some $n_k\in\mathbb{N}$ when $\mathcal{S}$ is infinite. Clearly, $\mathcal{F}_k^{2}\subset\{j\in\mathbb{N}\,|\,\textrm{Step \ref{alg:step6} is executed at iteration $j$}\}$. 

We first establish three auxiliary lemmas for the subsequent complexity analysis. Specifically, Lemma \ref{bound-lemma} establishes the boundedness of $\{(v^k,\overline{v}^k,z^k,\zeta^k,\xi^k)\}_{k\in\mathbb{N}}$, Lemma \ref{lemma1-complexity} provides a uniform bound on $|\mathcal{F}_k^2|$, and Lemma \ref{Lemma-Sfinite} shows that the point generated by the last successful iteration is a KKT point of \eqref{Rcprob} when $\mathcal{S}$ is finite. The second-order splitting-residual bound in \eqref{inexact-cond2}, together with the model and retraction error estimates in Lemmas \ref{Lemma1-Thetak} and \ref{Lemma2-Thetak}, respectively, plays a key role in establishing the uniform bound in Lemma \ref{lemma1-complexity}.
\begin{lemma}\label{bound-lemma}
 Under Assumption \ref{ass1}, the sequence $\{(v^k,\overline{v}^k,z^{k},\zeta^k,\xi^k)\}_{k\in\mathbb{N}}$ is bounded, so there exist compact convex sets  $\mathcal{X}_1\supset\bigcup_{k\in\mathbb{N}}\big([x^k+v^k,R_{x^{k}}(v^{k})]\cup \mathcal{U}_{k,\overline{\tau}}\big)$, $\mathcal{X}_2\supset\bigcup_{k\in\mathbb{N}}[x^k+v^k,R_{x^{k}}(v^{k})]$, and $\mathcal{X}_3\supset\bigcup_{k\in\mathbb{N}}[x^k,x^k+v^k]$, where $\mathcal{U}_{k,\overline{\tau}}$ is defined in Lemma \ref{Lemma1-Thetak}. Moreover, there exists a constant $M>0$ such that $\|R_{x^{k}}(v^{k})-(x^{k}\!+v^{k})\|\le M\|v^{k}\|^2$ for all $k\in\mathbb{N}$. 
\end{lemma}
\begin{proof}
 Since $-h$ is locally Lipschitz continuous and $\zeta^k\!\in\partial(-h)(x^k)$ for all $k$, by Assumption \ref{ass1}(i) and \cite[Theorem 9.13 \& Proposition 5.15]{RW98}, the sequence $\{\zeta^k\}_{k\in\mathbb{N}}$ is bounded. Since  $\varepsilon_k\le\varepsilon_0$ and $\beta_k\ge\beta_{\min}$ for all $k\in\mathbb{N}$, it follows from \eqref{inexact-cond1} that $\{v^k\}_{k\in\mathbb{N}_1}$ and $\{z^k\}_{k\in\mathbb{N}_1}$ are bounded. Since $\mathbb{N}=\mathbb{N}_1\cup\mathbb{N}_2$, to establish the boundedness of $\{(v^k,z^k)\}_{k\in\mathbb{N}}$, we only need to show that $\{v^k\}_{k\in\mathbb{N}_2}$ and $\{z^k\}_{k\in\mathbb{N}_2}$ are bounded. Suppose, on the contrary, that $\{v^k\}_{k\in\mathbb{N}_2}$ is unbounded. For each $k\in\mathbb{N}_2$, we have 
 \begin{equation*}
  \begin{aligned}
  \Theta_{\rho_k}(x^k)&\overset{\eqref{model-descent}}{>}\widehat{\Theta}_{\rho_k,\beta_k}(x^k\!+\!v^{k};x^k)=\langle\nabla\!f(x^k)\!+\!\zeta^k,v^k\rangle+(\beta_k/{2})\|v^k\|^2+\vartheta(x^k\!+\!v^k)\\
  &\qquad\qquad\qquad\qquad\qquad\ +\rho_k\textrm{dist}(\ell_g(x^k\!+\!v^k;x^k),K)+f(x^k)-h(x^k).
  \end{aligned}
 \end{equation*}
 For each $k\in\mathbb{N}$, noting that $\partial\vartheta(x^k)\ne\emptyset$, we choose an element $\widehat{\xi}_1^k\in\partial\vartheta(x^k)$. The convexity of $\vartheta$ implies that $\vartheta(x^k+v^k)\ge\vartheta(x^k)+\langle \widehat{\xi}_1^k,v^k\rangle$. Substituting it into the above inequality yields, for each $k\in\mathbb{N}_2$, 
 \[
	\Theta_{\rho_k}(x^k)>({\beta_k}/{2})\|v^k\|^2-\|\nabla\! f(x^k)+\zeta^k\!+\widehat{\xi}_1^k\|\|v^k\|+f(x^k)+\vartheta(x^k)-h(x^k),
 \]
 which, by the expression of $\Theta_{\rho_k}(x^k)$, can equivalently be rearranged into 
 \[
 -\|\nabla\! f(x^k)+\zeta^k\!+\widehat{\xi}_1^k\|\|v^k\|-\rho_{k}{\rm dist}(g(x^k),K)+({\beta_k}/{2})\|v^k\|^2\le 0.
 \]
 Since $\vartheta$ is locally Lipschitz continuous, $\{\partial\vartheta(x^k)\}_{k\in\mathbb{N}}$ is uniformly bounded on the bounded set $\{x^k\}_{k\in\mathbb{N}}$ by \cite[Theorem 9.13 \& Proposition 5.15]{RW98}. Then the sequence $\{\widehat{\xi}_1^k\}_{k\in\mathbb{N}_2}$ is bounded, and so is $\{\|\nabla\! f(x^k)+\zeta^k+\widehat{\xi}_1^k\|\}_{k\in\mathbb{N}_2}$. Together with $\beta_k\ge\beta_{\rm min}>0$ and Assumption \ref{ass1}(ii), the above inequality contradicts the unboundedness of $\{v^k\}_{k\in\mathbb{N}_2}$. Thus, $\{v^k\}_{k\in\mathbb{N}_2}$ is bounded. Since $a_k\le a_{\max}$ for all $k$, the first inequality in \eqref{inexact-cond2}, together with the boundedness of $\{x^k\}_{k\in\mathbb{N}}$ and $\{v^k\}_{k\in\mathbb{N}_2}$, implies that $\{z^k\}_{k\in\mathbb{N}_2}$ is bounded. Consequently, the sequence $\{(v^k,z^k)\}_{k\in\mathbb{N}}$ is bounded. 
 Since $\xi^k=(\xi_1^k,\xi_2^k)\in\partial\vartheta(z_1^k)\times[\rho_k\partial{\rm dist}(z_2^k,K)$ and $\vartheta$ is locally Lipschitz continuous, $\{\partial\vartheta(z_1^k)\}_{k\in\mathbb{N}}$ is uniformly bounded on the bounded set $\{z_1^k\}_{k\in\mathbb{N}}$ by \cite[Theorem 9.13 \& Proposition 5.15]{RW98}. Moreover, $\{\rho_k\}_{k\in\mathbb{N}}$ is bounded and $|\!\lVert \xi_2^k/\rho_k\lVert\!|_*\le 1$. Hence, $\{\xi^k\}_{k\in\mathbb{N}}$ is bounded. Finally, since $\overline{v}^k$ is an exact minimizer of the subproblem, 
 \[
	\widehat{\Theta}_{\rho_k,\beta_k}(x^k+\overline{v}^k;x^k)
	\le \widehat{\Theta}_{\rho_k,\beta_k}(x^k;x^k)=
 	\Theta_{\rho_k}(x^k).
 \]
 Applying the preceding coercivity argument to $\overline{v}^k$ shows that $\{\overline{v}^k\}_{k\in\mathbb{N}}$ is bounded. Thus, the sequence $\{(v^k,\overline{v}^k,z^{k},\zeta^k,\xi^k)\}_{k\in\mathbb{N}}$ is bounded. The boundedness of $\{v^k\}_{k\in\mathbb{N}}$ implies that there exists $\delta_1>0$ such that $\|v^k\|\le\delta_1$ for all $k\in\mathbb{N}$. Invoking He et al. \cite[Lemma 2.3]{He2026} with $\varLambda=\varLambda_0$ from Remark \ref{remark-ass1}(a) and $\delta=\delta_1$, there exists a constant $M>0$ such that $\|R_{x^{k}}(v^{k})-(x^{k}\!+v^{k})\|\le M\|v^{k}\|^2$ for all $k\in\mathbb{N}$. 
\end{proof}
\begin{lemma}\label{lemma1-complexity}
Under Assumption \ref{ass1}, the following assertions hold. 
\begin{itemize}
 \item[(i)] For each $k\in\mathbb{N}$, $|\mathcal{F}_k^2|\le N:=\max\{N_1,N_2\}$, where
 \begin{equation*}	
 \begin{aligned}
 N_{1}:=\Big\lceil\frac{\log\big[2\big(b_{\max}+2a_{\max}(L_{\vartheta}^2\!+\!c_u^2\overline{\rho}^2)^{1/2}+\overline{\gamma}\big)\big]-\log\beta_{\min}}{\log \varsigma_1}\Big\rceil_++1,\qquad\qquad\\ 
 N_2:=\!\bigg\lceil\frac{\log\big[2(1\!-\!\eta_1)\big(b_{\max}+\!2a_{\max}(L_{\vartheta}^2\!+\!c_u^2\overline{\rho}^2)^{1/2}\big)+2L]-\!\log\left[(1\!-\!\eta_1)\beta_{\min}\right]}{\log \varsigma_1}\bigg\rceil_+\!+\!1
 \end{aligned}
 \end{equation*}
 with $L:=\frac{1}{2}(L_{\nabla\!f}+c_u\overline{\rho} L_{\nabla g})+(L_{f}+c_u\overline{\rho}L_{g}+L_{\vartheta}+L_{h})M$. Here, $M$ is the same as in Lemma \ref{bound-lemma}, $L_{\vartheta}$ is a Lipschitz constant of $\vartheta$ on $\mathcal{X}_1$, while $L_{f},L_{h}$ and $L_{g}$ are Lipschitz constants of $f,h$ and $g$ on $\mathcal{X}_1\cap\mathcal{X}_2$, respectively. Furthermore, $L_{\nabla\!f}$ and $L_{\nabla g}$ are Lipschitz constants of $\nabla\!f$ and $\nabla g$ on $\mathcal{X}_3$, respectively.  
		
 \item[(ii)] For each $k\in\mathbb{N}$, $\beta_k\le \beta_{\rm min}\varsigma_1^{N}:=\overline{\beta}$.
\end{itemize} 
\end{lemma}
\begin{proof}
{\bf(i)} Fix any $k\in\mathbb{N}$. We only need to consider that $\mathcal{F}_k^2\neq\emptyset$, in which case $k\notin\mathcal{S}$. Suppose, to the contrary, that $|\mathcal{F}_k^2|>N$. Since Step \ref{alg:step6} of Algorithm \ref{PenAl} is executed at every iteration $j\in\mathcal{F}_k^2$, there exist indices $j_1,\ldots,j_N\in\mathbb{N}_{+}$ satisfying $k\le j_1<\cdots<j_N$ such that $\{j_1,\ldots,j_{N}\}\subset\mathcal{F}_k^2$. By the definition of $\mathcal{F}_k^2$ and Step \ref{alg:step6}, for each $i\in[N]_{+}$, either $\text{pred}_{j_i}<\overline{\gamma}\|v^{j_i}\|^2$ or $r_{j_i}<\eta_1$. We next show that both $\text{pred}_{j_{N}}\ge\overline{\gamma}\|v^{j_N}\|^2$ and $r_{j_N}\ge\eta_1$, which contradicts $|\mathcal{F}_k^2|>N$. Indeed, since $\mathcal{U}_{j_N,\overline{\tau}}\subset \mathcal{X}_1$, we have $\widetilde{L}_{\vartheta,j_{N}}\le L_{\vartheta}$, where $\widetilde{L}_{\vartheta,j_{N}}$ is defined as in Lemma \ref{Lemma1-Thetak}. Using Lemma \ref{Lemma1-Thetak} with  $k=j_{N}$ and $v=0$, and noting that $\Theta_{\rho_{j_{N}}}(x^{j_{N}})=\widehat{\Theta}_{\rho_{j_N},\beta_{j_N}}(x^{j_{N}};x^{j_{N}})$, we get
\begin{equation*}	
 \begin{aligned}
 {\rm pred}_{j_{N}}&=\widehat{\Theta}_{\rho_{j_N},\beta_{j_N}}(x^{j_{N}};x^{j_{N}})-\widehat{\Theta}_{\rho_{j_{N}},\beta_{j_N}}(x^{j_{N}}\!+\!v^{j_{N}};x^{j_{N}})\\
 &\ge\Big(\frac{1}{2}\beta_{j_{N}}-b_{j_{N}}-2a_{j_{N}}(\widetilde{L}^2_{\vartheta,{j_{N}}}\!+\!c_{u}^2\rho_{j_{N}}^2)^{1/2}\Big)\|v^{j_{N}}\|^2\\
 &\ge\Big(\frac{1}{2}\beta_{\rm min}\varsigma_1^{N-1}-b_{\max}-2a_{\max}(L_{\vartheta}^2\!+c_{u}^2\overline{\rho}^2)^{1/2}\Big)\|v^{j_{N}}\|^2\ge\overline{\gamma} \|v^{j_{N}}\|^2,
\end{aligned}
\end{equation*}
where the second inequality follows from the update rule for $\beta_k$ in Step \ref{alg:step6}, together with $b_{j_{N}}\in(0,b_{\rm max}]$, $a_{j_{N}}\in(0,a_{\rm max}]$, and Assumption \ref{ass1}(ii), while the third inequality follows from $N\ge N_1$ and the definition of $N_{1}$. Consequently, ${\rm pred}_{j_{N}}\ge\overline{\gamma}\|v^{j_{N}}\|^2$. We now show that $r_{j_{N}}\ge\eta_1$. For each $j\in\mathcal{F}_k^2$, since $[x^j+v^j,R_{x^j}(v^j)]\subset\mathcal{X}_1\cap\mathcal{X}_2$ and $[x^j,x^j+v^j]\subset\mathcal{X}_3$, the Lipschitz moduli in Lemma \ref{Lemma2-Thetak} satisfy $\widehat{L}_{f,j}\le L_{f},\widehat{L}_{\vartheta,j}\le L_{\vartheta},\widehat{L}_{h,j}\le L_{h},\widehat{L}_{g,j}\le L_{g}$ and $\widehat{L}_{\nabla\!f,j}\le L_{\nabla\!f},\widehat{L}_{\nabla g,j}\le L_{\nabla g}$. Moreover, by Lemma \ref{bound-lemma}, $\|R_{x^{j}}(v^{j})-(x^{j}+v^{j})\|\le M\|v^{j}\|^2$. By Lemma \ref{Lemma2-Thetak}, Assumption \ref{ass1}(ii), and the definition of $L$, we have $\Theta_{\rho_{j_{N}}}(R_{x^{j_{N}}}(v^{j_{N}}))\le\!\widehat{\Theta}_{\rho_{j_{N}}}(x^{j_{N}}+v^{j_{N}};x^{j_{N}})+L\|v^{j_{N}}\|^2$. Together with the definition of $r_{j_{N}}$, it then follows that 
\begin{equation*}	
\begin{aligned}
r_{j_{N}}&\ge \frac{\Theta_{\rho_{j_{N}}}(x^{j_{N}})-\widehat{\Theta}_{\rho_{j_{N}},\beta_{j_{N}}}(x^{j_{N}}+v^{j_{N}};x^{j_{N}})-L\|v^{j_{N}}\|^2}{\Theta_{\rho_{j_{N}}}(x^{j_{N}})- \widehat{\Theta}_{\rho_{j_{N}},\beta_{j_{N}}}(x^{j_{N}}+v^{j_{N}};x^{j_{N}})}\\
&\ge \frac{\big(\frac{1}{2}\beta_{j_{N}}-b_{j_{N}}-2a_{j_{N}}(\widetilde{L}^2_{\vartheta,j_{N}}\!+c_u^2\rho_{j_{N}}^2)^{1/2}\big)\|v^{j_{N}}\|^2-L\|v^{j_{N}}\|^2}{\big(\frac{1}{2}\beta_{j_{N}}-b_{j_{N}}-2a_{j_{N}}(\widetilde{L}^2_{\vartheta,j_{N}}+c_u^2\rho_{j_{N}}^2)^{1/2}\big)\|v^{j_{N}}\|^2}\\
&\ge\frac{\big(\frac{1}{2}\beta_{\min}\varsigma_1^{N-1}-b_{\max}-2a_{\max}(L_{\vartheta}^2+c_u^2\overline{\rho}^2)^{1/2}\big)\|v^{j_{N}}\|^2-L\|v^{j_{N}}\|^2}{\big(\frac{1}{2}\beta_{\min}\varsigma_1^{N-1}-b_{\max}-2a_{\max}(L_{\vartheta}^2\!+c_u^2\overline{\rho}^2)^{1/2}\big)\|v^{j_{N}}\|^2}\ge \eta_1,
\end{aligned}
\end{equation*}
where the second inequality follows from Lemma \ref{Lemma1-Thetak} with $k=j_N$ and $v=0$, together with the nondecreasing of $\mathbb{R}_{++}\ni t\mapsto\frac{t-L\|v^{j_{N}}\|^2}{t}$, and the fourth one follows from $N\ge N_2$ and the definition of $N_{2}$. Thus, $r_{j_N}\ge\eta_1$. 
	
\noindent
{\bf(ii)} Suppose, to the contrary, that $\beta_{k_0}\!>\beta_{\rm min}\varsigma_1^{N}$ for some $k_0\in\mathbb{N}$. Since the parameter $\beta_k$ is updated only at Step \ref{alg:step6}, there exists an index $k_1<k_0$ such that $\beta_{k_1}\ge\beta_{\rm min}\varsigma_1^{N-1}$ and $\beta_k$ is increased at the $k_1$-th iteration. It then follows from Step \ref{alg:step6} that either $\text{pred}_{k_1}<\overline{\gamma}\|v^{k_1}\|^2$ or $r_{k_1}<\eta_1$. On the other hand, following arguments similar to those for item (i) can show that ${\rm pred}_{k_1}\ge\overline{\gamma}\|v^{k_1}\|^2$ and $r_{k_1}\ge\eta_1$, a contradiction.
\end{proof}
\begin{lemma}\label{Lemma-Sfinite}
Let $k_{\mathcal{S}}$ denote the largest element of $\mathcal{S}$ whenever $\mathcal{S}$ is nonempty and finite, and set $k_{\mathcal{S}}\!:=-1$ when $\mathcal{S}=\emptyset$. Under Assumption \ref{ass1}, if $\mathcal{S}$ is finite, then $x^k=x^{k_{\mathcal{S}}+1}$ for all $k\ge k_{\mathcal{S}}+1$, and $x^{k_{\mathcal{S}}+1}$ is a KKT point. 
\end{lemma}
\begin{proof}
By the definition of $\mathcal{S}$ and the iterative mechanism of Algorithm \ref{PenAl}, the finiteness of $\mathcal{S}$ implies $x^k=x^{k_{\mathcal{S}}+1}$ for all $k\ge k_{\mathcal{S}}+1$. It remains to prove that $x^{k_{\mathcal{S}}+1}$ is a KKT point of \eqref{Rcprob}. By the finiteness of $\mathcal{S}$ and the boundedness of $\{\beta_k\}_{k\in\mathbb{N}}$ established in Lemma \ref{lemma1-complexity}(ii), the condition $r_k<\eta_1$ or ${\rm pred}_k<\overline{\gamma}\|v^k\|^2$ can occur only finitely many times for $k\ge k_{\mathcal{S}}+1$. Indeed, each occurrence of this condition triggers an increase of $\beta_k$ by the factor $\varsigma_1>1$, and hence infinitely many such occurrences would contradict the boundedness of $\{\beta_k\}_{k\in\mathbb{N}}$. Therefore, there exists an index $\widehat{k}\in\mathbb{N}$ such that Algorithm~\ref{PenAl} executes Step~\ref{alg:step5} for all $k\ge\widehat{k}$. Thus, 
\begin{equation}\label{lim-Qkvk}
 \left\{\begin{aligned} &\lim_{k\to\infty}\varepsilon_k=0,\,\lim_{k\to\infty}\|\beta_kv^k\|=0,\, \lim_{k\to\infty}\|G(x^k,v^k)-z^k\|=0,\\
 &\lim_{k\to\infty}\big\|\mathcal{P}_{T_{\!x^k}\mathcal{M}}(\nabla \ell_k(v^k)+\beta_kv^k+\!\nabla_{\!v}G(x^k,v^k)\xi^k)\big\|=0.
\end{aligned}\right.
\end{equation}
The second limit in \eqref{lim-Qkvk}, combined with $\beta_k\ge\beta_{\min}$, yields $\lim_{k\to\infty}v^k=0$. From the third limit in \eqref{lim-Qkvk} and $x^k=x^{k_{\mathcal{S}}+1}$ for all $k\ge k_{\mathcal{S}}+1$, we have $\lim_{k\to\infty}z^k=\lim_{k\to\infty}G(x^k,v^k)=(x^{k_{\mathcal{S}}+1};g(x^{k_{\mathcal{S}}+1}))$. By Lemma \ref{bound-lemma}, the sequence $\{(\zeta^k,\xi^k)\}_{k\in\mathbb{N}}$ is bounded. Passing to a subsequence if necessary, we may assume $\lim_{k\to\infty}(\zeta^k,\xi^k)=(\zeta^*,\xi^*)$. Then, the outer semicontinuity of the multifunctions $\partial(-h)(\cdot)$ and $\partial\psi_{\overline{\rho}}(\cdot)$ implies that $\zeta^*\in \partial(-h) (x^{k_{\mathcal{S}}+1})$ and $\xi^*=(\xi^*_1;\xi^*_2)\in\partial \vartheta(x^{k_{\mathcal{S}}+1})\times [\overline{\rho}\partial{\rm dist}(g(x^{k_{\mathcal{S}}+1}),K)]$. Combining the second and fourth limits in \eqref{lim-Qkvk} with the continuity of $\mathcal{P}_{T_{x}\mathcal{M}}$ as a function of \(x\in\mathcal{M}\) yields
\[ \mathcal{P}_{T_{\!x^{k_{\mathcal{S}}+1}}\mathcal{M}}\big[\nabla\!f(x^{k_{\mathcal{S}}+1})+\zeta^*+\xi^*_1+\nabla g(x^{k_{\mathcal{S}}+1})\xi^*_2\big]=0.
\]
By Assumption \ref{ass1}(ii) and Step \ref{alg:step5}, ${\rm dist}(z^k_2,K)\le \varepsilon_k$ for all sufficiently large $k$. Passing to the limit $k\to\infty$ in this inequality yields ${\rm dist}(g(x^{k_{\mathcal{S}}+1}),K)=0$. Together with $\xi^*_2\in\overline{\rho} \partial{\rm dist}(g(x^{k_{\mathcal{S}}+1}),K)$ and Lemma \ref{subdiff-dist}, we have $\xi^*_2\in\mathcal{N}_{K}(g(x^{k_{\mathcal{S}}+1}))$. Combining this with the above equality shows that $x^{k_{\mathcal{S}}+1}$ is a KKT point. 
\end{proof}

In view of Lemma \ref{Lemma-Sfinite}, we focus on the nontrivial case \(\mathcal S\neq\emptyset\) in the remainder of this paper.
\subsection{Outer iteration complexity}\label{sec4.1}

The following proposition provides an upper bound on the number of Step \ref{alg:step5} executions required by Algorithm \ref{PenAl} to generate an $\epsilon$-approximate KKT point of \eqref{Rcprob}. 
\begin{proposition}\label{prop1-complexity}
Suppose that Assumption \ref{ass1} holds. Let
$\epsilon\in(0,\varepsilon_0)$, and let $n_\rho$
denote the total number of increases in the penalty parameter. 
Define $c_K:=\sup\big(\{0\}\cup
\{{\rm dist}(z_2^k,K):k\in\mathbb{N}_1\}\big)$ and 
\[
  T:=\left\{\begin{array}{cl}
		T_\epsilon
		&{\rm if}\ \mathbb{N}_1=\emptyset\ \text{or}\ c_K=0,\\
		T_\rho+T_\epsilon
		& {\rm if}\ \mathbb{N}_1\neq\emptyset\ \text{and}\ c_K>0,
	\end{array}\right.
\]
where $T_\epsilon:=\!\Big\lceil	\frac{\log(\chi^{-1}\varepsilon_0\epsilon^{-1})}	{\log\varsigma^{-1}}\Big\rceil+1$ and $T_\rho:=\min\Big\{	n_\rho,\Big\lceil\frac{\log\!(	2c_l^{-1}c_K\overline{\rho}\rho_0^{-1}\epsilon^{-1})}{\log\tau}\Big\rceil_+\Big\}$. Here, $\chi\!:=\!\frac{\min\{1,c_{l}\}\beta_{\min}}{\max\{1+\beta_{\min},2(\beta_{\min}+c_{\nabla\!g})\}}$ and $c_{\nabla\!g}:=\sup(\{0\}\cup\{\|\nabla g(x^k)\|\!: k\in\mathbb{N}_1\})$. If $T\le |\mathbb{N}_1|$, then there exists an index $i_1$ among the first $T$ elements of $\mathbb{N}_1$ such that $x^{i_1}$ is an $\epsilon$-approximate KKT point of \eqref{Rcprob}. 
\end{proposition}
\begin{proof}
 By Assumption \ref{ass1}(i) and the continuity of $\nabla g$, the constant $c_{\nabla g}$ is well defined. Hence, the constant $\chi$ is well defined and satisfies $0<\chi<\frac{1}{2}\min\{1,c_{l}\}$. Consequently, $T_{\epsilon}\ge 1$ is well defined. Let $\widehat{N}_1$ consist of the first $T$ elements of $\mathbb{N}_1$. Since $T\le |\mathbb{N}_1|$, we have $\widehat{\mathbb{N}}_1\ne\emptyset$. Suppose that there is an index $i_1\in\widehat{\mathbb{N}}_1$ such that 
 \begin{equation}\label{aim-ineq-i1}
 \varepsilon_{i_1}\le \chi\epsilon\ \ {\rm and}\ \ {\rm dist}(z_2^{i_1},K)\le c_l\epsilon/2.
 \end{equation}
 Then the triple $(v^{i_1},z^{i_1},\xi^{i_1})\in T_{\!x^{i_1}}\mathcal{M}\times\mathbb{Z}\times\partial\psi_{\rho_{i_1}}(z^{i_1})$ satisfies \eqref{inexact-cond1} for $k=i_1$. Since $\varepsilon_{i_1}\le \chi\epsilon$, we have 
 \begin{equation*}
 \left\{\begin{aligned}
  &\|x^{i_1}+v^{i_1}-z^{i_1}_1\|\le \chi\epsilon,\,\|g(x^{i_1})+g'(x^{i_1})v^{i_1}-z^{i_1}_2\|\le \chi\epsilon,\,\|\beta_{i_1}v^{i_1}\|\le \chi\epsilon,\\
  &\big\|\mathcal{P}_{T_{\!x^{i_1}}\mathcal{M}}\big[\nabla\! f(x^{i_1})+\zeta^{i_1}+\xi_1^{i_1}+\nabla g(x^{i_1})\xi_2^{i_1}+\beta_{i_1}v^{i_1}\big]\big\|\le \chi\epsilon.
 \end{aligned}\right.
 \end{equation*}
 The third inequality, together with $\beta_{k}\ge\beta_{\min}$ for all $k\in\mathbb{N}$, yields $\|v^{i_1}\|\le \beta_{\min}^{-1}\chi\epsilon$, and consequently, 
 \begin{equation}\label{bound-chieps}
 \left\{\begin{aligned}
  &\|x^{i_1}-z^{i_1}_1\|\le (1+\beta^{-1}_{\min})\chi\epsilon,\,\|g(x^{i_1})-z^{i_1}_2\|\le (1+c_{\nabla g}\beta^{-1}_{\min})\chi\epsilon,\\			&\big\|\mathcal{P}_{T_{\!x^{i_1}}\mathcal{M}}\big(\nabla\! f(x^{i_1})+\zeta^{i_1}+\xi_1^{i_1}+\nabla g(x^{i_1})\xi_2^{i_1}\big)\big\|\le 2\chi\epsilon.
  \end{aligned}\right.
 \end{equation}
 Note that $\xi_2^{i_1}\in\rho_{i_1}\partial{\rm dist}(z_2^{i_1},K)$. By Lemma \ref{subdiff-dist}, there exists $u_2^*\in\mathcal{P}_K(z^{i_1}_2)$ such that $\xi_2^{i_1}\in\mathcal{N}_{K}(u_2^*)$. Using $c_{l}\|\cdot\|\le|\!\lVert\cdot\rVert\!|$, the second inequality of \eqref{bound-chieps}, and that of \eqref{aim-ineq-i1} yields that 
 \begin{equation*}
  \|g(x^{i_1})\!-u_2^*\|\le\|g(x^{i_1})\!-z_2^{i_1}\|+c_l^{-1}{\rm dist}(z_2^{i_1},K)\le\!\big[(1\!+c_{\nabla\! g}\beta^{-1}_{\min})\chi+1/2\big]\epsilon.   
 \end{equation*} 
 This, together with the first and third inequalities in \eqref{bound-chieps} and the definition of $\chi$, shows that $x^{i_1}$ is an $\epsilon$-approximate KKT point of \eqref{Rcprob}. Thus, it remains to show that some index \(i_1\in\widehat{\mathbb{N}}_1\) satisfies \eqref{aim-ineq-i1}. To this end, let $i_T$ denote the $T$-th element of $\widehat{\mathbb{N}}_1$, and distinguish between the cases $c_K\le c_l\epsilon/2$ and $c_K> c_l\epsilon/2$.
	
 \noindent
 {\bf Case 1: $c_K\le c_l\epsilon/2$.} We have ${\rm dist}(z_2^{i_T},K)\le c_K\le c_l\epsilon/2$. Since $\varepsilon_k$ is updated only at Step \ref{alg:step5}, combining the definition of $\widehat{\mathbb{N}}_1$ and the update rule of $\varepsilon_k$ in Step \ref{alg:step5} yields $\varepsilon_{i_T}=\varepsilon_0\varsigma^{T-1}\le\chi\epsilon$,  
 where the inequality follows from $T\ge T_\epsilon$ and $\varsigma\in(0,1)$. Thus, choosing \(i_1=i_T\) establishes the claim in this case.
	
\noindent
{\bf Case 2: $c_K>c_l\epsilon/2$.} Now we have $c_{K}\in(0,\infty)$, since Lemma \ref{bound-lemma} ensures the boundedness of $\{z^k\}_{k\in\mathbb{N}}$. Then  $N_{\rho}:=\Big\lceil\frac{\log(2c_{l}^{-1}c_{K}\overline{\rho}\rho_0^{-1}\epsilon^{-1})}{\log \tau }\Big\rceil_{+}$ is well defined. Since each penalty parameter update multiplies \(\rho_k\) by \(\tau\), and \(\rho_k\) eventually reaches \(\overline{\rho}\), we have $\overline{\rho}=\rho_0\tau^{n_\rho}$, and hence $n_{\rho}=\frac{\log(\overline{\rho}/\rho_0)}{\log\tau}$. Thus, we infer from $c_K>c_l\epsilon/2$ that $N_{\rho}=\big\lceil n_{\rho}+\frac{{\rm log}(2c_Kc_l^{-1}\epsilon^{-1})}{{\rm log}\tau}\big\rceil> n_{\rho}$. Define $T_{\le}\!:=|\{k\in\widehat{\mathbb{N}}_1\mid {\rm dist}(z_2^k,K)\le \varepsilon_k\}|$ and $T_{>}\!:=|\{k\in\widehat{\mathbb{N}}_1 \mid {\rm dist}(z_2^k,K)>\varepsilon_k\}|$. Clearly, $T=|\widehat{\mathbb{N}}_1|=T_{\le}+T_{>}$. Since Step \ref{alg:step5} gives $T_{>}\le n_{\rho}$, it follows that
\[
	T_{\le}=T-T_{>}\ge T-n_{\rho}=T_{\rho}-n_{\rho}+T_{\epsilon}=T_{\epsilon}, 
\]
where the last equality is due to $T_{\rho}=\min\{n_{\rho},N_{\rho}\}=n_{\rho}$. Let $i_1$ be the largest index in $\widehat{\mathbb{N}}_1$ at which the penalty parameter is not updated. Then there are \(T_{\le}\) such indices up to and including \(i_1\). Since $\widehat{\mathbb{N}}_1\subset\mathbb{N}_1$, by the definition of $\mathbb{N}_1$ and Step \ref{alg:step5}, we have ${\rm dist}(z_2^{i_1},K)\le \varepsilon_{i_1}$. Recall that $\varepsilon_k$ is updated only at Step \ref{alg:step5} and  reduced by the factor \(\varsigma\) at each execution of Step \ref{alg:step5}. Hence,  $\varepsilon_{i_1}\le\varepsilon_0\varsigma^{T_{\le}-1}\le\chi\epsilon$. Consequently, ${\rm dist}(z_2^{i_1},K)\le \varepsilon_{i_1}\le\chi\epsilon\le c_l\epsilon/2$, and the index $i_1$ has the desired property. Thus, we complete the proof. 
\end{proof}

We next bound the number of executions of Step~\ref{alg:step6} to obtain an $\epsilon$-approximate KKT point of \eqref{Rcprob}.
\begin{proposition}\label{prop2-complexity}
 Suppose that Assumption \ref{ass1} holds. Let $\epsilon\in(0,1)$. Set $\gamma_{K}\!:=\sup\limits_{k\in\mathbb{N}}{\rm dist}(z_2^k,K)$ and
 \[
  N_{\epsilon}:=\left\{\begin{array}{cl}
		0 & {\rm if}\ \gamma_{K}=0,\\
		\Big\lceil\frac{\log(2c_{l}^{-1}\gamma_{K}\overline{\rho}\rho_0^{-1}\epsilon^{-1})}{\log \tau }\Big\rceil_{+}&{\rm otherwise}.
	\end{array}\right.
 \]
 If $(N\!+1)\widetilde{N}_{\epsilon}\le|\mathbb{N}_2|$, then there exists an index $i_2$ among the first $(N\!+1)\widetilde{N}_{\epsilon}$ elements of $\mathbb{N}_2$ such that $x^{i_2}$ is an $\epsilon$-approximate KKT point, where $\widetilde{N}_{\epsilon}:=\big\lceil\frac{4\mu_1^2\mu_2}{\epsilon^2}\big\rceil+\min\{N_{\epsilon},n_{\rho}\}+1$ and $N$ is given in Lemma \ref{lemma1-complexity}(i). Here,  
 \begin{equation*}
  \begin{aligned}
  \mu_1\!:=\max\Big\{(1+a_{\max}\gamma_{v})/2,\,(b_{\max}+\overline{\beta})/2,\,a_{\max}\gamma_v+\gamma_{\nabla g}+c_l^{-1}\max\{1,\gamma_{K}\overline{\rho}\}\overline{\beta}\Big\},\\
  \mu_2:=(\overline{\gamma}\eta_1)^{-1}\big[\Theta_{\rho_{0}}(x^0)+\rho_0\gamma_{c}(\tau^{n_{\rho}}\!-\!1)-\min_{x\in\varLambda_0}\Theta(x)\big],\qquad\qquad
 \end{aligned}
 \end{equation*}
with $\gamma_v:=\sup\limits_{k\in\mathbb{N}}\|v^k\|,\gamma_{\nabla\!g}:=\sup\limits_{k\in\mathbb{N}}\|\nabla g(x^k)\|$ and $\gamma_{c}\!:=c_u[\max\{\varepsilon_{0},a_{\rm max}\gamma_{v}^2\}+\gamma_{\nabla\!g}\gamma_{v}]+\gamma_{K}$.
\end{proposition}
\begin{proof}
 Recall that $\mathcal{S}\ne\emptyset$. Lemma \ref{bound-lemma} ensures that $\gamma_{K}\in[0,\infty)$. Assumption \ref{ass1}(i) and Lemma \ref{bound-lemma} imply that both $\gamma_{v}$ and $\gamma_{\nabla\!g}$ are finite. We first establish the following claim.
	
 \begin{claim}\label{claim0}
 $|\mathcal{S}\cap\widehat{\mathbb{N}}_2|\ge \widetilde{N}_{\epsilon}$, where $\widehat{\mathbb{N}}_2$ consists of the first $(N\!+1)\widetilde{N}_{\epsilon}$ elements of $\mathbb{N}_2$.
 \end{claim}	
 Suppose, on the contrary, that $|\mathcal{S}\cap \widehat{\mathbb{N}}_2|<\widetilde{N}_{\epsilon}$. It is not hard to infer that  $\widehat{\mathbb{N}}_2\backslash\mathcal{S}\ne\emptyset$. Moreover, $\mathcal{S}\cap \widehat{\mathbb{N}}_2\ne\emptyset$ (if not, the definition of $\mathcal{F}_{\!k}^2$ would give $|\mathcal{F}_{\!0}^2|\ge |\widehat{\mathbb{N}}_2|\ge N+1$, contradicting $|\mathcal{F}_{\!0}^2|\le N$ from Lemma \ref{lemma1-complexity}(i)). Hence, there exist $l\in\mathbb N_+$ and indices $0\le j_1<\cdots<j_l$ such that $\mathcal S\cap\widehat{\mathbb N}_2=\{j_1,\ldots,j_l\}$. Set $j_0:=-1$ and $j_{l+1}:=\infty$. Then the elements of $\widehat{\mathbb N}_2\backslash\mathcal S$ lie in $(j_i,j_{i+1})$ for $i=0,\ldots,l$. Define $\widetilde{k}_0:=0$ and, for each $i\in[l]$, let  
 \[
	\widetilde{k}_i:=\left\{\begin{array}{cl}
		j_{i}+1 &{\rm if}\ j_{i}+1<j_{i+1},\\
		j_i &{\rm otherwise}.
	\end{array}\right.
 \]
 It follows from Lemma \ref{lemma1-complexity}(i) that $|\mathcal{F}_{\!\widetilde{k}_i}^2|\le N$ for all $i\in[l]$. Moreover, since $|\mathcal{S}\cap \widehat{\mathbb{N}}_2|<\widetilde{N}_{\epsilon}$, we have
 \[
 |\widehat{\mathbb{N}}_2\backslash \mathcal{S}|=|\widehat{\mathbb{N}}_2|-|\mathcal{S}\cap \widehat{\mathbb{N}}_2|>(N\!+1)\widetilde{N}_{\epsilon}-\widetilde{N}_{\epsilon}=N\widetilde{N}_{\epsilon}\ \ {\rm and}\ \ l+1\le\widetilde{N}_{\epsilon}.
 \]
 Consequently, $\sum_{i=0}^{l}|\mathcal{F}_{\widetilde{k}_i}^2|\le (l\!+\!1)N\le N\widetilde{N}_{\epsilon}<|\widehat{\mathbb{N}}_2\backslash \mathcal{S}|$. On the other hand, by the definitions of $\widetilde{k}_i$ and $\mathcal{F}_{\!\widetilde{k}_i}^2$, we have $\widehat{\mathbb{N}}_2\backslash \mathcal{S}\subset \bigcup_{i\in[l]}\mathcal{F}_{\!\widetilde{k}_i}^2$, which implies that $\sum_{i=0}^{l}|\mathcal{F}_{\!\widetilde{k}_i}^2|\ge|\widehat{\mathbb{N}}_2\backslash \mathcal{S}|$, a contradiction. 
	
 Let $\mathcal{K}_{\rho}:=\emptyset$ if $n_{\rho}=0$; otherwise, let $\mathcal{K}_{\rho}:= \{k_1, \ldots, k_{n_{\rho}}\}$ with $0 \le k_1 < \cdots < k_{n_{\rho}}$, where $k_1, \ldots, k_{n_{\rho}}$ are the indices at which the penalty parameter is updated. By the penalty parameter update rule, $\rho_{k_1}=\rho_0$ and $\rho_{k_{i+1}}=\tau\rho_{k_i}$ for all $i\in[n_\rho\!-1]_{+}$. Let $\mathcal{K}_{\rho}':=\emptyset$ if $\mathcal{K}_{\rho}=\emptyset$; otherwise, let $\mathcal{K}_{\rho}'$ consist of the first $\min\{N_{\epsilon},n_{\rho}\}$ indices in $\mathcal{K}_{\rho}$. We next establish the following key property of the indices outside $\mathcal{K}_\rho'$. 
	
 \begin{claim}\label{claim1} 
 If $k\in \mathcal{S}\backslash\mathcal{K}'_{\rho}$ and $\|v^{k}\|\le {\epsilon}/{(2\mu_1)}$, then $x^k$ is an $\epsilon$-approximate KKT point.
 \end{claim}	
 Fix any $k\in \mathcal{S}\backslash\mathcal{K}'_{\rho}$. We show that ${\rm dist}(z^k_2,K)\le\max\{{c_l\epsilon}/{2},\max\{1,\gamma_{K}\overline{\rho}\}\overline{\beta}\|v^k\|\}$. It suffices to consider the case $\gamma_K>c_l\epsilon/2$, since otherwise ${\rm dist}(z^k_2,K)\le\gamma_K\le c_l\epsilon/2$. In this case, $N_{\epsilon}=\big\lceil n_{\rho}+\frac{{\rm log}[2\gamma_K(c_l\epsilon)^{-1}]}{{\rm log}\,\tau}\big\rceil> n_{\rho}$, which implies that $\mathcal{K}_{\rho}=\mathcal{K}_{\rho}'$. Then, by the definition of $\mathcal{K}_{\rho}$ and the penalty parameter update rule in \eqref{update-rho}, we have $\rho_{k}\beta_k\|v^k\|\ge 1$ or ${\rm dist}(z^k_2,K)\le \beta_k\|v^k\|$. In the former case, Assumption \ref{ass1}(ii) yields ${\rm dist}(z^k_2,K)\le\gamma_{K}\le\gamma_{K}(\overline{\rho}/\rho_k)\le\gamma_{K}\overline{\rho}\beta_k\|v^k\|$. Hence, in either case,   
 ${\rm dist}(z^k_2,K)\le \max\{1,\gamma_{K}\overline{\rho}\}\beta_k\|v^k\|\le \max\{1,\gamma_{K}\overline{\rho}\}\overline{\beta}\|v^k\|$. Consequently, ${\rm dist}(z^k_2,K)\le\max\big\{{c_{l}\epsilon}/{2},\max\{1,\gamma_{K}\overline{\rho}\}\overline{\beta}\|v^k\|\big\}$ for all $k\in\mathcal{S}\backslash\mathcal{K}_{\rho}'$. For such $k$, since $\xi_2^k\in\rho_k\partial{\rm dist}(z_2^k,K)$, by Lemma \ref{subdiff-dist}, there exists $u_2^k\in\mathcal{P}_{K}(z_2^k)$ such that $\xi_2^k\in\mathcal{N}_{K}(u_2^k)$. From \eqref{inexact-cond2}, we obtain 
 \begin{subequations}
 \begin{equation}
 \|x^k\!-\!z_1^k\|\le \|x^k+v^k\!-\!z^k_1\|+\|v^k\|\le a_k\|v^k\|^2+\|v^k\|\le (1\!+\!a_{\max}\gamma_v)\|v^k\|,
 \end{equation}
 \begin{equation}
 \begin{aligned}\label{aim-ineq42}
 &\|g(x^k)-u_2^k\|=\|g(x^k)+g'(x^k)v^k-z^k_2+z_2^k-u_2^k-g'(x^k)v^k\|\nonumber\\
 &\qquad\qquad\qquad\le \|g(x^k)+g'(x^k)v^k-z^k_2\|+\|\nabla g(x^k)\|\|v^k\|+c_l^{-1}{\rm dist}(z^k_2,K)\nonumber\\
 &\qquad\qquad\qquad\le (a_{\max}\gamma_v+\gamma_{\nabla\!g})\|v^k\|+c_l^{-1}{\rm dist}(z^k_2,K)\nonumber\\
 &\qquad\qquad\qquad\le \big(a_{\max}\gamma_v+\gamma_{\nabla\!g}+c_l^{-1}\max\{1,\gamma_{K}\overline{\rho}\}\overline{\beta}\big)\|v^k\|+\epsilon/2,
 \end{aligned}
 \end{equation}
 \begin{equation}\label{aim-ineq43}
\|\mathcal{P}_{T_{\!x^k}\mathcal{M}}\big[\nabla\! f(x^k)+\zeta^k+\xi_1^k+\nabla g(x^k)\xi_2^k\big]\|\le (b_k+\beta_k)\|v^k\|\le (b_{\max}\!+\overline{\beta})\|v^k\|,
\end{equation}
 \end{subequations}
 where the second inequality in \eqref{aim-ineq43} is due to $b_k\le b_{\max}$, and $\beta_k\le\overline{\beta}$ in Lemma \ref{lemma1-complexity}(ii). Thus, for any $k\in \mathcal{S}\backslash\mathcal{K}'_{\rho}$ and $\|v^{k}\|\le {\epsilon}/{(2\mu_1)}$, the definition of $\mu_1$ implies that $x^{k}$ is an $\epsilon$-approximate KKT point of \eqref{Rcprob}.  
	
 By Claim \ref{claim0}, $|\mathcal{S}\cap\widehat{\mathbb{N}}_2|\ge \widetilde{N}_{\epsilon}$. Let $J\!:=\big\{j_1,\ldots,j_{\widetilde{N}_{\epsilon}}\big\}$ consist of the first $\widetilde{N}_{\epsilon}$ elements of $\mathcal{S}\cap\widehat{\mathbb{N}}_2$, arranged so that $0\le j_1<\cdots< j_{\widetilde{N}_{\epsilon}}$. We next prove the existence of ${k_{\epsilon}}\in J\backslash\mathcal{K}_{\rho}'$ such that $\|v^{k_{\epsilon}}\|\le {\epsilon}/{(2\mu_1)}$. If such ${k_{\epsilon}}$ exists, since $J\backslash\mathcal{K}_{\rho}'\subset [\mathcal{S}\cap\widehat{\mathbb{N}}_2]\backslash\mathcal{K}_{\rho}'\subset\mathcal{S}\backslash\mathcal{K}_{\rho}'$, it  follows from Claim \ref{claim1} that $x^{k_{\epsilon}}$ is an $\epsilon$-approximate KKT point of \eqref{Rcprob}, and consequently, the desired conclusion follows. Suppose, to the contrary, that $\|v^j\|>{\epsilon}/{(2\mu_1)}$ for all $j\in J\backslash\mathcal{K}_{\rho}'$. Since $|J\backslash\mathcal{K}_{\rho}'|\ge\widetilde{N}_{\epsilon}-\min\{N_{\epsilon},n_{\rho}\}> {4\mu_1^2\mu_2}/{\epsilon^2}$, we have 
 \begin{equation}\label{temp-ineq42}
 \mu_2< \sum_{j\in  J\backslash\mathcal{K}_{\rho}'}\frac{\epsilon^2}{4\mu_1^2}< \sum_{j\in  J\backslash\mathcal{K}_{\rho}'}\|v^j\|^2\le \sum_{j\in J}\|v^j\|^2.
 \end{equation}
 To achieve a contradiction, we prove that $\sum_{j\in J}\|v^j\|^2\le\mu_2$. Since \eqref{S-descent} holds for all $j\in\mathcal{S}$ by Remark \ref{remark-alg}(c), $\|v^{j}\|^2\le (\overline{\gamma}\eta_1)^{-1}\big[\Theta_{\rho_{j}}(x^{j})-\Theta_{\rho_{j}}(R_{x^{j}}(v^{j}))\big]$ for all $j\in J$. Together with $R_{x^{j}}(v^{j})=x^{j+1}$ for each $j\in J$, it follows that
 \begin{equation*}
 \begin{aligned}
 \sum_{j\in J}\|v^j\|^2\le \frac{1}{\overline{\gamma}\eta_1} \sum_{j\in J}\big[\Theta_{\rho_{j}}(x^{j})-\Theta_{\rho_{j}}(x^{j+1})\big]&=\frac{1}{\overline{\gamma}\eta_1}\!\!\sum_{k\in J\cup([j_{\widetilde{N}_{\varepsilon}}]\backslash J)}\!\!\big[\Theta_{\rho_{k}}(x^{k})-\Theta_{\rho_{k}}(x^{k+1})\big]\\ 
 &=\frac{1}{\overline{\gamma}\eta_1} \sum_{k=0}^{j_{\widetilde{N}_{\epsilon}}}\left[\Theta_{\rho_k}(x^k)-\Theta_{\rho_{k}}(x^{k+1})\right],
 \end{aligned}
\end{equation*}
where the first equality is due to $x^k=x^{k+1}$ for all $k\in[j_{\widetilde{N}_{\varepsilon}}]\backslash J$. Note that $\Theta_{\rho_k}(x^k)-\Theta_{\rho_{k}}(x^{k+1})=\Theta_{\rho_k}(x^k)-\Theta_{\rho_{k+1}}(x^{k+1})+(\rho_{k+1}-\rho_k){\rm dist}(g(x^{k+1}),K)$. Then, we have
 \begin{equation}\label{temp-ineq43}
  \begin{aligned}
  \sum_{j\in J}\|v^j\|^2&\le \frac{1}{\overline{\gamma}\eta_1}\Big[\Theta_{\rho_0}(x^0)-\Theta_{\rho_{j_{\widetilde{N}_{\epsilon}}+1}}(x^{j_{\widetilde{N}_{\epsilon}}+1})+\sum_{k=0}^{j_{\widetilde{N}_{\epsilon}}}(\rho_{k+1}-\rho_k){\rm dist}(g(x^{k+1}),K)\Big]\\
 &\le\frac{1}{\overline{\gamma}\eta_1}\Big[\Theta_{\rho_0}(x^0)-\Theta_{\rho_{j_{\widetilde{N}_{\epsilon}}+1}}(x^{j_{\widetilde{N}_{\epsilon}}+1})+(\tau\!-\!1)\sum_{j\in\mathcal{K}_{\rho}}\rho_j{\rm dist}(g(x^{j+1}),K)\Big]\\
 &\le\frac{1}{\overline{\gamma}\eta_1}\Big[\Theta_{\rho_0}(x^0)-\min_{x\in\varLambda_0}\Theta(x)+(\tau\!-\!1)\sum_{j\in\mathcal{K}_{\rho}}\rho_j{\rm dist}(g(x^{j+1}),K)\Big],
 \end{aligned}
 \end{equation}
 where the second inequality is due to $\rho_{j+1}=\tau\rho_j$ for $j\in \mathcal{K}_\rho$ and $\rho_{j+1}=\rho_j$ for $j\notin\mathcal{K}_{\rho}$, and the last inequality is since $\Theta_{\rho_{j_{\widetilde{N}_{\epsilon}+1}}}(x^{j_{\widetilde{N}_{\epsilon}+1}})\ge \Theta(x^{j_{\widetilde{N}_{\epsilon}+1}})\ge \min_{x\in\varLambda_0}\Theta(x)$. By the Lipschitz continuity of ${\rm dist}(\cdot,K)$, we have
 \begin{align}\label{temp-ineq44}
 {\rm dist}(g(x^{j+1}),K)
 &\le\!c_u[\|g(x^{j+1})\!-\!z_2^{j+1}\!+\!g'(x^{j+1})v^{j+1}\|\!+\!\|g'(x^{j+1})v^{j+1}\|]\!+\!{\rm dist}(z_2^{j+1},K) \nonumber\\ 
 &\le c_u\max\{\varepsilon_{j+1},a_{j+1}\|v^{j+1}\|^2\}+c_u\|g'(x^{j+1})\|\|v^{j+1}\|+{\rm dist}(z_2^{j+1},K)\nonumber\\
 &\le c_u\max\{\varepsilon_{0},a_{\rm max}\gamma_{v}^2\}+c_u\gamma_{\nabla\!g}\gamma_{v}+\gamma_{K}=\gamma_{c}\quad{\rm for\ each}\ j\in\mathbb{N},
\end{align}
where the second inequality is due to \eqref{inexact-cond1} and \eqref{inexact-cond2}, and the third inequality follows from the definitions of $\gamma_{K},\gamma_{v}$ and $\gamma_{\nabla g}$. Substituting \eqref{temp-ineq44} into \eqref{temp-ineq43} and noting that $\sum_{j\in\mathcal{K}_{\rho}}\rho_{j}=\sum_{i=0}^{n_{\rho}-1}\rho_0\tau^i=\frac{\rho_0(\tau^{n_{\rho}}-1)}{\tau-1}$, the definition of $\mathcal{K}_{\rho}$ and the update rule of $\rho_k$ yield 
\begin{equation*}
 \begin{aligned}
 \sum_{j\in J}\|v^j\|^2&\le \frac{1}{\overline{\gamma}\eta_1}\Big[\Theta_{\rho_0}(x^0)-\min_{x\in\varLambda_0}\Theta(x)+\gamma_{c}(\tau\!-\!1)\sum_{j\in\mathcal{K}_{\rho}}\rho_j\Big]\\
 &\le (\overline{\gamma}\eta_1)^{-1}\big[\Theta_{\rho_{0}} (x^0)-\min_{x\in\varLambda_0}\Theta(x)+\rho_0\gamma_{c}(\tau^{n_{\rho}}\!-\!1)\big]=\mu_2.
 \end{aligned}
 \end{equation*}
This is a contradiction to the above \eqref{temp-ineq42}. Thus, we complete the proof.
\end{proof}

To establish the iteration complexity of Algorithm 1, we consider its finite-accuracy variant with the following $\epsilon$-dependent stopping criteria. In Step \ref{alg:step5}, we terminate when $\varepsilon_k\le\chi\epsilon$ and ${\rm dist}(z_2^k,K)\le c_{l}\epsilon/2$. In Step \ref{alg:step6}, we terminate when $\|v^k\|\le\epsilon/(2\mu_1)$ and ${\rm dist}(z^k_2,K)\le \max\big\{{c_{l}\epsilon}/{2},\max\{1,\gamma_{K}\overline{\rho}\}\overline{\beta}\|v^k\|\big\}$.
Here, $\chi$ is defined in Proposition \ref{prop1-complexity}, while $\gamma_K$ and $\mu_1$ are defined in Proposition \ref{prop2-complexity}. The following theorem follows directly from Propositions \ref{prop1-complexity}-\ref{prop2-complexity} and the construction of Algorithm \ref{PenAl}.  
\begin{theorem}\label{complexity-theorem}
Suppose that Assumption \ref{ass1} holds. Let $\epsilon\in(0,\min\{1,\varepsilon_0\})$. Then, Algorithm \ref{PenAl} generates an $\epsilon$-approximate KKT point within at most $T+(N\!+1)\big(\big\lceil\frac{4\mu_1^2\mu_2}{\epsilon^2}\big\rceil+\min\{N_{\epsilon},n_{\rho}\}+1\big)$ iterations, where $T$ and $N$ are given in Proposition \ref{prop1-complexity} and Lemma \ref{lemma1-complexity}(i), respectively, and $\mu_1$, $\mu_2$ and $N_{\epsilon}$ are given in Proposition \ref{prop2-complexity}.
\end{theorem}
\subsection{Overall oracle complexity}\label{sec4.2}

We establish overall oracle complexity bounds for Algorithm~\ref{PenAl} with $a_k\ge a_{\rm min},b_k\ge b_{\rm min}$ for all $k\in\mathbb{N}$, where $a_{\rm min},b_{\rm min}>0$ are constants. Our analysis accounts for the computational cost of solving the subproblems and considers two classes of inner solvers: primal-dual methods and dual methods. A key ingredient is the derivation of uniform bounds on the primal objective gaps and KKT residuals generated by the inner solvers across outer iterations. To this end, define 
\begin{equation}\label{KKT-res}
\mathcal{R}_k(w):=\begin{pmatrix}
 v-\mathcal{P}_{T_{\!x^k}\mathcal{M}}\big[v-\nabla\ell_k(v)-\beta_kv-\mathcal{A}_k^*\lambda\big]\\
 z-\mathcal{P}_{\psi_{\rho_k}}(z+\lambda)\\
 z-G(x^k,v)
\end{pmatrix}
\end{equation} 
for $w=(v,z,\lambda)\in\mathbb{X}\times\mathbb{Z}\times\mathbb{Z}$, where $\mathcal{A}_k$ denotes the Jacobian of $G(x^k,\cdot)$. 
By definition, the zeros of $\mathcal{R}_k$ coincide with
the KKT points of subproblem \eqref{Esubprob}. We also introduce the following family of decay functions:
\[
\mathscr{D}:=\Big\{\phi:\mathbb{R}_{+}\to(0,1]\mid\phi\  \textrm{is continuous and strictly decreasing},\,\phi(0)=1,\,\lim_{t\to\infty}\phi(t)=0\Big\}.
\]
\subsubsection{Primal-dual inner solvers}\label{sec4.2.1}

The following theorem shows how uniform decay estimates for the primal objective gaps and KKT residuals translate into an overall complexity bound for Algorithm~\ref{PenAl}. 
\begin{theorem}\label{complexity-oracle}
 For each $k\in\mathbb{N}$, let $\{w^{k,l}\}_{l\in\mathbb{N}_{+}}$ be a sequence generated by a primal-dual method for solving subproblem \eqref{Esubprob}, where $w^{k,l}:=(v^{k,l},z^{k,l},\lambda^{k,l})\in T_{\!x^k}\mathcal{M}\times\mathbb{Z}\times\partial\psi_{\rho_k}(z^{k,l})$. Assume that the method has access to oracles for evaluating $g'(x^k)(\cdot),\nabla g(x^k)(\cdot),\mathcal{P}_{T_{\!x^k}\mathcal{M}}(\cdot)$ and $\mathcal{P}_{\psi_{\rho_k}}(\cdot)$. Suppose that there exist constants $c_{\rm gap},c_{\mathcal{R}}>0$ and $\phi_{\rm gap},\phi_{\mathcal{R}}\in\mathscr{D}$ such that, for each $k\in\mathbb{N}$ and $l\in\mathbb{N}_{+}$,
 \begin{equation}\label{overall-cond}
 \widehat{\Theta}_{\rho_k,\beta_k}(x^k+v^{k,l};x^k)-\widehat{\Theta}_{\rho_k,\beta_k}(x^k+\overline{v}^k;x^k)\le c_{\rm gap}\phi_{\rm gap}(l)\ \ {\rm and}\ \ \|\mathcal{R}_k(w^{k,l})\|\le c_{\mathcal{R}}\phi_{\mathcal{R}}(l).
\end{equation}
 Then, under Assumption \ref{ass1}, for any $\epsilon\in(0,\min\{1,\varepsilon_0\})$, Algorithm \ref{PenAl} generates an $\epsilon$-approximate KKT point within at most $O(\max\{l_\epsilon,1\}\epsilon^{-2})$ inner iterations in total, where 
 \[
  l_{\epsilon}:=\max\bigg\{\Big\lceil\phi_{\rm gap}^{-1}\Big(\min\!\Big\{1,\frac{\beta_{\min}(\widehat{\chi}\epsilon)^2}{32c_{\rm gap}\overline{\beta}^2}\Big\}\Big)\Big\rceil,\Big\lceil \phi_{\mathcal{R}}^{-1}\Big(\min\!\Big\{\frac{\widehat{\chi}\epsilon}{c_{\mathcal{R}}},\frac{a_{\min}(\widehat{\chi}\epsilon)^2}{4c_{\mathcal{R}}\overline{\beta}^2},\frac{b_{\min}\widehat{\chi}\epsilon}{2c_{\mathcal{R}}\overline{\beta}},1\Big\}\Big)\Big\rceil\bigg\}.
 \]
 In particular, if each inner iteration requires only a uniformly bounded number of oracle evaluations, the same order holds for the overall oracle complexity.
\end{theorem}
\begin{proof}
 Let $k_{\epsilon}$ be the index at which Algorithm \ref{PenAl} first returns an $\epsilon$-stationary point. By Theorem \ref{complexity-theorem},  $k_{\epsilon}$ is bounded by $O(\epsilon^{-2})$. To establish the conclusions, it only needs to show that, for each $k\in[k_{\epsilon}]$, the inner solver requires at most $l_{\epsilon}$ iterations to generate a triple $(v^{k},z^{k},\xi^{k})\in T_{\!x^k}\mathcal{M}\times\mathbb{Z}\times\partial\psi_{\rho_k}(z^{k})$ satisfying the inexactness condition \eqref{inexact-cond1} or the alternative conditions \eqref{model-descent}-\eqref{inexact-cond2}. 
 
 By Proposition \ref{prop1-complexity}, Step~\ref{alg:step5} is executed at most $T$ times up to iteration $k_{\epsilon}$. Then, for each $k\in[k_{\epsilon}]$, recalling that the parameters $\varepsilon_k$ are updated only in Step~\ref{alg:step5} and using  $T\le n_{\rho}+\Big\lceil \frac{\log(\chi^{-1}\varepsilon_0\epsilon^{-1})}{\log\varsigma^{-1}}\Big\rceil+1$, we have
 \begin{equation}\label{epsk-lbound}
 \varepsilon_k\ge \varepsilon_{k_\epsilon}\ge \varepsilon_0\varsigma^T\ge \widehat{\chi}\epsilon\ \ {\rm with}\ \widehat{\chi}:=\chi \varsigma^{n_{\rho}+2}.
 \end{equation}
 Fix any $k\in[k_{\epsilon}]$. The $\beta_k$-strong convexity of $\widehat{\Theta}_{\rho_k,\beta_k}(x^k\!+\!\cdot;x^k)$, together with the first inequality in \eqref{overall-cond} and Lemma \ref{bound-lemma}, implies the boundedness of $\{v^{k,l}\}_{l\in\mathbb{N}}$. Combining the definition of $\mathcal{R}_k$ with the second inequality in \eqref{overall-cond} yields the boundedness of $\{z^{k,l}\}_{l\in\mathbb{N}}$. Together with $\lambda^{k,l}\in\partial\psi_{\rho_k}(z^{k,l})$ for each $l\in\mathbb{N}_{+}$, it follows from Assumption \ref{ass1} and Rockafellar and Wets \cite[Theorem 9.13 \& Proposition 5.15]{RW98} that $\{\lambda^{k,l}\}_{l\in\mathbb{N}_{+}}$ is bounded. Thus, the sequence $\{w^{k,l}\}_{l\in\mathbb{N}_{+}}$ is bounded. Let $w^{k,*}=(v^{k,*},z^{k,*},\lambda^{k,*})$ be any cluster point of $\{w^{k,l}\}_{l\in\mathbb{N}_{+}}$. From the second inequality in \eqref{overall-cond} and $\lim_{l\to\infty}\phi_{\mathcal{R}}(l)=0$, we obtain $\mathcal{R}_k(w^{k,*})=0$. Consequently, $\mathcal{R}_{k}^{-1}(0)$ is nonempty. Moreover, the strong convexity of the objective of \eqref{Esubprob} with respect to $v$ implies that 
 \begin{equation}\label{Rk-zero}
 \emptyset\ne\mathcal{R}_{k}^{-1}(0)\subset\{\overline{v}^k\}\times\{G(x^k,\overline{v}^k)\}\times\partial \psi_{\rho_k}(G(x^k,\overline{v}^k)). 
 \end{equation}
 Clearly, $v^{k,*}=\overline{v}^k$. By the $\beta_k$-strong convexity of $\widehat{\Theta}_{\rho_k,\beta_k}(x^k\!+\!\cdot;x^k)$ and the first inequality of \eqref{overall-cond}, we have 
 \begin{equation*}
 ({\beta_{\min}}/{2})\|v^{k,l}-v^{k,*}\|^2\le\widehat{\Theta}_{\rho_k,\beta_k}(x^k+v^{k,l};x^k)-\widehat{\Theta}_{\rho_k,\beta_k}(x^k+v^{k,*};x^k)\le c_{\rm gap}\phi_{\rm gap}(l)\quad\forall l\in\mathbb{N}.
 \end{equation*}
 Together with the strict monotonicity of $\phi_{\rm gap}$, for each $l\ge\overline{l}_{\epsilon}:=\!\!\Big\lceil\phi_{\rm gap}^{-1}\Big(\!\min\Big\{\frac{\beta_{\min}(\widehat{\chi}\epsilon)^2}{32c_{\rm gap}\overline{\beta}^2},1\Big\}\Big)\Big\rceil$,
 \begin{equation}\label{vk-vstar-bound}
 \|v^{k,l}-v^{k,*}\|\le\frac{\widehat{\chi}\epsilon}{4\overline{\beta}}\ \ {\rm and}\ \
 \widehat{\Theta}_{\rho_k,\beta_k}(x^k+v^{k,*};x^k)\ge   \widehat{\Theta}_{\rho_k,\beta_k}(x^k+v^{k,l};x^k)-\frac{\beta_{\min}(\widehat{\chi}\epsilon)^2}{32\overline{\beta}^2}.
\end{equation}
On the other hand, the strict monotonicity of $\phi_{\mathcal{R}}$ and the second inequality of \eqref{overall-cond} imply 
 \begin{equation}\label{bound-chik1}
 \|\mathcal{R}_k(w^{k,l})\|\le \min\bigg\{\widehat{\chi}\epsilon,\frac{a_{\min}(\widehat{\chi}\epsilon)^2}{4\overline{\beta}^2}, \frac{b_{\min}\widehat{\chi}\epsilon}{2\overline{\beta}}\bigg\}
\end{equation}
for every $l\ge \widetilde{l}_{\epsilon}:=\Big\lceil\phi_{\mathcal{R}}^{-1}\Big(\min\Big\{\frac{\widehat{\chi}\epsilon}{c_{\mathcal{R}}},\frac{a_{\min}(\widehat{\chi}\epsilon)^2}{4c_{\mathcal{R}}\overline{\beta}^2},\frac{b_{\min}\widehat{\chi}\epsilon}{2c_{\mathcal{R}}\overline{\beta}},1\Big\}\Big)\Big\rceil$. We claim that the triple $(v^{k,l},z^{k,l},\lambda^{k,l})\in T_{\!x^k}\mathcal{M}\times\mathbb{Z}\times\partial\psi_{\rho_k}(z^{k,l})$ for every $l\ge\max\{\overline{l}_{\epsilon},\widetilde{l}_{\epsilon}\}$  satisfies the criterion \eqref{inexact-cond1} or the conditions \eqref{model-descent}-\eqref{inexact-cond2} by two cases. 
	
\noindent 
\textbf{Case 1:} $\|v^{k,*}\|\le\frac{3\widehat{\chi}\epsilon}{4\overline{\beta}}$. For each $l\ge\widetilde{l}_{\epsilon}$, the first inequality in \eqref{vk-vstar-bound} implies that 	$\|v^{k,l}\|\le\|v^{k,*}\|+\frac{\widehat{\chi}\epsilon}{4\overline{\beta}} \le \frac{\widehat{\chi}\epsilon}{\overline{\beta}}$, so that $\|\beta_kv^{k,l}\|\le \widehat{\chi}\epsilon\le\varepsilon_k$, where the second inequality is due to \eqref{epsk-lbound}. In addition, for each $l\ge\max\{\overline{l}_{\epsilon},\widetilde{l}_{\epsilon}\}$, $\|\mathcal{R}_k(w^{k,l})\|\le\widehat{\chi}\epsilon\le\varepsilon_k$ by \eqref{bound-chik1} and \eqref{epsk-lbound}. Thus, $(v^{k,l},z^{k,l},\lambda^{k,l})$ for each $l\ge\max\{\overline{l}_{\epsilon},\widetilde{l}_{\epsilon}\}$ satisfies \eqref{inexact-cond1}. 
	
\noindent 
\textbf{Case 2:} $\|v^{k,*}\|>\frac{3\widehat{\chi}\epsilon}{4\overline{\beta}}$. The $\beta_k$-strong convexity of $\widehat{\Theta}_{\rho_k,\beta_k}(x^k\!+\!\cdot;x^k)+\delta_{T_{\!x^k}\mathcal{M}}(\cdot)$, together with the second inequality in \eqref{vk-vstar-bound}, yields that, for each $l\ge\overline{l}_{\epsilon}$, 
\begin{equation*}
\begin{aligned}
\Theta_{\rho_k}(x^k)&=\widehat{\Theta}_{\rho_k,\beta_k}(x^k;x^k)\ge \widehat{\Theta}_{\rho_k,\beta_k}(x^k\!+\!v^{k,*};x^k)+({\beta_{\rm min}}/{2})\|v^{k,*}\|^2\\
&>\widehat{\Theta}_{\rho_k,\beta_k}(x^k\!+\!v^{k,*};x^k)+ \frac{9\beta_{\rm min}(\widehat{\chi}\epsilon)^2}{32\overline{\beta}^2}>\widehat{\Theta}_{\rho_k,\beta_k}(x^k\!+\!v^{k,l};x^k).
\end{aligned}
\end{equation*}
Moreover, the first inequality in \eqref{vk-vstar-bound} implies that $\|v^{k,l}\|\ge \frac{\widehat{\chi}\epsilon}{2\overline{\beta}}$ for every $l\ge\max\{\overline{l}_{\epsilon},\widetilde{l}_{\epsilon}\}$. Together with \eqref{bound-chik1}, $a_k\in[a_{\rm min},a_{\rm max}]$ and $b_k\in[b_{\rm min},b_{\rm max}]$, it follows that
$\|\mathcal{R}_k(w^{k,l})\|\le \min\big\{a_k\|v^{k,l}\|^2,b_k\|v^{k,l}\|\big\}$ for every $l\ge\max\{\overline{l}_{\epsilon},\widetilde{l}_{\epsilon}\}$. Consequently, the triple $(v^{k,l},z^{k,l},\lambda^{k,l})$ for every $l\ge\max\{\overline{l}_{\epsilon},\widetilde{l}_{\epsilon}\}$ satisfies \eqref{model-descent}-\eqref{inexact-cond2}. 
	
The desired conclusion then follows from the definitions of $\overline{l}_{\epsilon}$ and $\widetilde{l}_{\epsilon}$.
\end{proof}
\begin{corollary}\label{aADMM-oracle}
Let $\overline{\sigma}\ge1$ and   $B_0\subset\mathbb{X}\times\mathbb{Z}\times\mathbb{Z}$ be a compact set. Suppose that Assumption \ref{ass1} holds. When Algorithm \ref{aADMM} with $\sigma_k\in [\overline{\sigma}^{-1},\overline{\sigma}]$ and $\overline{w}^{k,0}\in B_0$ is used as the inner solver, for any $\epsilon\in(0,\min\{1,\varepsilon_0\})$, Algorithm \ref{PenAl} computes an $\epsilon$-approximate KKT point within at most $O(\epsilon^{-4})$ oracle calls.
\end{corollary}
\begin{proof}
 For each $k\in\mathbb{N}$, let $\{w^{k,l}\}_{l\in\mathbb{N}}\subset T_{x^k}\mathcal{M}\times\mathbb{Z}\times\mathbb{Z}$ be the sequence generated by Algorithm \ref{aADMM} with $\sigma_k\in [\overline{\sigma}^{-1},\overline{\sigma}]$ and $\overline{w}^{k,0}\in B_0$. Then $w^{k,l}=(v^{k,l},z^{k,l},\lambda^{k,l})\in T_{\!x^k}\mathcal{M}\times\mathbb{Z}\times\partial\psi_{\rho_k}(z^{k,l})$ for every $l\in\mathbb{N}_{+}$. By \cite[Corollary 3.5]{Sun2025}, the sequence $\{w^{k,l}\}_{l\in\mathbb{N}}$ converges to a point $w^{k,*}\in\mathcal{R}_k^{-1}(0)$, and consequently, for each $k\in\mathbb{N}$,
 \begin{equation}\label{Rk-ineq41}
 \mathcal{P}_{T_{x^k}\mathcal{M}}(\nabla\ell_k(0)+\beta_kv^{k,*}+\mathcal{A}_k^*\lambda^{k,*})=0\ \ {\rm and}\ \ \lambda^{k,*}\in \partial\psi_{\rho_k}(z^{k,*}).
 \end{equation}
 Moreover, due to \eqref{Rk-zero}, $v^{k,*}=\overline{v}^k$ and $z^{k,*}=G(x^k,\overline{v}^{k})$. For each $k\in\mathbb{N}$, define $R_{k,0}:=\|\overline{w}^{k,0}\!-\!w^{k,*}\|_{\mathcal{E}_k}$ with $\mathcal{E}_k:=[\sigma_k\|\mathcal{A}_k\|^2\mathcal{I}\ \ 0 \ \  \mathcal{A}_k^*;0\ \ 0 \ \ 0; \mathcal{A}_k\ \ 0\ \ \sigma_k^{-1}\mathcal{I}]$, where $\|u\|_{\mathcal Q}:=\sqrt{\langle u,\mathcal Q u\rangle}$ for any self-adjoint positive semidefinite linear operator $\mathcal Q$. For Algorithm \ref{aADMM} with $\alpha=2$, by \cite[Theorem 3.7(a)]{Sun2025}, we have 
 \begin{equation}\label{Rk-obj-ubound}
 \|\mathcal{R}_k(w^{k,l})\|\le \frac{\widehat{\varpi}_k}{l+1}
 \ \ {\rm and}\ \ F_k(v^{k,l},z^{k,l})-F_k(v^{k,*},z^{k,*})\le \frac{\widetilde{\varpi}_k}{l+1}\quad{\rm for\ all}\ k,l\in\mathbb{N},
 \end{equation}
 where $\widehat{\varpi}_k:=\frac{2R_{k,0}(\sigma_k+1)}{\varrho\sqrt{\sigma_k}}$ and $\widetilde{\varpi}_k:=\frac{2R_{k,0}(3R_{k,0}\sqrt{\sigma_k}+\|\lambda^{k,*}\|)}{\varrho\sqrt{\sigma_k}}$, and $F_k(v,z)$ denotes the objective function of subproblem \eqref{Esubprob}. For Algorithm \ref{aADMM} with $\alpha>2$, the two inequalities in \eqref{Rk-obj-ubound} continue to hold for all $k,l\in\mathbb{N}$ with $\widehat{\varpi}_k$ and $\widetilde{\varpi}_k$ replaced by $\frac{c(\alpha)(\sigma_k+1)}{\varrho\sqrt{\sigma_k}}$ and $\frac{c(\alpha)(3R_{k,0}\sqrt{\sigma_k}+\|\lambda^{k,*}\|)}{\varrho\sqrt{\sigma_k}}$, respectively, where $c(\alpha)$ is a constant depending only on $\alpha$. This follows by  combining Sun et al. \cite[Proposition 2.9(b)]{Sun2025} with the detailed proof of Bo\c{t} and Nguyen \cite[Theorem 3.5]{Bot2023} and adapting the proof of Theorem~3.7(a) in Sun et al. \cite[Theorem 3.7(a)]{Sun2025}. Since $\{\mathcal{A}_k\}_{k\in \mathbb{N}}$ is bounded and $\sigma_k\in[\overline{\sigma}^{-1},\overline{\sigma}]$ for all $k\in\mathbb{N}$, the sequence $\{\mathcal{E}_k\}_{k\in\mathbb{N}}$ is bounded. While the boundedness of $\{w^{k,*}\}_{k\in\mathbb{N}}$ is implied by the inclusion \eqref{Rk-zero}. Together with the boundedness of $\{\overline{w}^{k,0}\}_{k\in\mathbb{N}}$, there exists a constant $R_0>0$ such that $R_{k,0}\le R_0$ for all $k\in\mathbb{N}$. Then, $\widehat{\varpi}_k\le\widehat{\varpi}$ with $\widehat{\varpi}:=\frac{2R_0(\overline{\sigma}+1)}{\varrho\sqrt{\overline{\sigma}}}$ for $\alpha=2$ and $\widehat{\varpi}:=\frac{c(\alpha)(\overline{\sigma}+1)}{\varrho\sqrt{\overline{\sigma}}}$ for $\alpha>2$. Combining this with the first inequality in \eqref{Rk-obj-ubound} shows that the second inequality of \eqref{overall-cond} holds with $c_{\mathcal{R}}=\widehat{\varpi}$ and $\phi_{\mathcal{R}}(t)=\frac{1}{t+1}$ for $t\ge 0$. 
	
 To show that the first inequality of \eqref{overall-cond} holds for the iterates $w^{k,l}$, we prove that $\{v^{k,l}\}_{k,l\in\mathbb{N}}$ is bounded. From the $\beta_k$-strong convexity of $\ell_k(\cdot)+\frac{\beta_k}{2}\|\cdot\|^2$, for each $k,l\in\mathbb{N}$, 
\begin{equation*}
 \begin{aligned}
 &\ell_k(v^{k,l})+\frac{\beta_k}{2}\|v^{k,l}\|^2-\ell_k(v^{k,*})-\frac{\beta_k}{2}\|v^{k,*}\|^2\\
 &\ge\langle \nabla \ell_k(v^{k,*})+\beta_kv^{k,*},v^{k,l}-v^{k,*}\rangle+\frac{\beta_k}{2}\|v^{k,l}-v^{k,*}\|^2\\
 &=\langle\mathcal{P}_{T_{x^k}\mathcal{M}}(\nabla \ell_k(0)+\beta_kv^{k,*}),v^{k,l}-v^{k,*}\rangle+\frac{\beta_k}{2}\|v^{k,l}-v^{k,*}\|^2\\
 &=\langle-\mathcal{A}_k^*\lambda^{k,*},v^{k,l}-v^{k,*}\rangle+\frac{\beta_k}{2}\|v^{k,l}-v^{k,*}\|^2\\
 &=\langle \lambda^{k,*},z^{k,*}-z^{k,l}\rangle+\langle \lambda^{k,*},z^{k,l}-G(x^k,v^{k,l})\rangle+\frac{\beta_k}{2}\|v^{k,l}-v^{k,*}\|^2,
 \end{aligned}
 \end{equation*}
 where the first equality is due to $v^{k,l}-v^{k,*}\in T_{x^k}\mathcal{M}$, the first one is by \eqref{Rk-ineq41}, and the third one is obtained by using $z^{k,l}-z^{k,*}-\!\mathcal{A}_k(v^{k,l}\!-v^{k,*})=z^{k,l}-G(x^k,v^{k,l})$. From the inclusion in \eqref{Rk-ineq41}, 	$\psi_{\rho_k}(z^{k,l})-\psi_{\rho_k}(z^{k,*})\ge \langle\lambda^{k,*},z^{k,l}-z^{k,*}\rangle$ for all $k,l\in\mathbb{N}$. Together with the above inequality and the definition of $F_k(v,z)$, we have 
 \begin{equation*}
 \frac{\beta_k}{2}\|v^{k,l}-v^{k,*}\|^2\le F_k(v^{k,l},z^{k,l})-F_k(v^{k,*},z^{k,*})+\|\lambda^{k,*}\|\|z^{k,l}-G(x^k,v^{k,l})\|\quad{\rm for\ all}\ k,l\in\mathbb{N}.
 \end{equation*}
 Combining this inequality with the two estimates in \eqref{Rk-obj-ubound}, we obtain, for each $k,l\in\mathbb{N}$,
 \begin{equation*}
 \|v^{k,l}\!-v^{k,*}\|^2\le \frac{2(\|\lambda^{k,*}\|\widehat{\varpi}_k+\widetilde{\varpi}_k)}{\beta_k(l+1)}\le \frac{2(\|\lambda^{k,*}\|\widehat{\varpi}_k+\widetilde{\varpi}_k)}{\beta_{\min}(l+1)}.
\end{equation*}
Recall that $\{w^{k,*}\}_{k\in\mathbb{N}}$ is bounded, so there exists $c_{\lambda}>0$ such that $\|\lambda^{k,*}\|\le c_{\lambda}$ for all $k\in\mathbb{N}$. Together with $R_{k,0}\le R_0$ for all $k\in\mathbb{N}$, we have $\widetilde{\varpi}_k\le\widetilde{\varpi}$ with $\widetilde{\varpi}:= \frac{2R_0(3R_0\sqrt{\overline{\sigma}}+c_{\lambda})}{\varrho\sqrt{\overline{\sigma}^{-1}}}$ for $\alpha=2$ and $\widetilde{\varpi}:= \frac{c(\alpha)(3R_0\sqrt{\overline{\sigma}}+c_{\lambda})}{\varrho\sqrt{\overline{\sigma}^{-1}}}$ for $\alpha>2$. Combining with $\widehat{\varpi}_k\le\widehat{\varpi}$ for all $k\in\mathbb{N}$, the above inequality shows that $\{v^{k,l}\}_{k,l\in\mathbb{N}}$ is bounded. Together with the boundedness of $\{x^k\}_{k\in\mathbb{N}}$, it follows that $\{G(x^k,v^{k,l})\}_{k,l\in\mathbb{N}}$ is bounded. Note that 
\begin{equation}\label{zk-GK-bound} \|z^{k,l}-G(x^k,v^{k,l})\|\le\|\mathcal{R}_k(w^{k,l})\|\le\widehat{\varpi}/(l+1) \quad{\rm for\ all}\ k,l\in\mathbb{N}.
\end{equation}
 It then follows that $\{z^{k,l}\}_{k,l\in\mathbb{N}}$ is bounded. Thus, there exists a compact convex set $\mathcal{X}$ containing the sequences $\{x^k+v^{k,l}\}_{k,l\in\mathbb{N}}$, $\{z^{k,l}_1\}_{k,l\in\mathbb{N}}$ and $\{x^k+v^{k,*}\}_{k\in\mathbb{N}}$. Let $L_{\vartheta,\mathcal{X}}$ denote the Lipschitz constant of $\vartheta$ on $\mathcal{X}$. The relationship between $\widehat{\Theta}_{\rho_k,\beta_k}(x^k+v;x^k)$ and $F_k(v,z)$, together with Lemma \ref{lemma-psirhok}, implies that 
\begin{equation*}
 \begin{aligned}
 &\widehat{\Theta}_{\rho_k,\beta_k}(x^k+v^{k,l};x^k)-\widehat{\Theta}_{\rho_k,\beta_k}(x^k+v^{k,*};x^k)\\
 &\le  F_k(v^{k,l},z^{k,l})-F_k(v^{k,*},z^{k,*})+(L_{\vartheta,\mathcal{X}}^2+c_u^2\overline{\rho}^2)^{1/2}\|z^{k,l}\!-\!G(x^k,v^{k,l})\|\\
 &\overset{\eqref{Rk-obj-ubound}\&\eqref{zk-GK-bound}}{\le} \big[\widetilde{\varpi}+(L_{\vartheta,\mathcal{X}}^2+c_u^2\overline{\rho}^2)^{1/2}\widehat{\varpi}\big]/(l+1)\quad{\rm for\ all}\ k,l\in\mathbb{N}.
 \end{aligned}
\end{equation*}
This shows that the first inequality of \eqref{overall-cond} holds for the iterates $w^{k,l}$ with $\phi_{\rm gap}(t)=\frac{1}{t+1}$ for $t\ge 0$. The conclusion follows from Theorem \ref{complexity-oracle} and the iterations of Algorithm \ref{aADMM}.
\end{proof}
\begin{corollary}\label{sPADMM-oracle}
Suppose that Algorithm \ref{sPADMM} is used as the inner solver with $\sigma_k\in [\overline{\sigma}^{-1},\overline{\sigma}]$ and $\overline{w}^{k,0}\in B_0$ for all $k\in\mathbb{N}$, where  $\overline{\sigma}$ and $B_0$ are as specified in Corollary \ref{aADMM-oracle}. Suppose further that there exists a constant $c_{\rm err}>0$ such that ${\rm dist}(w^{k,l},\mathcal R_k^{-1}(0))\le c_{\rm err}\|\mathcal R_k(w^{k,l})\|$ for all $k,l\in\mathbb{N}$, where $w^{k,l}$ denotes the $l$-th iterate of Algorithm \ref{sPADMM} at outer iteration $k$. Then, under Assumption \ref{ass1}, for any $\epsilon\in(0,\min\{1,\varepsilon_0\})$, Algorithm \ref{PenAl} generates an $\epsilon$-approximate KKT point within at most $O(\epsilon^{-2}\log(1/\epsilon))$ oracle calls.
\end{corollary}
\begin{proof}
 For each $k\in\mathbb{N}$, $w^{k,l}=(v^{k,l},z^{k,l},\lambda^{k,l})\in T_{\!x^k}\mathcal{M}\times\mathbb{Z}\times\partial\psi_{\rho_k}(z^{k,l})$ for every $l\in\mathbb{N}_{+}$, and furthermore, by Fazel et al. \cite[Appendix B]{Fazel2013}, $\{w^{k,l}\}_{l\in\mathbb{N}}$ converges to a point $w^{k,*}\in\mathcal{R}_k^{-1}(0)$. As shown in the proof of Corollary \ref{aADMM-oracle}, we have $v^{k,*}=\overline{v}^k$ and $z^{k,*}=G(x^k,\overline{v}^{k})$. Since \eqref{Esubprob} is a special case of \cite[Problem (1)]{Han2018} with $\mathscr{A}^*=\mathcal{A}_k$ and $\mathscr{B}^*=-\mathcal{I}$, the given error-bound condition, together with \cite[Corollary 1]{Han2018}, yields 
 \begin{equation*}
 {\rm dist}_{\mathcal{E}_k}(w^{k,l+1},\mathcal{R}_{k}^{-1}(0))\le\Big(\frac{2c_{\rm err}^2\alpha_k\|\mathcal{E}_k\|+1}{2c_{\rm err}^2\alpha_k\|\mathcal{E}_k\|+2}\Big)^{1/2}{\rm dist}_{\mathcal{E}_k}(w^{k,l},\mathcal{R}_{k}^{-1}(0))\quad\forall l\in\mathbb{N},
 \end{equation*}
 where $\mathcal{E}_k:={\rm Diag}(\sigma_k(\mathcal{A}_k\|^2\mathcal{I}\!-\!\mathcal{A}_k^*\mathcal{A}_k)+\beta_k\mathcal{I},\sigma_{k}\mathcal{I},\sigma_k^{-1}\mathcal{I})+\frac{\sigma_{k}}{4}\mathcal{B}_k\mathcal{B}_k^*$ with $\mathcal{B}_k^*(w):=\mathcal{A}_kv-z$ for $w=(v,z,\lambda)$,  $\alpha_k:=\max\{3\sigma_k\|\|\mathcal{A}_k\|^2\mathcal{I}\!-\!\mathcal{A}_k^*\mathcal{A}_k\|,3\sigma_k\|\mathcal{A}_k\|^2,\sigma_k^{-1}\}$, and ${\rm dist}_{\mathcal{E}_k}(\cdot,\mathcal{R}_{k}^{-1}(0))$ is the distance induced by $\|\cdot\|_{\mathcal{E}_k}$. From $\sigma_k\in [\overline{\sigma}^{-1},\overline{\sigma}]$ and Assumption \ref{ass1}(i), $\{\alpha_k\}_{k\in\mathbb{N}}$ and $\{\|\mathcal{E}_k\|\}_{k\in\mathbb{N}}$ are bounded, so there exists $\overline{\alpha}>0$ such that $\alpha_k\|\mathcal{E}_k\|\le \overline{\alpha}$ for all $k\in\mathbb{N}$. Then $\Big(\frac{2c_{\rm err}^2\alpha_k\|\mathcal{E}_k\|+1}{2c_{\rm err}^2\alpha_k\|\mathcal{E}_k\|+2}\Big)^{1/2}\le \Big(\frac{2c_{\rm err}^2\overline{\alpha}+1}{2c_{\rm err}^2\overline{\alpha}+2}\Big)^{1/2}:=\alpha_{r}\in(0,1)$. Thus, 
 \begin{equation}\label{linear-rate}
 {\rm dist}_{\mathcal{E}_k}(w^{k,l+1},\mathcal{R}_{k}^{-1}(0))\le \alpha_r{\rm dist}_{\mathcal{E}_k}(w^{k,l},\mathcal{R}_{k}^{-1}(0))\quad{\rm for\ every}\ k,l\in\mathbb{N}.
 \end{equation}
 Fix any $k\in\mathbb{N}$. Let $L_{\mathcal{R}_k}$ be the Lipschitz constant of $\mathcal{R}_k$. For any $l\in\mathbb{N}$ and $w\in\mathcal{R}_{k}^{-1}(0)$, we have 
 \begin{equation*}
 \|\mathcal{R}_k(w^{k,l})\|=\|\mathcal{R}_k(w^{k,l})-\mathcal{R}_k(w)\|\le L_{\mathcal{R}_k}\|w^{k,l}\!-\!w\|\le\frac{L_{\mathcal{R}_k}\|w^{k,l}-w\|_{\mathcal{E}_k}}{\sqrt{\min\{\beta_{\min},\sigma_k,\sigma_k^{-1}\}}}.
 \end{equation*}
 Taking the infimum over $w\in\mathcal{R}_{k}^{-1}(0)$ and then applying \eqref{linear-rate} yields, for all $l\in\mathbb{N}$, 
 \begin{equation*}
 \|\mathcal{R}_k(w^{k,l})\|\le\!\frac{L_{\mathcal{R}_k}{\rm dist}_{\mathcal{E}_k}(w^{k,l},\mathcal{R}_{k}^{-1}(0))}{\sqrt{\min\{\beta_{\min},\sigma_k,\sigma_k^{-1}\}}}\le L_{\mathcal{R}_k}\varpi_k\alpha_r^{l}\ \ {\rm with}\  \varpi_k\!:=\!\frac{{\rm dist}_{\mathcal{E}_k}(w^{k,0},\mathcal{R}_{k}^{-1}(0))}{\sqrt{\min\{\beta_{\min},\sigma_k,\sigma_k^{-1}\}}}.
\end{equation*}
 Combining the definition of $\|\cdot\|_{\mathcal{E}_k}$ with $v^{k,*}=\overline{v}^k$ and using \eqref{Rk-zero} yield, for each $l\in\mathbb{N}$,
 \begin{equation}\label{vkl-ubound}
  \|v^{k,l}\!-v^{k,*}\|\le\!\frac{\inf_{w\in\mathcal{R}_k^{-1}(0)}\|w^{k,l}-w\|_{\mathcal{E}_k}}{\sqrt{\min\{\beta_{\min},\sigma_k,\sigma_k^{-1}\}}}=\frac{{\rm dist}_{\mathcal{E}_k}(w^{k,l},\mathcal{R}_k^{-1}(0))}{\sqrt{\min\{\beta_{\min},\sigma_k,\sigma_k^{-1}\}}}\overset{\eqref{linear-rate}}{\le}\varpi_k\alpha_r^{l}.
 \end{equation}
 Since $\bigcup_{k\in\mathbb{N}}\mathcal{R}_{k}^{-1}(0)$ and $\{\mathcal{E}_k\}_{k\in\mathbb{N}}$ are bounded and $\{w^{k,0}\}_{k\in\mathbb{N}}\subset B_0$, there exists a constant $d_{\mathcal{E}}>0$ such that ${\rm dist}_{\mathcal{E}_k}(w^{k,0},\mathcal{R}_{k}^{-1}(0))\le d_{\mathcal{E}}$ for all $k\in\mathbb{N}$. The definition of $\varpi_k$ then yields $\varpi_k\le\varpi:=\frac{d_{\mathcal{E}}}{\sqrt{\min\{\beta_{\min},\,\overline{\sigma}^{-1},\,\overline{\sigma}\}}}$. Moreover, by the definition of $\mathcal{R}_k$,  $L_{\mathcal{R}_k}\le\sqrt{(\beta_k+\|\mathcal{A}_k\|)^2+(1+\|\mathcal{A}_k\|)^2}\le L_{\mathcal{R}}:=\sqrt{(\overline{\beta}+c_{\mathcal{A}})^2+(1+c_{\mathcal{A}})^2}$ with $c_{\mathcal{A}}:=\sup_{k\in\mathbb{N}}\|\mathcal{A}_k\|$. Thus, $\|\mathcal{R}_k(w^{k,l})\|\le L_{\mathcal{R}}\varpi \alpha_r^{l}$, and hence $\|v^{k,l}-\overline{v}^k\|=\|v^{k,l}-v^{k,*}\|\le\varpi$. The latter inequality and the boundedness of $\{\overline{v}^k\}_{k\in\mathbb{N}}$ imply that $\{v^{k,l}\}_{k,l\in\mathbb{N}}$ is bounded. Then there exists a compact convex set $\mathcal{X}$ containing  $\{\overline{v}^k\}_{k\in\mathbb{N}}$ and $\{v^{k,l}\}_{k,l\in\mathbb{N}}$. Let $\widehat{L}_{\vartheta,\mathcal{X}}$ denote the Lipschitz constant of $\vartheta$ on $\mathcal{X}$, and define $\widehat{L}:=L_{\nabla\!f}+L_{h}+\widehat{L}_{\vartheta,\mathcal{X}}+c_u\overline{\rho}\gamma_{\nabla g}+\overline{\beta}c_{v}$ with $c_{v}:=\sup_{v\in\mathcal{X}}\|v\|$, where $L_{\nabla\!f}$ and $L_h$ are defined in Lemma \ref{lemma1-complexity}(i), and $\gamma_{\nabla g}$ is defined in Proposition \ref{prop2-complexity}. For each $k\in \mathbb{N}$, the expression of $\widehat{\Theta}_{\rho_k,\beta_k}(x^k+\cdot;x^k)$ implies that it is $\widehat{L}$-Lipschitz continuous on $\mathcal{X}$. Consequently, for all $k,l\in\mathbb{N}$, 
 \begin{equation*}
 \big|\widehat{\Theta}_{\rho_k,\beta_k}(x^k+v^{k,l};x^k)-\widehat{\Theta}_{\rho_k,\beta_k}(x^k+\overline{v}^{k};x^k)\big|\le \widehat{L}\|v^{k,l}-\overline{v}^{k}\|.
 \end{equation*}
 Combining with \eqref{vkl-ubound}, we readily obtain 
 $\widehat{\Theta}_{\rho_k,\beta_k}(x^k\!+\!v^{k,l};x^k)-\widehat{\Theta}_{\rho_k,\beta_k}(x^k\!+\overline{v}^{k};x^k)\le\widehat{L}\varpi\alpha_r^{l}$ for all $k,l\in\mathbb{N}$. The above arguments show that the iterates $w^{k,l}$ satisfy \eqref{overall-cond} with $\phi_{\mathcal{R}}(t)=\alpha_r^{t}$ for $t\ge 0$. The conclusion then follows from Theorem \ref{complexity-oracle} and the iterations of Algorithm \ref{sPADMM}.
\end{proof}
\begin{remark}\label{remark-uniform-EB-ass}
 {\bf(a)} Example~\ref{example-UEB} in Appendix~A illustrates that the uniform error-bound condition required by Corollary~\ref{sPADMM-oracle} can hold under Assumption~\ref{ass1}, even in the presence of a Stiefel-manifold constraint and nonnegativity constraints.
	
 \noindent
 {\bf(b)} Suppose that both $\vartheta$ and ${\rm dist}(\cdot,K)$ are polyhedral functions. For each fixed $k\in\mathbb{N}$, the mapping $\mathcal{R}_k$ is piecewise polyhedral by Sun~\cite{Sun2025}. By Robinson~\cite[Proposition~1]{Robinson1981}, there exists a constant $c_{\mathcal{R},k}>0$ such that ${\rm dist}(w^{k,l},\mathcal R_k^{-1}(0))\le c_{\mathcal{R},k}\|\mathcal{R}_k(w^{k,l})\|$ for all $l\in\mathbb{N}_{+}$.
 Combining this error bound with the convergence analysis of Han et al.~\cite[Corollary~1]{Han2018} yields global linear convergence of the distance sequence for each fixed subproblem. This may help explain why Algorithm~\ref{PenAl} equipped with Algorithm~\ref{sPADMM} achieves slightly better empirical performance than that equipped with Algorithm~\ref{aADMM} in  Sections~\ref{sec6.4}--\ref{sec6.5}. However, since both the error-bound constant and the linear convergence factor depend on $k$, these subproblem-wise linear convergence results alone do not establish the overall oracle complexity bound.
 \end{remark}
\subsubsection{Dual inner solvers}\label{sec4.2.2}

Using $\mathcal{A}_kv-z+G(x^k,0)=0$, the dual of \eqref{Esubprob} in minimization form takes the form of
\begin{equation}\label{dsubprob}	\min_{\lambda\in\mathbb{Z}}\,D_k(\lambda):=q_k(\lambda)+\psi_{\rho_k}^*(\lambda)\ \ {\rm with}\ q_k(\lambda):=\frac{1}{2\beta_k}\big\|\mathcal{P}_{T_{\!x^k}\mathcal{M}}\big[\mathcal{A}_k^*\lambda+\nabla\ell_k(0)\big]\big\|^2-\langle\lambda,G(x^k,0)\rangle.
\end{equation}
For notational simplicity, for each $k\in\mathbb{N}$ and $\lambda\in\mathbb{Z}$, we define the following functions 
\begin{subequations}
\begin{equation}
 d_{v}^k(\lambda):=-\beta_k^{-1}\mathcal{P}_{T_{\!x^k}\mathcal{M}}(\mathcal{A}_k^*\lambda+\nabla\ell_k(0)),\ d_{z}^k(\lambda):=\mathcal{P}_{\psi_{\rho_k}}(\lambda-\!\nabla q_k(\lambda)),
\end{equation}
\begin{equation}
d_{\xi}^{k}(\lambda):=\lambda-\nabla q_k(\lambda)-d_{z}^k(\lambda)\ \ {\rm and}\ \ \Gamma_k(\lambda):=\mathcal{P}_{\psi_{\rho_k}^*}(\lambda-\!\nabla q_k(\lambda))-\lambda.
\end{equation}
\end{subequations}
For any fixed $k\in\mathbb{N}$, we first show how to construct a triplet $(v^k,z^k,\xi^k)\in T_{x^k}\mathcal{M}\times\mathbb{Z}\times\partial\psi_{\rho_k}(z^k)$ satisfying either \eqref{inexact-cond1} or  \eqref{model-descent}-\eqref{inexact-cond2} from the iterates of a solver applied to \eqref{dsubprob}. 
\begin{lemma}\label{dual-inexact}
 Fix any $k\in\mathbb{N}$. Let $\{\lambda^l\}_{l\in\mathbb{N}}$ be a bounded sequence generated by a method applied to \eqref{dsubprob}. Then, for each $l\in\mathbb{N}$, \((d_{v}^k(\lambda^l), d_{z}^k(\lambda^l), d_{\xi}^k(\lambda^l))\in T_{x^k}\mathcal{M} \times \mathbb{Z} \times \partial \psi_{\rho_k}(d_{z}^k(\lambda^l))\). If in addition  $\lim_{l\to\infty}\Gamma_k(\lambda^l)=0$, there exists $l_0\in\mathbb{N}$ such that $(d_{v}^k(\lambda^l), d_{z}^k(\lambda^l), d_{\xi}^k(\lambda^l))$ for all $l\ge l_0$ satisfy \eqref{inexact-cond1} or both \eqref{model-descent} and \eqref{inexact-cond2}. 
\end{lemma}
\begin{proof}
For each $l\in\mathbb{N}$, the inclusion \((d_{v}^k(\lambda^l), d_{z}^k(\lambda^l), d_{\xi}^k(\lambda^l))\in T_{x^k}\mathcal{M} \times \mathbb{Z} \times \partial \psi_{\rho_k}(d_{z}^k(\lambda^l))\) follows from the definitions of $d_{v}^k,d_{z}^k$ and $d_{\xi}^k$ and the identity $\mathcal{P}_{\psi_{\rho_k}}=(\mathcal{I}+\partial\psi_{\rho_k})^{-1}$. By Moreau's decomposition, we have
\begin{equation}\label{Gammak-relation}
\Gamma_k(\lambda)=-\nabla q_k(\lambda)-d_{z}^k(\lambda)		=\mathcal{A}_kd_{v}^k(\lambda)+G(x^k,0)-d_{z}^k(\lambda)\quad{\rm for\ any}\  \lambda\in\mathbb{Z}.
\end{equation}
Moreover, the definition of $d_{\xi}^k$ yields $d_{\xi}^k(\lambda^l)-\lambda^l=\Gamma_k(\lambda^l)$. Consequently, for each $l\in\mathbb{N}$, 
\begin{subequations}{}
\begin{equation}\label{feasi-tempineq}
\|G(x^k,d_{v}^k(\lambda^l))-d_{z}^k(\lambda^l)\|		=\|\Gamma_k(\lambda^l)\|,\qquad\qquad\qquad\qquad\qquad\qquad\quad
\end{equation}
\begin{equation}\label{opt-tempineq}
\begin{aligned}
 &\big\|\mathcal{P}_{T_{\!x^k}\mathcal{M}}		\big[\nabla\ell_k(0)+\beta_kd_{v}^k(\lambda^l)+\mathcal{A}_k^*d_{\xi}^k(\lambda^l)\big]\big\|\\
 &\le \big\|\mathcal{P}_{T_{\!x^k}\mathcal{M}}	\big[\nabla\ell_k(0)+\beta_kd_{v}^k(\lambda^l)+\mathcal{A}_k^*\lambda^l
		\big]\big\|+\|\mathcal{A}_k^*\Gamma_k(\lambda^l)\|
 =\|\mathcal{A}_k^*\Gamma_k(\lambda^l)\|,
 \end{aligned}
\end{equation}
\end{subequations}
where the second equality follows by the definition of $d_{v}^k$. Let $\lambda^*$ be any cluster point of $\{\lambda^l\}_{l\in\mathbb{N}}$. There exists an index set $\mathcal{K}\subset\mathbb{N}$ such that $\lim_{\mathcal{K}\ni l\to\infty}\lambda^l=\lambda^*$. By the continuity of $\Gamma_k$ and $\lim_{l\to\infty}\Gamma_k(\lambda^l)=0$, we have $\Gamma_k(\lambda^*)=0$, which means that $\lambda^*$ is an optimal solution of \eqref{dsubprob}. Moreover, it follows from \eqref{feasi-tempineq}-\eqref{opt-tempineq} that $\mathcal{R}_k(d_{v}^{k}(\lambda^*),d_{z}^k(\lambda^*),\lambda^*)=0$, so that $(d_{v}^{k}(\lambda^*),d_{z}^k(\lambda^*))$ is an optimal solution of \eqref{Esubprob}. Hence, $\overline{v}^k=d_{v}^{k}(\lambda^*)$. Since every cluster point of the bounded sequence $\{d_{v}^{k}(\lambda^l)\}_{l\in\mathbb{N}}$ equals $\overline{v}^k$, we have $\lim_{l\to\infty}d_{v}^k(\lambda^{l})=\overline{v}^k$. 
	
\noindent
\textbf{Case 1:} $\|\overline{v}^k\|=0$. Now $\lim_{l\to\infty}d_{v}^k(\lambda^{l})=0$. Since $\varepsilon_k>0$, there exists an index $l_0\in\mathbb{N}$ such that $\beta_k\|d_{v}^k(\lambda^l)\|\le \varepsilon_k$ for all $l\ge l_0$. Since $\lim_{l\to\infty}\Gamma_k(\lambda^l)=0$, it follows from \eqref{feasi-tempineq}-\eqref{opt-tempineq} that, for all $l\ge l_0$ (if necessary by enlarging $l_0$),
\begin{equation*}
\max\big\{\|G(x^k,d_{v}^k(\lambda^l))\!-\!d_{z}^k(\lambda^l)\|, \|\mathcal{P}_{T_{x^k}\mathcal{M}}[\nabla \ell_k(0)+\beta_kd_{v}^k(\lambda^l)+\!\mathcal{A}_k^*d_{\xi}^k(\lambda^l)]\|\big\}\le\varepsilon_k.    
\end{equation*}
This shows that \((d_{v}^k(\lambda^l), d_{z}^k(\lambda^l), d_{\xi}^k(\lambda^l))\) satisfies \eqref{inexact-cond1} for every $l\ge\!l_0$. 
	
\noindent
\textbf{Case 2:} $\|\overline{v}^k\|>0$. Since $\lim_{l\to\infty}d_{v}^k(\lambda^l)=\overline{v}^k$, there exists $l_0\in\mathbb{N}$ such that $\|d_{v}^k(\lambda^l)-\overline{v}^k\|\le \frac{1}{2}\|\overline{v}^k\|$ for every $l\ge l_0$, so $\|d_{v}^k(\lambda^l)\|\ge \frac{1}{2}\|\overline{v}^k\|$. Recall that $\lim_{l\to\infty}\Gamma_k(\lambda^l)=0$ and $\min\{a_k,b_k\}\|\overline{v}^k\|>0$. From \eqref{feasi-tempineq}-\eqref{opt-tempineq}, for each $l\ge l_0$ (if necessary by enlarging $l_0$), 
\begin{equation*}
\begin{aligned}
\|G(x^k,d_{v}^k(\lambda^l))-d_{z}^k(\lambda^l)\|\le \frac{1}{4}a_k\|\overline{v}^k\|^2\le a_k\|d_{v}^k(\lambda^l)\|^2,\qquad\qquad\qquad\\ 
\big\|\mathcal{P}_{T_{\!x^k}\mathcal{M}}		\big[\nabla\ell_k(0)+\beta_kd_{v}^k(\lambda^l)+\mathcal{A}_k^*d_{\xi}^k(\lambda^l)\big]\big\|\le\|\mathcal{A}^*_k\Gamma_k(\lambda^l)\|\le \frac{1}{2}b_k\|\overline{v}^k\|\le b_k\|d_{v}^k(\lambda^l)\|. 
\end{aligned}
\end{equation*}
This shows that \((d_{v}^k(\lambda^l), d_{z}^k(\lambda^l), d_{\xi}^k(\lambda^l)\) for every $l\ge l_0$ satisfies \eqref{inexact-cond2}. Note that
\begin{equation*}
\Theta_{\rho_k}(x^k)=\widehat{\Theta}_{\rho_k,\beta_k}(x^k;x^k)\ge \widehat{\Theta}_{\rho_k,\beta_k}(x^k\!+\overline{v}^k;x^k)+\frac{\beta_k}{2}\|\overline{v}^k\|^2>\widehat{\Theta}_{\rho_k,\beta_k}(x^k\!+\overline{v}^k;x^k),
\end{equation*}
where the first inequality follows from the $\beta_k$-strong convexity of $\widehat{\Theta}_{\rho_k,\beta_k}(x^k\!+\!\cdot;x^k)+\delta_{T_{x^k}\mathcal{M}}(\cdot)$. Due to  $\lim_{l\to\infty}d_{v}^k(\lambda^l)=\overline{v}^k$, the continuity of $\widehat{\Theta}_{\rho_k,\beta_k}(x^k\!+\!\cdot;x^k)$ implies $\Theta_{\rho_k}(x^k)>\widehat{\Theta}_{\rho_k,\beta_k}(x^k\!+\!d_{v}^k(\lambda^l);x^k)$ for all $l\ge l_0$ (if necessary by enlarging $l_0$). Thus, the triple \((d_{v}^k(\lambda^l), d_{z}^k(\lambda^l),d_{\xi}^k(\lambda^l)\) for all $l\ge l_0$ satisfy \eqref{model-descent}-\eqref{inexact-cond2}.
\end{proof}

While the preceding lemma ensures that the inexactness criteria can be satisfied in finitely many inner iterations for each fixed subproblem, the following theorem translates uniform dual objective-gap decay into the primal estimates needed for the total inner-iteration complexity analysis.
\begin{theorem}\label{dual-oracle}
 Let $B_0\subset\mathbb{Z}$ be a compact set. For each $k\in\mathbb{N}$, let $\{\lambda^{k,l}\}_{l\in\mathbb{N}}$ be the sequence of iterates generated by applying a solver to the dual problem \eqref{dsubprob}, with initial point $\lambda^{k,0}\in B_0$, and let $w^{k,l}:=(d_{v}^k(\lambda^{k,l}),d_{z}(\lambda^{k,l}),d_{\xi}(\lambda^{k,l}))$. Suppose that there exists a constant $c_d>0$ and $\phi_{d}\in\mathscr{D}$ such that  
\begin{equation}\label{Dk-levelset}
 D_k(\lambda^{k,l})-D^*_k\le c_{d}\phi_d(l),
\end{equation}
where $D_k^*$ is the optimal value of \eqref{dsubprob}. 
Then, under Assumption \ref{ass1}, there exist $c_{\rm gap},c_{\mathcal{R}}>0$ such that 
\begin{equation*}
 \widehat{\Theta}_{\rho_k,\beta_k}(x^k\!+d_{v}^k(\lambda^{k,l});x^k)-\widehat{\Theta}_{\rho_k,\beta_k}(x^k\!+\overline{v}^k;x^k)\le c_{\rm gap}\sqrt{\phi_d(l)}\ \ {\rm and}\ \ \|\mathcal{R}_k(w^{k,l})\|\le c_{\mathcal{R}}\sqrt{\phi_d(l)}
\end{equation*}
 for all $k,l\in\mathbb{N}$, and consequently, for any $\epsilon\in(0,\min\{1,\varepsilon_0\})$, Algorithm \ref{PenAl} generates an $\epsilon$-approximate KKT point within at most $O(\max\{\widehat{l}_\epsilon,1\}\epsilon^{-2})$ inner iterations in total, where 
 \[
 \widehat{l}_{\epsilon}:=\max\bigg\{\Big\lceil\phi_{d}^{-1}\Big(\min\Big\{1,\frac{\beta_{\min}(\widehat{\chi}\epsilon)^2}{32c_{\rm gap}\overline{\beta}^2}\Big\}^2\Big)\Big\rceil,\Big\lceil \phi_{d}^{-1}\Big(\min\!\Big\{\frac{\widehat{\chi}\epsilon}{c_{\mathcal{R}}},\frac{a_{\min}(\widehat{\chi}\epsilon)^2}{4c_{\mathcal{R}}\overline{\beta}^2},\frac{b_{\min}\widehat{\chi}\epsilon}{2\overline{\beta}c_{\mathcal{R}}},1\Big\}^2\Big)\Big\rceil\bigg\}.
\]
\end{theorem}
\begin{proof}
 By Lemma \ref{dual-inexact}, for each $k,l\in\mathbb{N}$, \(w^{k,l}\in T_{x^k}\mathcal{M} \times \mathbb{Z} \times \partial \psi_{\rho_k}(d_{z}^k(\lambda^{k,l}))\). For each $k\in\mathbb{N}$, let $\mathcal{D}^*_k$ denote the solution set of \eqref{dsubprob}. From the optimality conditions of \eqref{dsubprob}, we have $\mathcal{D}^*_k\subset\partial\psi_{\rho_k}(G(x^k,\overline{v}^k))$. By the definition of $\psi_{\rho_k}$ and Assumption \ref{ass1}, the set $\bigcup_{k\in\mathbb{N}}\mathcal{D}^*_k$ is bounded. This, together with \eqref{Dk-levelset} and \cite[Theorem 8.7]{Roc70}, implies that $\{\lambda^{k,l}\}_{l\in\mathbb{N}}$ for every $k\in\mathbb{N}$ is bounded. Fix any $k\in\mathbb{N}$. Let $\lambda^{k,*}$ be any cluster point of $\{\lambda^{k,l}\}_{l\in\mathbb{N}}$. Then  $\lambda^{k,*}$ is an optimal solution of \eqref{dsubprob}, and $\Gamma_k(\lambda^{k,*})=0$. Together with \eqref{Gammak-relation}, it follows that
 \begin{equation}\label{Ak-equa}
 \mathcal{A}_kd_{v}^k(\lambda^{k,*})+ G(x^k,0)=-\nabla q_k(\lambda^{k,*})=d_{z}^k(\lambda^{k,*})\in\partial\psi_{\rho_k}^*(\lambda^{k,*}).
 \end{equation} 
 For each $l\in \mathbb{N}$, from $d_{v}^k(\lambda^{k,l})\in T_{x^k}\mathcal{M}$ and the definition of $D_k$ in \eqref{dsubprob}, it follows that
 \begin{equation*}
  \begin{aligned}
 D_k(\lambda^{k,l})-D^*_k&=\frac{1}{2\beta_k}\|\mathcal{P}_{T_{x^k}\mathcal{M}}(\mathcal{A}^*\lambda^{k,l}+\nabla\ell_k(0))\|^2-\langle G(x^k,0),\lambda^{k,l}-\lambda^{k,*}\rangle \\
 &\quad -\frac{1}{2\beta_k}\|\mathcal{P}_{T_{x^k}\mathcal{M}}(\mathcal{A}^*\lambda^{k,*}+\nabla\ell_k(0))\|^2+\psi^*_{\rho_k}(\lambda^{k,l})-\psi^*_{\rho_k}(\lambda^{k,*})\\
 &\ge\frac{\beta_k}{2}\|d_{v}^k(\lambda^{k,l})-d_{v}^k(\lambda^{k,*})\|^2 +\langle d_{z}^k(\lambda^{k,*}),\lambda^{k,l}-\lambda^{k,*}\rangle\\
 &\quad-\langle\mathcal{A}_kd_{v}^k(\lambda^{k,*})+ G(x^k,0),\lambda^{k,l}-\lambda^{k,*}\rangle\\
 &=\frac{\beta_k}{2}\|d_{v}^k(\lambda^{k,l})-d_{v}^k(\lambda^{k,*})\|^2,
 \end{aligned}
 \end{equation*}
 where the inequality is due to the definition of $d_{v}^k(\cdot)$ and the inclusion in \eqref{Ak-equa}, and the second equality is implied by the equalities in \eqref{Ak-equa}. Combining the above inequality with \eqref{Dk-levelset} yields, for each $l\in\mathbb{N}$, 
 \begin{equation}\label{Dk-vk}
 c_{d}\phi_d(l)\ge D_k(\lambda^{k,l})-D^*_k\ge \frac{\beta_k}{2}\|d_{v}^k(\lambda^{k,l})-d_{v}^k(\lambda^{k,*})\|^2=\frac{\beta_k}{2}\|d_{v}^k(\lambda^{k,l})-\overline{v}^k\|^2,
 \end{equation}
 where the equality is due to the fact that $(d_{v}^{k}(\lambda^{k,*}),d_{z}^k(\lambda^{k,*}))$ is an optimal solution of \eqref{Esubprob} with $d_{v}^{k}(\lambda^{k,*})=\overline{v}^k$. A similar estimate relating the dual objective gap to the squared primal-iterate error appears in the proof of Beck and Teboulle \cite[Theorem 4.1]{Beck2014}. Inequality \eqref{Dk-vk} implies the boundedness of $\{d_{v}^k(\lambda^{k,l})\}_{k,l\in\mathbb{N}}$, since $\beta_k\ge\beta_{\min}$ for all $k\in\mathbb{N}$, $\phi_d(l)\in (0,1]$, and $\{\overline{v}^k\}_{k\in\mathbb{N}}$ is bounded. Then there exists a compact convex set $\mathcal{X}$ containing the sequences $\{d_{v}^k(\lambda^{k,l})\}_{k,l\in\mathbb{N}}$ and $\{\overline{v}^k\}_{k\in\mathbb{N}}$. Let $\widehat{L}_{\vartheta,\mathcal{X}}$ denote the Lipschitz constant of $\vartheta$ on the set $\mathcal{X}$, and define $\widehat{L}_{d}:=L_{\nabla\!f}+L_{h}+\widehat{L}_{\vartheta,\mathcal{X}}+c_u\overline{\rho}\gamma_{\nabla g}+\overline{\beta}\widehat{c}_{v}$ with $\widehat{c}_{v}:=\sup_{v\in\mathcal{X}}\|v\|$, where  $L_{\nabla\!f}$ and $L_{h}$ are defined in Lemma \ref{lemma1-complexity}(i), and $\gamma_{\nabla g}$ is defined in Proposition \ref{prop2-complexity}. Since $\widehat{\Theta}_{\rho_k,\beta_k}(x^k\!+\!\cdot;x^k)$ is $\widehat{L}_d$-Lipschitz continuous on the set $\mathcal{X}$, for each $k,l\in \mathbb{N}$, we have
 \begin{equation*}
 \widehat{\Theta}_{\rho_k,\beta_k}(x^k+d_{v}^k(\lambda^{k,l});x^k)-\widehat{\Theta}_{\rho_k,\beta_k}(x^k+\overline{v}^{k};x^k)\le \widehat{L}_d\|d_{v}^k(\lambda^{k,l})-\overline{v}^{k}\|\le \frac{\widehat{L}_d\sqrt{2c_d}}{\sqrt{\beta_{\min}}}\sqrt{\phi_d(l)}.
 \end{equation*}
 This shows that the desired first inequality holds. We next establish the second inequality. Assumption \ref{ass1} implies the existence of a constant $c_{\mathcal{A}}>0$ such that $\sup_{k\in\mathbb{N}}\|\mathcal{A}_k\|\le c_{\mathcal{A}}$, so $\nabla q_k(\cdot)$ for every $k\in \mathbb{N}$ is Lipschitz continuous with Lipschitz constant $L_q:={c_{\mathcal{A}}^2}/{\beta_{\min}}$. Define $\Gamma_k^+(\lambda):=\mathcal{P}_{L_q^{-1}\psi_{\rho_k}^*}(\lambda-\!L_q^{-1}\nabla q_k(\lambda))-\lambda$ for $\lambda\in\mathbb{Z}$. Invoking Beck \cite[Corollary 10.6]{Beck2017} yields, for all $k,l\in\mathbb{N}$,
 \begin{equation*}
 \frac{L_{q}}{2}\|\Gamma_k^+(\lambda^{k,l})\|^2\le D_k(\lambda^{k,l})-D_k\big(\mathcal{P}_{L_{q}^{-1}\psi_{\rho_k}^*}\big(\lambda^{k,l}-\!L_{q}^{-1}\nabla q_k(\lambda^{k,l})\big)\big)\le D_k(\lambda^{k,l})-D^*_k.
\end{equation*}
 Furthermore, by Beck \cite[Theorem 10.9]{Beck2017}, we have $\|\Gamma_k(\lambda^{k,l})\|\le \max\{1,L_{q}\}\|\Gamma_k^+(\lambda^{k,l})\|$ for every $k,l\in\mathbb{N}$. Together with the above inequality and \eqref{Dk-levelset}, we get $\|\Gamma_k(\lambda^{k,l})\|^2\le 2\max^2\{L_{q},1\}L_{q}^{-1}c_d\phi_d(l)$ for all $k,l\in\mathbb{N}$. Combining this with \eqref{feasi-tempineq}-\eqref{opt-tempineq} and $d_{\xi}^k(\lambda^{k,l})\in\partial\psi_{\rho_k}(d_{z}^k(\lambda^{k,l}))$ yields
 \begin{equation*}
 \|\mathcal{R}_k(w^{k,l})\|^2\le 2c_d(1+c_{\mathcal{A}}^2)\max\{L_{q}^{-1},L_{q}\}\phi_d(l)\quad{\rm for\ all}\  k,l\in\mathbb{N}.
\end{equation*}
This establishes the second inequality. By Theorem \ref{complexity-oracle}, Algorithm \ref{PenAl} generates an $\epsilon$-approximate KKT point within $O(\max\{\widehat{l}_\epsilon,1\}\epsilon^{-2})$ inner iterations in total.
\end{proof}

Combining the proof of Theorem \ref{dual-oracle} with \cite[Theorem 10.34]{Beck2017}, we obtain the following overall oracle complexity bound for Algorithm \ref{PenAl} with FISTA as the dual inner solver.   
\begin{corollary}\label{APG-oracle}
Let $B_{d}\subset\mathbb{Z}$ be a nonempty compact set. Suppose that FISTA in Beck \cite{Beck2009}, with a fixed stepsize and an initial point in $B_{d}$, is employed as the dual inner solver. Then, under Assumption \ref{ass1}, for any $\epsilon\in(0,\min\{1,\varepsilon_0\})$, Algorithm \ref{PenAl} generates an $\epsilon$-KKT point with $O(\epsilon^{-4})$ oracle calls in total.
\end{corollary}
\section{Global convergence analysis}\label{sec5}

This section is dedicated to the convergence analysis of the sequence $\{x^k\}_{k\in\mathbb{N}}$ under Assumption \ref{ass1}. In view of Lemma \ref{Lemma-Sfinite}, we assume that the index set $\mathcal{S}$ is infinite throughout this section. The following theorem establishes the subsequential convergence of $\{x^k\}_{k\in\mathbb{N}}$. 
\begin{theorem}\label{converge-theorem1}
Under Assumption \ref{ass1}, the following assertions hold.
\begin{itemize}
\item[(i)] The sequence $\{\Theta_{\rho_{k}}(x^k)\}_{k\ge\overline{k}}$ is nonincreasing, and $\lim_{k\to\infty}\Theta_{\rho_k}(x^k)=\varpi_*$ for some $\varpi_*$. 
		
\item[(ii)] $\lim_{\mathcal{S}\ni k\to\infty}v^k=0$, and every cluster point of $\{x^k\}_{k\in\mathbb{N}}$ is a KKT point of \eqref{Rcprob}.
		
\item[(iii)] $\lim_{k\to\infty}{\rm dist}(g(x^k),K)=0$ and $\lim_{k\to\infty}\Theta(x^k)=\varpi_*$.
\end{itemize}
\end{theorem}
\begin{proof}
{\bf(i)} For each $\mathcal{S}\ni k\ge\overline{k}$, since $r_k\ge\eta_1$ and ${\rm pred}_k\ge\overline{\gamma}\|v^k\|^2$, inequality \eqref{S-descent} is specified as follows
\begin{equation}\label{temp-ineq51}
\Theta_{\overline{\rho}}(x^k)-\Theta_{\overline{\rho}}(x^{k+1})\ge \overline{\gamma}\eta_1\|v^k\|^2.
\end{equation}
Note that $x^k=x^{k+1}$ for all $k\notin\mathcal{S}$. Then, the sequence $\{\Theta_{\rho_{k}}(x^k)\}_{k\ge\overline{k}}$ is nonincreasing. Furthermore, from $\Theta_{\rho_{k}}(x^k)\ge\min_{x\in\varLambda_0}\Theta(x)$ for all $k\in\mathbb{N}$, the sequence $\{\Theta_{\rho_{k}}(x^k)\}_{k\in\mathbb{N}}$ is convergent. 

\noindent
{\bf(ii)} From inequality \eqref{temp-ineq51}, it follows that, for any $k,k'\in\mathcal{S}$ with $ \overline{k}\le k\le k'$,
\begin{equation*}
\resizebox{\linewidth}{!}{$
\begin{aligned}
\overline{\gamma}\eta_1\sum_{\mathcal{S}\ni k\ge \overline{k}}^{k'}\!\!\|v^k\|^2&\le\sum_{\mathcal{S}\ni k\ge\overline{k}}^{k'}\big[\Theta_{\overline{\rho}}(x^k)-\Theta_{\overline{\rho}}(x^{k+1})\big]=\Theta_{\overline{\rho}}(x^{\overline{k}})-\Theta_{\overline{\rho}}(x^{k'+1})\le \Theta_{\overline{\rho}}(x^{\overline{k}})-\min_{x\in\varLambda_0}\Theta(x),
\end{aligned}
$}
\end{equation*}
where the equality is obtained by using $x^k=x^{k+1}$ for all $k\notin\mathcal{S}$, and the second inequality is due to $\Theta_{\overline{\rho}}(x^{k'+1})\ge\Theta(x^{k'+1})\ge\min_{x\in\varLambda_0}\Theta(x)$. 
Passing to $\mathcal{S}\ni k'\to\infty$ in the above inequality gives $\overline{\gamma}\eta_1\sum_{\mathcal{S}\ni k\ge \overline{k}}^{\infty}\|v^k\|^2<\infty$. Consequently, $\lim_{\mathcal{S}\ni k\to\infty}v^k=0$. Let $x^*$ be any cluster point of $\{x^k\}_{k\in\mathbb{N}}$. By Remark \ref{remark-alg}(c), it is also a cluster point of $\{x^k\}_{k\in\mathcal{S}}$, so there exists an infinite index set $\mathcal{K}\subset\mathcal{S}$ such that $\lim_{\mathcal{K}\ni k\to\infty}x^k=x^*$. Then, for each $k\in\mathcal{S}$, the triple $(v^k,z^k,\xi^k)\in T_{x^k}\mathcal{M}\times \mathbb{Z}\times\partial\psi_{\rho_k}(z^k)$ satisfies \eqref{model-descent}-\eqref{inexact-cond2}. Taking the limit as $\mathcal{S}\ni k\to\infty$ in the first inequality of \eqref{inexact-cond2} and using $\lim_{\mathcal{S}\ni k\to\infty}v^k=0$ and $a_k\le a_{\max}$ for all $k\in\mathbb{N}$, we get $\lim_{\mathcal{K}\ni k\to\infty}z^k=(x^*;g(x^*))$. By Lemma \ref{bound-lemma}, the sequence $\{(\zeta^k,\xi^k)\}_{k\in\mathcal{K}}$ is bounded. By passing to a subsequence if necessary, we can assume that $\lim_{\mathcal{K}\ni k\to\infty}(\zeta^k,\xi^k)=(\zeta^*,\xi^*)$. The outer semicontinuity of $\partial(-h),\partial\vartheta(\cdot)$ and $\partial{\rm dist}(\cdot,K)$, together with $\rho_k=\overline{\rho}$ for all $k\ge\overline{k}$, implies that $\zeta^*\in\partial(-h)(x^*)$ and $\xi^*=(\xi_1^*;\xi_2^*)\in\partial\vartheta(x^*)\times[\overline{\rho}\partial{\rm dist}(g(x^*),K)]$. In addition, from Remark \ref{remark-ass1}(b) and the iterations of Algorithm \ref{PenAl}, it follows that, for each $\mathcal{S}\ni k\ge\overline{k}$, 
\begin{equation*}
\beta_k\|v^k\|
\ge\min\{\rho_k^{-1},{\rm dist}(z_2^k,K)\}. 
\end{equation*}
Passing to the limit $\mathcal{S}\ni k\to\infty$ in this inequality and using $\lim_{\mathcal{S}\ni k\to\infty}v^k=0$ yields ${\rm dist}(g(x^*),K)=0$, which together with $\overline{\rho}^{-1}\xi_2^*\in\partial{\rm dist}(g(x^*),K)$ and Lemma \ref{subdiff-dist} implies  $\xi_2^*\in\mathcal{N}_K(g(x^*))$. Taking the limit as $\mathcal{S}\ni k\to\infty$ in the second inequality of \eqref{inexact-cond2} and using the continuity of $\mathcal{P}_{T_{x}\mathcal{M}}$ as a function of $x\in\mathcal{M}$ and $b_k\le b_{\rm max}$, we have $0\in\nabla\! f(x^*)+\zeta^*+\xi_1^*+\nabla g(x^*)\xi_2^*+N_{x^*}\mathcal{M}$. This shows that $x^*$ is a KKT point of \eqref{Rcprob}. 
	
\noindent
{\bf(iii)} By item (ii), ${\rm dist}(g(x^*),K)=0$ for any cluster point $x^*$ of $\{x^k\}_{k\in\mathbb{N}}$. The continuity of ${\rm dist}(g(\cdot),K)$ then yields $\lim_{k\to\infty}{\rm dist}(g(x^k),K)=0$. Consequently, $\lim_{k\to\infty}\Theta(x^k)=\lim_{k\to\infty}\Theta_{\rho_k}(x^k)=\varpi_*$. 
\end{proof}

The remainder of this section establishes the convergence of $\{x^k\}_{k\in\mathbb{N}}$ under the KL property. By Remark~\ref{remark-alg}(c), the sequences $\{x^k\}_{k\in\mathbb{N}}$ and $\{x^k\}_{k\in\mathcal{S}}$ have the same set of cluster points. Therefore, it suffices to establish the convergence of $\{x^k\}_{k\in\mathcal{S}}$. To this end, we first establish the following lemma, which is instrumental in constructing an appropriate potential function.
\begin{lemma}\label{Lemma-Thetakbound}
 Suppose that Assumption \ref{ass1} holds. Let $\overline{c}:=L+2a_{\rm max}(L_{\vartheta}^2+\!c_u^2\overline{\rho}^2)^{1/2}+\overline{\beta}/2$, where $L$ and $L_{\vartheta}$ are the constants defined in Lemma \ref{lemma1-complexity}(i). Then, for each $\mathcal{S}\ni k\ge\overline{k}$, 
\begin{equation*}
\Theta_{\rho_{k}}(x^{k+1})\le f(x^k)-h(x^k)+\langle\nabla\!f(x^k)+\zeta^k,v^k\rangle+\langle \xi^k,G(x^k,v^k)-z^k\rangle+\psi_{\overline{\rho}}(z^k)+\overline{c}\|v^k\|^2.
\end{equation*}
\end{lemma}
\begin{proof}
 Fix any $\mathcal{S}\ni k\ge\overline{k}$. By Lemma \ref{Lemma2-Thetak} and the definition of $L$, we have
 \begin{equation*}
 \begin{aligned}
 \Theta_{\rho_{k}}(x^{k+1})&\le \widehat{\Theta}_{\rho_{k},\beta_k}(x^k+v^k;x^k)+L\|v^k\|^2\\
 &\overset{\eqref{Thetahat-rhobetak}}{=}f(x^k)-h(x^k)+\langle\nabla\!f(x^k)+\zeta^k,v^k\rangle+\frac{\beta_k}{2}\|v^k\|^2\!+\!L\|v^k\|^2+\psi_{\rho_{k}}(z^k)\\
 &\quad +\langle \xi^k,G(x^k,v^k)-z^k\rangle+\psi_{\rho_k}(G(x^k,v^k))-\psi_{\rho_k}(z^k)-\langle\xi^k,G(x^k,v^k)\!-\!z^k\rangle.
 \end{aligned}
\end{equation*} 
For the last three terms on the right hand side, following the arguments similar to \eqref{temp-ineq31} yields
\begin{equation*}
\psi_{\rho_k}(G(x^k,v^k))-\psi_{\rho_k}(z^k)-\langle\xi^k,G(x^k,v^k)-z^k\rangle\le 2a_{\rm max}(L_{\vartheta}^2+\!c_u^2\overline{\rho}^2)^{1/2}\,\|v^k\|^2.
\end{equation*}
The above two inequalities, together with $\rho_k=\overline{\rho}$ for $k\ge\overline{k}$ and $\beta_k\le\overline{\beta}$ for all $k$, imply the conclusion.
\end{proof}

Let $\{e_1,\ldots,e_p\}$ be an orthonormal basis of $\mathbb{X}$, and let $\{d_1,\ldots,d_q\}$ an orthonormal basis of $\mathbb{Y}$. For each $k\in\mathbb{N}$, let $Q^k\in\mathbb{R}^{q\times p}$ be the unique matrix representation of $g'(x^k)\!:\mathbb{X}\to\mathbb{Y}$ under these two bases. Then, with the linear mappings $\mathcal{G}:\mathbb{R}^p\to\mathbb{X}$ and $\mathcal{F}:\mathbb{R}^q\to\mathbb{Y}$ defined by $\mathcal{G}\widehat{x}:=\sum_{i=1}^{p}\widehat{x}_ie_i$ for $\widehat{x}\in\mathbb{R}^p$ and $\mathcal{F}\widehat{y}:=\sum_{i=1}^{q}\widehat{y}_id_i$ for $\widehat{y}\in\mathbb{R}^q$, respectively, the linear operator $g'(x^k)$ and its adjoint $\nabla g(x^k)$ are represented as 
\begin{equation}\label{Jac-gxk}
g'(x^k)v=\mathcal{F}Q^k\mathcal{G}^*v\ \ {\rm for}\ v\in\mathbb{X}\ \ {\rm and}\ \ \nabla g(x^k)y=\mathcal{G}(Q^k)^{\top}\mathcal{F}^*y\ \ {\rm for}\ y\in\mathbb{Y}.
\end{equation} 
Let $\mathbb{W}:=\mathbb{X}\times\mathbb{X}\times\mathbb{R}^{q\times p}\times\mathbb{Z}\times\mathbb{Z} \times\mathbb{X}\times\mathbb{X}$. Motivated by Lemma \ref{Lemma-Thetakbound}, we define the potential function 
\[
\Xi_{\overline{c}}(w):=f(x)+\langle x,\zeta\rangle+h^*(-\zeta)+\langle y,v\rangle+\psi_{\overline{\rho}}(z) +\langle\xi,H(x,v,Q)-z\rangle+\delta_{T\mathcal{M}}(x,v)+\overline{c}\|v\|^2
\]
for $w=(x,v,Q,\xi,z,\zeta,y)\in\mathbb{W}$, where $T\mathcal{M}$ denotes the tangent bundle of $\mathcal{M}$ (see Boumal \cite[Section 3.5]{Boumal2023} for more details), $\overline{c}$ is defined in Lemma \ref{Lemma-Thetakbound}, and $H:\mathbb{X}\times\mathbb{X}\times\mathbb{R}^{q\times p}\to\mathbb{Z}$ is the mapping defined by 
\begin{equation}\label{Hmap}
H(x,v,Q):=\begin{pmatrix}
 x+v\\ g(x)+\mathcal{F}Q\mathcal{G}^*v        
\end{pmatrix}.
\end{equation}
For each $k\in\mathbb{N}$, define $w^k:=(x^k,v^k,Q^k,\xi^k,z^k,\zeta^k,\nabla\!f(x^k)+\zeta^k)\in\mathbb{W}$. The following proposition establishes the convergence of the sequence $\{\Xi_{\overline{c}}(w^k)\}_{k\in\mathcal{S}}$ under Assumption \ref{ass1}.   
\begin{proposition}\label{prop-Xicbar}
Under Assumption \ref{ass1}, the following statements hold true.
\begin{itemize}
 \item[(i)] The set of cluster points of $\{w^k\}_{k\in\mathcal{S}}$, denoted by $\mathcal{W}^*$, is nonempty and compact.
		
 \item[(ii)] $\Theta(x^{k+1})\le\Theta_{\rho_{k}}(x^{k+1})\le\Xi_{\overline{c}}(w^k)$ for all $\mathcal{S}\ni k\ge\overline{k}$. 
		
 \item[(iii)] $\lim_{\mathcal{S}\ni k\to\infty}\Xi_{\overline{c}}(w^k)=\varpi_*$, and $\Xi_{\overline{c}}(w)=\varpi_*$ for all $w\in\mathcal{W}^*$. 
\end{itemize}
\end{proposition}
\begin{proof}
 Under Assumption \ref{ass1}, the linear mapping $g'(x^k):\mathbb{X}\to\mathbb{Y}$ is bounded, which implies that $\{Q^k\}_{k\in\mathbb{N}}$ is bounded. Item (i) then follows from the definition of $\{w^k\}_{k\in\mathbb{N}}$ and Lemma \ref{bound-lemma}. For item (ii), it suffices to prove the second inequality. Fix any $\mathcal{S}\ni k\ge\overline{k}$. Since $\zeta^k\in\partial(-h)(x^k)\subset-\partial h(x^k)$, the Fenchel-Young equality yields $h^*(-\zeta^k)+h(x^k)=-\langle x^k,\zeta^k\rangle$. Together with Lemma \ref{Lemma-Thetakbound}, we obtain
 \begin{equation}\label{ineq1-Xicbar}
  \begin{aligned}
  \Theta_{\rho_{k}}(x^{k+1})
  &\le f(x^k)-h(x^k)+\langle\nabla\!f(x^k)+\zeta^k,v^k\rangle+\langle \xi^k,G(x^k,v^k)-z^k\rangle+\psi_{\overline{\rho}}(z^k)+\overline{c}\|v^k\|^2\\
  &= f(x^k)+\langle x^k,\zeta^k\rangle+h^*(-\zeta^k)+\langle \nabla\!f(x^k)+\zeta^k,v^k\rangle+\psi_{\overline{\rho}}(z^k)+\overline{c}\|v^k\|^2\\
  &\quad +\langle\xi^k,H(x^k,v^k,Q^k)-z^k\rangle+\delta_{T\mathcal{M}}(x^k,v^k)=\Xi_{\overline{c}}(w^k),
  \end{aligned}
 \end{equation}
 where the first equality follows from  $G(x^k,v^k)=H(x^k,v^k,Q^k)$ and $\delta_{T\mathcal{M}}(x^k,v^k)=0$. Here, the equality $G(x^k,v^k)=H(x^k,v^k,Q^k)$ follows by combining \eqref{Jac-gxk}-\eqref{Hmap} with the definition of \(G(\cdot,\cdot)\) in \eqref{Esubprob}. Thus, we establish item (ii). From the equalities in \eqref{ineq1-Xicbar}, for each $\mathcal{S}\ni k\ge\overline{k}$,  
 \[
	\Xi_{\overline{c}}(w^k)=f(x^k)-h(x^k)
	+\langle y^k,v^k\rangle +\langle\xi^k,G(x^k,v^k)\!-\!z^k\rangle+\overline{c}\|v^k\|^2+\vartheta(z_1^k)+\overline{\rho}{\rm dist}(z_2^k,K).
 \]
 By Theorem \ref{converge-theorem1}(ii), $\lim_{\mathcal{S}\ni k\to\infty}v^k=0$ which, combined with the first inequality in \eqref{inexact-cond2}, implies that $\lim_{\mathcal{S}\ni k\to\infty}(G(x^k,v^k)-z^k)=0$. Consequently, $\lim_{\mathcal{S}\ni k\to\infty}(x^k-z_1^k)=0$ and $\lim_{\mathcal{S}\ni k\to\infty}(g(x^k)-z_2^k)=0$. Taking the limit as $\mathcal{S}\ni k\to\infty$ in the above equality and using Theorem \ref{converge-theorem1} leads to
$\lim_{\mathcal{S}\ni k\to\infty}\Xi_{\overline{c}}(w^k)=\lim_{\mathcal{S}\ni k\to\infty}\Theta_{\rho_k}(x^k)=\varpi_*$. Following the similar arguments also yields $\Xi_{\overline{c}}(w)=\varpi_*$ for all $w\in\mathcal{W}^*$. 
\end{proof}

To show that the sequence $\{x^k\}_{k\in\mathcal{S}}$ converges to some $x^*$ under the KL property, we need to characterize the relative error condition for the potential function $\Xi_{\overline{c}}$.
\begin{proposition}\label{rebound}
Under Assumption \ref{ass1}, there exists a constant $\widehat{c}>0$ such that for all sufficiently large $k\in\mathcal{S}$, it holds that ${\rm dist}(0,\partial\Xi_{\overline{c}}(w^k))\le \widehat{c}\,\|v^k\|$.
\end{proposition}
\begin{proof}
Since $\mathcal{M}$ is a closed $\mathcal{C}^2$-embedded
submanifold and $\varLambda_0\subset\mathcal{M}$ is compact
by Remark~\ref{remark-ass1}(a), invoking He et al. \cite[Lemma 2.2]{He2026} ensures the existence of an integer $l\in\mathbb{N}_+$, points
$\overline{x}^1,\ldots,\overline{x}^l\in\varLambda_0$,
positive numbers $\varepsilon_1,\ldots,\varepsilon_l$,
and $\mathcal{C}^2$-smooth mappings
$G_i:=G_{\overline{x}^i}$ from $\mathbb{X}$ to $\mathbb{U}$ for $i\in[l]_+$ such that $\varLambda_0\subset\bigcup_{i=1}^l
\mathbb{B}^{\circ}(\overline{x}^i,\varepsilon_i)$ and $G_i'(x)\!:\mathbb{X}\to\mathbb{U}$ for $i\in[l]_{+}$ is surjective
for all $x\in\mathbb{B}^{\circ}(\overline{x}^i,\varepsilon_i)$. Here, $\mathbb{U}$ denotes a real Euclidean space of dimension ${\rm dim}(\mathbb{X})-{\rm dim}(\mathcal{M})$. By the compactness of $\varLambda_0$, we may choose
$\delta_i\in(0,\varepsilon_i)$ for $i\in[l]_+$ such that
$\varLambda_0\subset
\bigcup_{i=1}^l\mathbb{B}^{\circ}(\overline{x}^i,\delta_i)$. For each $i\in[l]_{+}$, since $\mathbb{B}(\overline{x}^i,\delta_i)
\subset\mathbb{B}^{\circ}(\overline{x}^i,\varepsilon_i)$,
the continuity of $\nabla G_i$, the surjectivity
of $G_i'$ on $\mathbb{B}^{\circ}(\overline{x}^i,\varepsilon_i)$, and the compactness of $\mathbb{B}(\overline{x}^i,\delta_i)$ imply that there exists $\kappa_i>0$ such that $\|\nabla G_i(x)u\|\ge\kappa_i\|u\|$ for all $x\in\mathbb{B}(\overline{x}^i,\delta_i)$ and $u\in\mathbb{U}$. Therefore, by taking $\kappa:=\min_{i\in[l]_+}\kappa_i>0$, we have
\begin{equation}\label{uniform-G-bound}
\|\nabla G_i(x)u\|\ge\kappa\|u\|
\end{equation}
for every $i\in[l]_+$, $x\in\mathbb{B}(\overline{x}^i,\delta_i)$
and $u\in\mathbb{U}$. Moreover, as shown in the proof of
He et al. \cite[Lemma 2.2]{He2026}, each $G_i$ is a local defining function of $\mathcal{M}$ on $\mathbb{B}^{\circ}(\overline{x}^i,\varepsilon_i)$.
Consequently, the corresponding normal-space representations hold at any point in $\mathcal{M}\cap\mathbb{B}^{\circ}(\overline{x}^i,\varepsilon_i)$. Since $\{x^k\}_{k\in\mathcal{S}}\subset\varLambda_0\subset\mathcal{M}$
and $v^k\in T_{x^k}\mathcal{M}$, for each $k\in\mathcal{S}$, we may choose $j_k\in[l]_+$ such that
\begin{subnumcases}{}\label{Normal-space}
 x^k\in\mathcal{M}\cap\mathbb{B}(\overline{x}^{j_k},\delta_{j_k}),\ N_{x^k}\mathcal{M} = \big\{\nabla G_{j_k}(x^k)u \mid u \in \mathbb{U}\big\}, \\
\label{Tangent-space}
 N_{(x^k, v^k)} T\mathcal{M} =\left\{ \begin{pmatrix}
 \nabla G_{j_k}(x^k)u_1 + [D^2G_{j_k}(x^k) v^k]^*u_2 \\
  \nabla G_{j_k}(x^k)u_2
 \end{pmatrix} \mid u_1 \in \mathbb{U}, u_2 \in\mathbb{U} \right\}.
\end{subnumcases}
Since $s-\mathcal{P}_{T_{x^k}\mathcal{M}}(s)\in N_{x^k}\mathcal{M}$
for every $s\in\mathbb{X}$, combining the second inequality in
\eqref{inexact-cond2} with \eqref{Normal-space} guarantees the existence of $u^k\in\mathbb{U}$ such that $\|\nabla f(x^k)+\zeta^k+\beta_kv^k+\xi_1^k+\nabla g(x^k)\xi_2^k
+\nabla G_{j_k}(x^k)u^k\|\le b_k\|v^k\|$. This, together with $\beta_k\le\overline{\beta}$ and $b_k\le b_{\rm max}$ for every $k\in\mathcal{S}$, implies that
\begin{equation}\label{uk-bound}
 \|\nabla\!f(x^k)+\zeta^k+\xi_1^k+\nabla g(x^k)\xi_2^k+\nabla G_{j_k}(x^k)u^k\|\le (b_{\max}+\overline{\beta})\|v^k\|.
\end{equation} 
In particular, combining \eqref{uk-bound} with the above \eqref{uniform-G-bound} yields, for every $k\in\mathcal{S}$, 
\begin{equation*}
\kappa\|u^k\|\le\|\nabla G_{j_k}(x^k)u^k\|\le (b_{\max}+\overline{\beta})\|v^k\|
+\|\nabla f(x^k)+\zeta^k+\xi_1^k
+\nabla g(x^k)\xi_2^k\|.
\end{equation*}
The right-hand side is bounded by Lemma~\ref{bound-lemma} and the continuity of $\nabla f$ and $\nabla g$, so $\{u^k\}_{k\in\mathcal{S}}$ is bounded. 

For each $k\in\mathcal{S}$, from $w^k=(x^k,v^k,Q^k,\xi^k,z^k,\zeta^k,\nabla\!f(x^k)+\zeta^k)$ and the definition of $\Xi_{\overline{c}}$, it follows that
\begin{subequations}
\begin{equation}\label{xv-Xi}
\partial_{x,v,Q}\Xi_{\overline{c}}(w^k)=\begin{pmatrix}
\nabla\!f(x^k)+\zeta^k\\ \nabla\!f(x^k)+\zeta^k+2\overline{c}v^k\\ 0
\end{pmatrix}+\nabla H(x^k,v^k,Q^k)\xi^k+\begin{pmatrix}
N_{(x^k, v^k)} T \mathcal{M}\\
0
\end{pmatrix},
\end{equation}
\begin{equation}\label{others-Xi}
\partial_{\xi,z,\zeta,y}\Xi_{\overline{c}}(w^k)=\{H(x^k,v^k,Q^k)\!-\!z^k\}\times[\partial\psi_{\rho_{\overline{k}}}(z^k)-\xi^k]\times[x^k\!-\!\partial h^*(-\zeta^k)]\times\{v^k\}.
\end{equation}
\end{subequations}
Then, for each $k\in\mathcal{S}$, combining \eqref{Normal-space}-\eqref{Tangent-space} with the definition of $H$ in \eqref{Hmap} yields
\begin{equation}\label{sk-def}
\resizebox{\linewidth}{!}{$
 s^k:=\begin{pmatrix}
 \nabla\!f(x^k)+\zeta^k+\xi_1^k+\nabla g(x^k)\xi_2^k+\!\nabla G_{j_k}(x^k)u^k+ [D^2G_{j_k}(x^k) v^k]^*u^k\\ \nabla\!f(x^k)+\zeta^k+2\overline{c}v^k+\xi_1^k+\mathcal{G}(Q^k)^{\top}\mathcal{F}^*\xi_2^k+\nabla G_{j_k}(x^k)u^k\\
 (\mathcal{F}^*\xi_2^k)(\mathcal{G}^*v^k)^{\top}  
\end{pmatrix}\in\partial_{x,v,Q}\Xi_{\overline{c}}(w^k).
$}
\end{equation}
For each $k\in\mathcal{S}$, recalling that $\xi^k\in\partial\psi_{\rho_{\overline{k}}}(z^k)$ and $\zeta^k\in\partial(-h)(x^k)\subset-\partial h(x^k)$, we have $0\in\partial\psi_{\rho_{\overline{k}}}(z^k)-\xi^k$ and $0\in x^k-\partial h^*(-\zeta^k)$. Comparing with \eqref{others-Xi} and using \eqref{sk-def} gives 
\begin{equation}\label{temp-inclusion}
 (s^k,H(x^k,v^k,Q^k)-z^k,0,0,v^k)\in\partial\Xi_{\overline{c}}(w^k)\quad{\rm for\ each}\ k\in\mathcal{S}.
\end{equation}
Recall that the sequence $\{(x^k,u^k,\xi^k)\}_{k\in\mathcal{S}}$ is bounded, $G_i$ for $i\in[l]_{+}$ are twice continuously differentiable, and $j_k\in[l]_{+}$ for each $k\in\mathcal{S}$. Hence, there exist constants $\widehat{c}_1>0$ and $\widehat{c}_2>0$ such that 
$\|[D^2G_{j_k}(x^k)v^k]^*u^k\|\le \widehat{c}_1\|v^k\|$ and $\|(\mathcal{F}^*\xi_2^k)(\mathcal{G}^*v^k)^{\top} \|\le\widehat{c}_2\|v^k\|$. Moreover, $\mathcal{G}(Q^k)^{\top}\mathcal{F}^*\xi_2^k=\nabla g(x^k)\xi_2^k$ holds by \eqref{Jac-gxk}. Combining these with \eqref{uk-bound} and the inclusion \eqref{sk-def}, we obtain that for all $k\in\mathcal{S}$,
\[
\|s^k\|^2\le\left[(b_{\max}+\overline{\beta}+\widehat{c}_1)^2
+(b_{\max}+\overline{\beta}+2\overline{c})^2+\widehat{c}_2^2
\right]\|v^k\|^2.
\]
Note that $H(x^k,v^k,Q^k)-z^k=G(x^k,v^k)-z^k$ for all $k\in\mathcal{S}$. It follows from the first inequality of \eqref{inexact-cond2} that $\|H(x^k,v^k,Q^k)-z^k\|\le a_k\|v^k\|^2\le \widehat{c}_3\|v^k\|$ for all $k\in\mathcal{S}$ with some $\widehat{c}_3>0$, where the second inequality is due to $a_{k}\le a_{\rm max}$ and the boundedness of $\{v^k\}_{k\in\mathbb{N}}$ by Lemma \ref{bound-lemma}. Together with the above inequality and the inclusion \eqref{temp-inclusion}, we get the conclusion for $\widehat{c}=\big[
1+(b_{\max}+\overline{\beta}+\widehat{c}_1)^2
+(b_{\max}+\overline{\beta}+2\overline{c})^2
+\widehat{c}_2^2+\widehat{c}_3^2\big]^{1/2}$.
\end{proof}

Combining Propositions \ref{prop-Xicbar}-\ref{rebound} with the sufficient-decrease inequality \eqref{temp-ineq51}, and using $x^{k+1} = x^k$ for all $k \notin \mathcal{S}$ yields the following convergence result. The proof follows the line of argument of He et al. \cite[Theorem 5.2]{He2026}, with $\Theta$ and $\Xi_{\widetilde{c}}$ therein replaced by $\Theta_{\overline{\rho}}$ and $\Xi_{\overline{c}}$, respectively, Combining Propositions \ref{prop-Xicbar}-\ref{rebound} with the sufficient-decrease inequality \eqref{temp-ineq51}, and using $x^{k+1} = x^k$ for all $k \notin \mathcal{S}$ yields the following convergence result. The proof follows the line of argument of He et al. \cite[Theorem 5.2]{He2026}, with $\Theta$ and $\Xi_{\widetilde{c}}$ therein replaced by $\Theta_{\overline{\rho}}$ and $\Xi_{\overline{c}}$, respectively, and is therefore omitted. 
\begin{theorem}\label{converge2-theorem}
Suppose that Assumption \ref{ass1} holds, and that $\Xi_{\overline{c}}$ satisfies the KL property on the set $\mathcal{W}^*$. Then $\sum_{k=0}^{\infty}\|x^{k+1}-x^k\|<\infty$, and consequently, the sequence $\{x^k\}_{k\in\mathbb{N}}$ converges to a KKT point of \eqref{Rcprob}.
\end{theorem}

The KL assumption in Theorem \ref{converge2-theorem} holds, in particular, if $f, g, \vartheta, h$ and $K$ are definable in a common o-minimal structure over the real field and $\mathcal{M}$ is a smooth semi-algebraic manifold. Indeed, under these conditions, $\Xi_{\overline{c}}$ is definable and hence satisfies the KL property (see, e.g., Attouch et al. \cite[Section 4]{Attouch2010}).
\section{Numerical experiments}\label{sec6}

To evaluate the performance of Algorithm \ref{PenAl}, we first describe its implementation details, including the solution of the subproblems and the parameter settings.
\subsection{Implementation details of Algorithm \ref{PenAl}}\label{sec6.1}

The separable structure of subproblem \eqref{Esubprob} motivates us to adopt the semi-proximal ADMM in Fazel et al. \cite[Appendix B]{Fazel2013} and its accelerated variant in Sun et al. \cite{Sun2025} as subproblem solvers. Each iteration of these methods involves only low-cost operations. Under the conditions stated in Corollaries \ref{aADMM-oracle} and \ref{sPADMM-oracle}, the resulting algorithms retain the corresponding oracle complexity guarantees, with the uniform error bound condition required only for the semi-proximal ADMM. 
\subsubsection{(Accelerated) semi-proximal ADMM}\label{sec6.1.1}

Fix any $k\in\mathbb{N}$. For any given $\sigma_k>0$, the augmented Lagrangian of \eqref{Esubprob} has the form  
\begin{equation}\label{alafunction}
\mathcal{L}_{\sigma_k}(v,z,\lambda):=\ell_k(v)+\frac{\beta_k}{2}\|v\|^2+\delta_{T_{\!x^k}\mathcal{M}}(v)+\psi_{\rho_k}(z)+\langle \lambda,G(x^k,v)-z\rangle+\frac{\sigma_k}{2}\|G(x^k,v)-z\|^2. 
\end{equation}
When $\mathcal{M}=\mathbb{X}$, by Remark \ref{remark1-Alg}, the above augmented Lagrangian is simplified to 
\begin{equation}\label{alafunction1}
\begin{aligned}
\mathcal{L}_{\sigma_k}(v,z,\lambda)&:=\ell_k(v)+({\beta_k}/{2})\|v\|^2+\vartheta(x^k+v)+\rho_k{\rm dist}(z,K)\\
&\quad+\langle\lambda,\ell_g(x^k+v;x^k)-z\rangle +({\sigma_k}/{2})\|\ell_g(x^k+v;x^k)-z\|^2. 
\end{aligned}
\end{equation}
The iteration steps of the semi-proximal ADMM for solving \eqref{Esubprob} are described as follows. 
\begin{algorithm}[H]
\begin{algorithmic}[1]
\renewcommand{\thealgorithm}{A}
\caption{\label{sPADMM}{\bf(sPADMM for solving subproblem \eqref{Esubprob})}}
\State Input: $k\in\mathbb{N},\sigma_k>0,\gamma_k=\sigma_k\|\mathcal{A}_k\|^2$, and an initial $w^{k,0}=(v^{k,0},z^{k,0},\lambda^{k,0})\in T_{\!x^k}\mathcal{M}\times\mathbb{Z}\times\mathbb{Z}$.
		
\For {$l=0,1,2,\ldots$}
		
\State $v^{k,l+1}:=\mathop{\arg\min}_{v\in\mathbb{X}}\mathcal{L}_{\sigma_k}(v,z^{k,l},\lambda^{k,l})+\frac{1}{2}\|v-v^{k,l}\|_{\gamma_k\mathcal{I}-\sigma_k\mathcal{A}_k^*\mathcal{A}_k}^2$.
		
\State $       z^{k,l+1}:=\mathop{\arg\min}_{z\in\mathbb{Z}}\mathcal{L}_{\sigma_k}(v^{k,l+1},z,\lambda^{k,l})$. 
		
\State $\lambda^{k,l+1}:=\lambda^{k,l}+\sigma_k(G(x^k,v^{k,l+1})-z^{k,l+1})$.
\EndFor	
\end{algorithmic}
\end{algorithm}
\begin{remark}\label{remark-admm}
{\bf(a)} Since $G(x^k,v)=\mathcal{A}_kv+G(x^k,0)$, rearranging the terms appearing in $\mathcal{L}_{\sigma_k}(\cdot,z^{k,l},\lambda^{k,l})$ gives $v^{k,l+1}=\frac{1}{\beta_k+\gamma_k}\mathcal{P}_{T_{\!x^k}\mathcal{M}}(u^{k,l})$ with $u^{k,l}=\gamma_kv^{k,l}+\sigma_k\mathcal{A}_k^*(z^{k,l}-\sigma_k^{-1}\lambda^{k,l}-G(x^k,v^{k,l}))-\nabla\!f(x^k)-\zeta^k$ for each $l\in\mathbb{N}$. For many common manifolds $\mathcal{M}$, the mapping $\mathcal{P}_{T_{\!x^k}\mathcal{M}}(\cdot)$ has an explicit expression. 
	
\noindent
{\bf(b)} From the definitions of $\mathcal{L}_{\sigma_k}(v^{k,l+1},\cdot,\lambda^{k,l})$ and $\psi_{\rho_k}(\cdot)$, it follows that, for each $l\in\mathbb{N}$, 
\begin{equation}\label{prox-psik}
 \resizebox{\linewidth}{!}{$z^{k,l+1}=\mathcal{P}_{\!\sigma_k^{-1}\psi_{\rho_k}}\big(G(x^k,v^{k,l+1})+\sigma_k^{-1}\lambda^{k,l}\big)=\begin{pmatrix}
\mathcal{P}_{\sigma_k^{-1}\vartheta}\big(x^k+v^{k,l+1}+\sigma_k^{-1}\lambda^{k,l}_1\big)\\
\mathcal{P}_{\frac{\rho_k}{\sigma_k}\mathrm{dist}(\cdot,K)}\big(g(x^k)+g'(x^k)v^{k,l+1}+\sigma_k^{-1}\lambda^{k,l}_2\big)
\end{pmatrix}.
$}
\end{equation}
Then $z^{k,l+1}$ is computed explicitly when $\mathcal{P}_{\sigma_k^{-1}\vartheta}(\cdot)$ and $\mathcal{P}_{\frac{\rho_k}{\sigma_k}{\rm dist}(\cdot,K)}(\cdot)$ have closed-form expressions. 
	
\noindent
{\bf(c)} For each $l\in\mathbb{N}$, by the first equality of \eqref{prox-psik}, we have $\lambda^{k,l}+\sigma_k(G(x^k,v^{k,l+1})-z^{k,l+1})\in\partial\psi_{\rho_k}(z^{k,l+1})$ which, by Step 5 of Algorithm \ref{sPADMM}, implies that $\lambda^{k,l+1}\in\partial\psi_{\rho_k}(z^{k,l+1})$. Note that $v^{k,l+1}\in T_{\!x^k}\mathcal{M}$. Then, for each $l\in\mathbb{N}_{+}$, $w^{k,l}\!:=(v^{k,l},z^{k,l},\lambda^{k,l})\in T_{\!x^k}\mathcal{M}\times\mathbb{Z}\times\partial\psi_{\rho_k}(z^{k,l})$. Moreover, from Fazel et al. \cite[Appendix B]{Fazel2013}, the sequence $\{w^{k,l}\}_{l\in\mathbb{N}}$ converges to a KKT point of \eqref{Esubprob}. Together with the definition of $\mathcal{R}_k$ in \eqref{KKT-res}, it follows that $\lim_{l\to\infty}\mathcal{R}_{k}(w^{k,l})=0$. Consequently,  $\lim_{l\to\infty}v^{k,l}=\overline{v}^k$. Thus, there exists an index $\widehat{l}\in\mathbb{N}_{+}$ such that, for every $l\ge\widehat{l}$, $(v^{k,l},z^{k,l},\lambda^{k,l})\in T_{\!x^k}\mathcal{M}\times\mathbb{Z}\times\partial\psi_{\rho_k}(z^{k,l})$  satisfies \eqref{inexact-cond1} if $\overline{v}^k=0$; otherwise, it satisfies \eqref{model-descent}-\eqref{inexact-cond2} due to the continuity of $\widehat{\Theta}_{\rho_k,\beta_k}(x^k+\cdot;x^k)$ and $\widehat{\Theta}_{\rho_k,\beta_k}(x^k+\overline{v}^k;x^k)< \widehat{\Theta}_{\rho_k,\beta_k}(x^k;x^k)$.  
\end{remark}

The accelerated pADMM proposed by Sun et al. \cite{Sun2025} is applied to solve subproblem \eqref{Esubprob} as follows.
\begin{algorithm}[H]
\begin{algorithmic}[1]
\renewcommand{\thealgorithm}{B}
\caption{\label{aADMM}{\bf(Accelerated pADMM for solving subproblem \eqref{Esubprob})}}
\State Input: $k\in\mathbb{N},\sigma_k>0,\gamma_k=\sigma_k\|\mathcal{A}_k\|^2$, parameters $\alpha\ge 2$ and $\varrho\in(0,2]$, and a starting \hspace*{1.1cm} point $\overline{w}^{k,0}=(\overline{v}^{k,0},\overline{z}^{k,0},\overline{\lambda}^{k,0})\in T_{\!x^k}\mathcal{M}\times\mathbb{Z}\times\mathbb{Z}$. Set $\widehat{w}^{k,0}:=\overline{w}^{k,0}$.
		
\For {$l=0,1,2,\ldots$}
		
\State $          z^{k,l}:=\mathop{\arg\min}_{z\in\mathbb{Z}}\mathcal{L}_{\sigma_k}(\overline{v}^{k,l},z,\overline{\lambda}^{k,l})$.
		
\State $\lambda^{k,l}:=\overline{\lambda}^{k,l}+\sigma_k(G(x^k,\overline{v}^{k,l})-z^{k,l})$.
		
\State $v^{k,l}:=\mathop{\arg\min}_{v\in\mathbb{X}}\mathcal{L}_{\sigma_k}(v,z^{k,l},\lambda^{k,l})+\frac{1}{2}\|v-\overline{v}^{k,l}\|_{\gamma_k\mathcal{I}-\sigma_k\mathcal{A}_k^*\mathcal{A}_k}^2$.
		
\State $\widehat{w}^{k,l+1}=(1-\varrho)\overline{w}^{k,l}+\varrho w^{k,l}$.
		
\State $\overline{w}^{k,l+1}=\overline{w}^{k,l}+\frac{\alpha}{2(l+\alpha)}(\widehat{w}^{k,l+1}-\overline{w}^{k,l})+\frac{l}{l+\alpha}(\widehat{w}^{k,l+1}-\widehat{w}^{k,l})$.
\EndFor	
\end{algorithmic}
\end{algorithm}
Remark \ref{remark-admm}(a) and (b) likewise apply to the iterates
$w^{k,l}:=(v^{k,l},z^{k,l},\lambda^{k,l})$ generated by
Algorithm \ref{aADMM}. By Sun et al.  
\cite[Corollary 3.5]{Sun2025}, the sequence
$\{w^{k,l}\}_{l\in\mathbb{N}}$ converges to a KKT point of
\eqref{Esubprob}. Hence, Remark \ref{remark-admm}(c) applies to
Algorithm \ref{aADMM} as well. Hereafter, we refer to Algorithm \ref{PenAl} equipped with Algorithm \ref{sPADMM} and Algorithm \ref{aADMM} as iPLNEPsp and iPLNEPac, respectively, and use iPLNEP to refer to the two variants collectively.

\subsubsection{Setting of parameters}\label{sec6.1.2}

Preliminary tests indicate that the values of \(\rho _{0}\) and \(\beta _{0}\) can substantially affect the performance of Algorithm \ref{PenAl}. Since $\rho_0$ determines the penalty weight assigned to constraint violation, we choose it based on the magnitude of the objective function at the initial point $x^0$:
\begin{equation}\label{set-rho0}
\rho_0:=\max\big\{10^{-2},\min\{\gamma_0|\Theta(x^0)|,10^4\}\big\}\ \ {\rm with}\ \ 
\gamma_0\in(0,1],
\end{equation}
where the lower and upper bounds prevent $\rho_0$ from being excessively small or large. The problem-specific values of \(\gamma_0\) are reported in 
Table~\ref{tab:exp-params}. The parameter $\beta_0$ controls the degree of proximal regularization in the first subproblem. Preliminary tests indicate that larger $\beta_0$ tends to yield solutions with higher objective values. We therefore set
\begin{equation}\label{set-beta0}
\beta_0:=\max\big\{10^{-8},\min\{\tau_0L_{\nabla\! f(x^0)},10^8\}\big\}\ \ {\rm with}\ \  \tau_0\in(0,1],
\end{equation}
where $L_{\nabla\!f(x^0)}$ is an estimate of the local Lipschitz constant of \(\nabla\!f\) near \(x^0\). The problem-specific values of \(\tau_0\) are also reported in Table~\ref{tab:exp-params}.

With a fixed initial $\rho_0$, Figure \ref{keans_fig}(a) shows that the penalty parameter update factor $\tau$ has little effect on the final objective value, whereas values close to $1$ lead to longer running times. We therefore choose $\tau=1.1$ as a compromise between computational efficiency and
conservative penalty updates. For the parameters $a_k$ and $b_k$ in Step \ref{alg:step4}, we set $a_k:=b_k^2$ and $b_k:=\max\big\{\frac{b_0}{(k+1-N_f)^{1/2}},1.0\big\}$, 
where $N_f$ denotes the cumulative number of unsuccessful iterations up to iteration $k$. This choice accounts for the different scalings of $\|v^k\|^2$ and $\|v^k\|$ in the inexactness criterion \eqref{inexact-cond2}. As shown in Figure \ref{keans_fig}(b), $b_0$ has little effect on the final objective value, while smaller values tend to increase the running time. Accordingly, we set $b_0=10^3$.
\begin{figure}[htbp]
\centering
\includegraphics[width=1.05\textwidth]{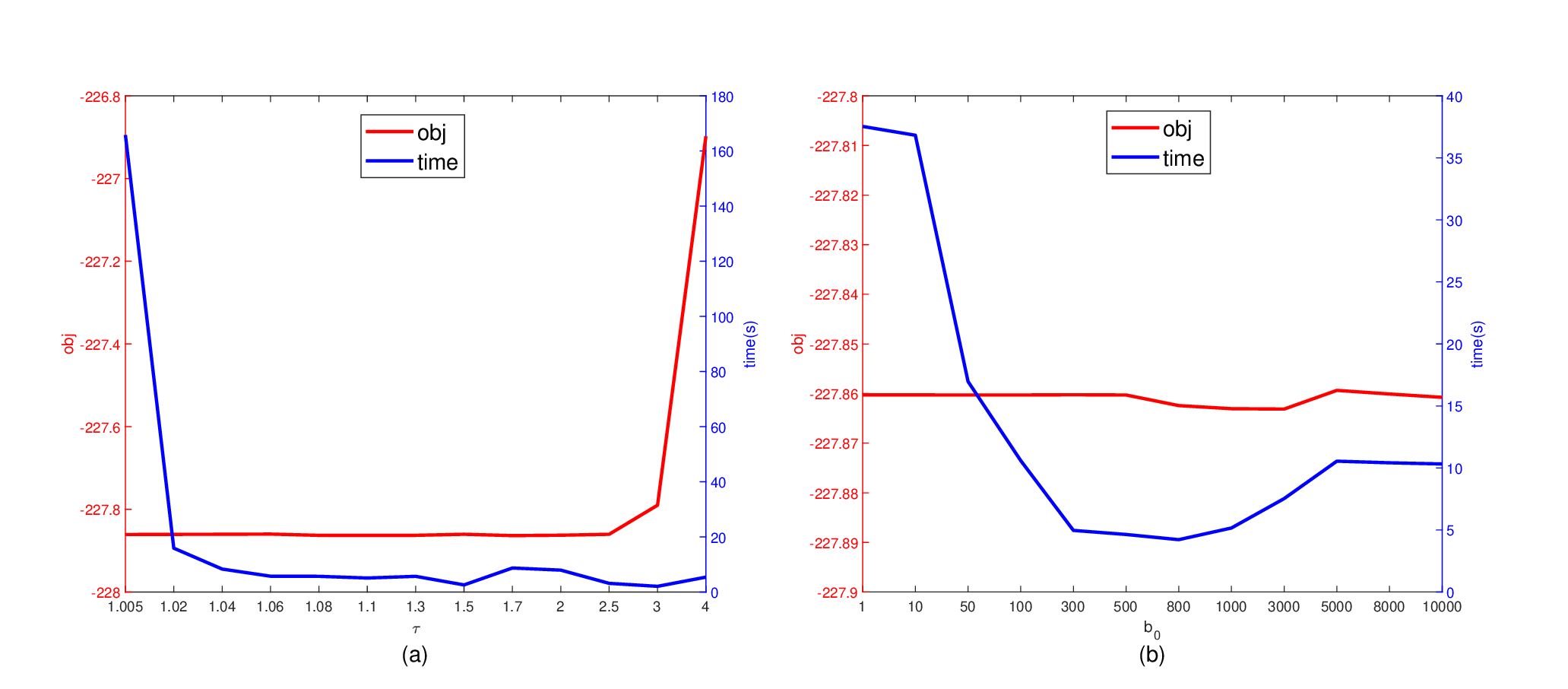}
\caption{Objective values and running times of Algorithm \ref{PenAl} for different $\tau$ in (a) and different $b_0$ in (b) when solving the test problem in Section \ref{sec6.5} on the Iris data.}
\label{keans_fig}
\end{figure}

Theorem \ref{complexity-theorem} and Lemma \ref{lemma1-complexity}(i) show
that a smaller $\overline{\gamma}$ leads to a larger complexity constant.
However, it also makes the condition ${\rm pred}_k<\overline{\gamma}\|v^k\|^2$ less likely to be satisfied, which can reduce the number of unsuccessful iterations and the associated computational cost. Balancing these effects, we set
$\overline{\gamma}=10^{-2}$. As illustrated in Figure
\ref{ex_inc_fig}(b), Step \ref{alg:step5} of Algorithm \ref{PenAl} is seldom invoked in practice, suggesting that $\varepsilon_0$ and $\varsigma$
have limited impact on the overall performance. We thus set
$\varepsilon_0=10^{-2}$ and $\varsigma=0.5$. Finally, the parameters $\varsigma_1$ and $\varsigma_2$ in 
Step~\ref{alg:step6} control the changes in $\beta_k$
across successive iterations. Values closer to $1$
lead to smaller changes, which tend to improve
stability at the expense of longer running times,
whereas values farther from $1$ may reduce the running
time but lead to less stable behavior. Balancing these effects, we set $\varsigma_1=1.05$ and $\varsigma_2=0.9$.

To summarize, the parameter values common to all test problems are specified as 
\begin{equation}\label{parameter}
\begin{aligned}
 \beta_{\rm min}=10^{-10},\,a_{\rm max}=10^8,\,b_{\rm max}=10^4,\,\tau=1.1,\,\varepsilon_0=10^{-2},\varsigma=0.5,\\  \eta_1=0.15,\,\eta_2=0.85,\,\varsigma_1=1.05,\,\varsigma_2=0.9,\,\overline{\gamma}=10^{-2}.\qquad\qquad
\end{aligned}  
\end{equation}
The parameters $a_k$ and $b_k$ appearing in Step 4 are set according to the following rule
\[
  a_k = b_k^2\ \ {\rm with}\ \ b_k =\max\Big\{\frac{b_0}{(k+1-N_{\!f})^{1/2}}, 1.0\Big\}\ \ {\rm for}\ \ b_0=10^3.
\]
For each $k\in\mathbb{N}_+$, when Algorithm \ref{sPADMM} or \ref{aADMM} is used to solve \eqref{Esubprob}, the penalty parameter $\sigma_k$ is updated as follows
\begin{equation}\label{sigma-update}
\sigma_k=\left\{\begin{array}{cl}
\min\{1.01\sigma_{k-1},10^4\}\ 
&\text{if \eqref{model-descent}-\eqref{inexact-cond2} hold\ and}\ 
a_k\|v^k\|^2
\le 2\|G(x^k,v^k)-z^{k,l}\|,\\[2mm]
\sigma_{k-1}&\text{otherwise}.
\end{array}\right.
\end{equation}
The parameters of Algorithm~\ref{aADMM} are fixed at $\alpha=10^4$ and $\varrho=1.9$. Table~\ref{tab:exp-params} summarizes the problem-specific parameter settings, including the initial penalty parameter $\sigma_0$ in \eqref{sigma-update}, as well as the retraction and norm $\vert\!\|\cdot\|\!\vert$ used for each test problem. Here, $\|\cdot\|_{1,\mathrm{entry}}$ denotes the entrywise $\ell_1$-norm of a matrix.

\begin{table}[htbp]
\centering
\caption{Problem-specific parameter settings, retractions, and norms for iPLNEP}
\label{tab:exp-params}
\begin{tabular}{lccccc}
\hline
Test problem & $\gamma_0$ & $\tau_0$ & $\sigma_0$ & retraction $R$ & $\vert\!\|\cdot\|\!\vert$\\
\hline
Row-sparse PCA
& $10^{-1}$ & $10^{-4}$  & $0.5$ & -- & $\|\cdot\|_{1,{\rm entry}}$\\
Community detection
& $10^{-3}$ & $10^{-3}$  & $5.0$ & QR & $\|\cdot\|_{1,{\rm entry}}$\\
$k$-means clustering
& $0.2$ & $0.05$ & $2.0$ & QR & $\|\cdot\|_{1,{\rm entry}}$\\
Robust SVM
& $0.5$ & $10^{-3}$  & $0.1$  &-- & $\|\cdot\|_2$\\
Synthetic nonconvex SOCP
& $0.1$ & $0.1$ & $20.0$ & Projection & $\|\cdot\|_1$\\
\hline
\end{tabular}
\end{table}
\subsection{Numerical illustration of theoretical results}\label{sec6.2}

Let $A$ be a real symmetric matrix generated in MATLAB using the commands \lstinline|A = randn(n)| and \lstinline|A = (A+A')/2|. Let $u,s,w\in\mathbb{R}^n$ be randomly generated linearly independent vectors. Define
$\mathcal{M}:={x\in\mathbb{R}^n\mid u^{\top}x=0,\ |x|=1}$.
We consider the following test problem:
\begin{equation}\label{example5}
\min_{x\in \mathcal{M}}\Big\{\frac{1}{2}x^{\top}Ax +\mu\|x\|_1\ \  {\rm s.t.}\ \ (u^{\top}x)^2(s^{\top}x)=0,\, (u^{\top}x)^2-w^{\top}x\le 0,\,w^{\top}x\le 0\Big\},
\end{equation}
where $\mu>0$ is the regularization parameter. If the sphere constraint is removed from $\mathcal M$, the remaining constraints coincide with those in Andreani et al. \cite[Example 1]{Andreani2026}, for which the MFCQ fails. We incorporate the sphere constraint into $\mathcal M$ and use a composite objective to ensure that the example falls within our problem class. Since $f$ and $g$ are polynomial and $\mathcal M$ is compact, Assumption~\ref{ass1}(i) holds. 

Figure~\ref{fig-Penobj}(a) shows that $\rho_k$ becomes constant from iteration $\overline{k}\le300$ onward. This behavior is consistent with the boundedness of the penalty parameter sequence assumed in Assumption~\ref{ass1}(ii), even though the extended CQ in Definition~\ref{def-ECQ} fails at the accumulation points. Establishing weaker conditions that guarantee boundedness of the penalty parameter sequence is an interesting direction for future research. Figure~\ref{fig-Penobj}(b) shows the convergence of the penalty objective values $\{\Theta_{\rho_k}(x^k)\}_{k\in\mathbb{N}}$, in agreement with Theorem~\ref{converge-theorem1}(i).

\begin{figure}[htbp]
\centering
\includegraphics[width=1.0\textwidth]{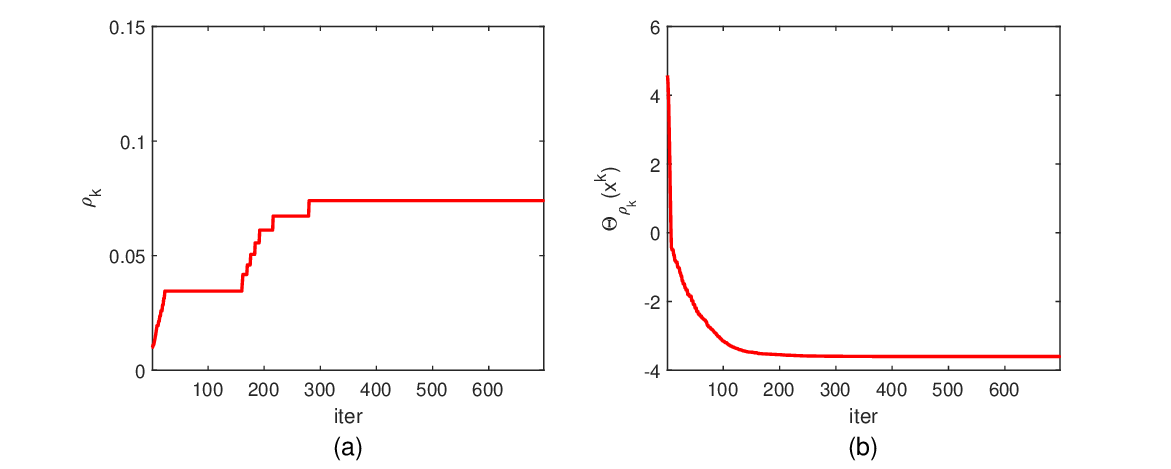}
\caption{Evolution of the penalty parameter $\rho_k$ and penalty function value $\Theta_{\rho_k}(x^k)$}
\label{fig-Penobj}
\end{figure}
\begin{figure}[H]
\centering
\includegraphics[width=1.0\textwidth]{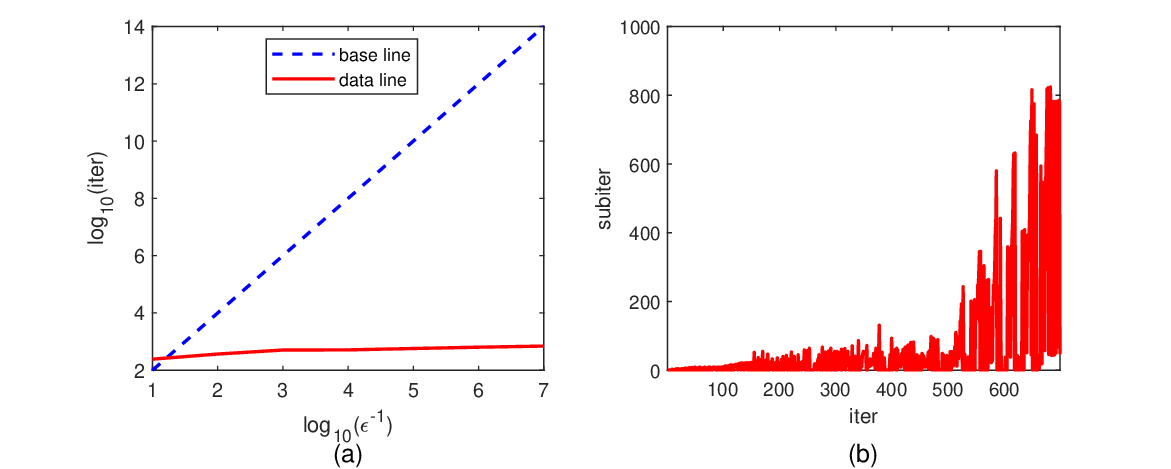}
\caption{(a) Number of outer iterations versus the accuracy tolerance $\epsilon$. (b) Number of inner iterations of Algorithm~\ref{aADMM} at each outer iteration of Algorithm~\ref{PenAl}.}
\label{ex_kkt_fig}
\end{figure}

Figure~\ref{ex_kkt_fig}(a) shows that the number of outer iterations increases as the accuracy tolerance $\epsilon$ decreases, but grows much more slowly than the worst-case $O(\epsilon^{-2})$ complexity bound would suggest. This behavior is consistent with Theorem~\ref{complexity-theorem}, which provides a worst-case upper bound rather than a prediction of the observed iteration count. Figure~\ref{ex_kkt_fig}(b) reports the number of inner iterations of Algorithm~\ref{aADMM} at each outer iteration. The inner iteration counts remain relatively low during the first $500$ outer iterations and become more variable thereafter, with peaks of approximately $800$.

The next four subsections present the numerical results for iPLNEP and the comparison methods on the test problems in Examples~\ref{Example1}--\ref{Example3}. All experiments were conducted using MATLAB R2024b on a 64-bit Windows laptop equipped with an Intel Core i7-9750H CPU running at 2.60 GHz and 16 GB of RAM.

\subsection{Experiments on row-sparse PCA}\label{sec6.3}

We apply iPLNEP to the problem in Example \ref{Example1} with $\varphi(X)=\|X\|_{2,1}$ and $u=0$. This yields an instance of \eqref{Rcprob} with $h\equiv 0, g(X)=X^{\top}X-I_p,\mathcal{M}=\mathbb{X}=\mathbb{R}^{n\times p}$ and $K=\{0\}$. In this setting, iPLNEP provides a retraction-free approach for nonsmooth optimization with orthogonality constraints. The linear operator $\mathcal{A}_k$ in Remark \ref{remark-admm}(a) is 
$\mathcal{A}_kH=(X^k)^\top H+H^\top X^k$ for $H\in\mathbb{R}^{n\times p}$. 

We compare iPLNEP with two single-loop retraction-free approaches (PenCPG in Xiao et al. \cite{Xiao2021} and LSALM in Zhu et al. \cite{Zhu2026}) and a single-loop retraction-based approach (RADMM in Li et al. \cite{Li2025}). PenCPG is a method specifically tailored to row-sparse PCA. 
The data matrix $A$ is synthetically generated following the generation procedure of PenCPG in Xiao et al. \cite{Xiao2021}, as implemented in MATLAB with $m=100$:
\begin{lstlisting}
      B = randn(m,n); B = B-repmat(mean(B,1),m,1);            
      B = B/max(vecnorm(B,2,1)); B = B'/sqrt(eigs(B'*B,1)); 
      A = B*B'. 
\end{lstlisting}
 All methods start from the same feasible initial point $X^0$, generated by \lstinline|qr(randn(n,p),0)|.

For PenCPG, we use the default settings in the source code, except that the thresholds $0.5$ and $0.05$ in \texttt{safe\_proj} are replaced by $0.01$ and $0.001$, respectively, and the post-processing step is disabled. For LSALM, we set $\rho=0.5$, $\alpha=0.2$, $\beta=2$, $\epsilon=10^{-8}$, $r=2$, $\lambda=1.25$, $R_X^{\rm op}=10$, and $R_Y=5$, following the notation of Zhu et al. \cite{Zhu2026}. For RADMM, we set the smoothing parameter $\gamma=10^{-8}$, the penalty parameter $\rho=50$, and the fixed step size $\eta=0.025$. Since the original implementations of LSALM and RADMM use $\varphi(X)=\|X\|_1$, we modify them to use $\varphi(X)=\|X\|_{2,1}$ so that all methods solve the same problem. 

The infeasible methods iPLNEP, PenCPG, and LSALM terminate when either the maximum number of iterations is reached or both the constraint violation condition $\|(X^k)^{\top}X^k-I_p\|_F<10^{-5}$ and a method-specific stopping criterion are satisfied. Specifically, the latter requires $\|V^k\|_F<10^{-4}$ for iPLNEP and $\|X^k-X^{k-1}\|_F<10^{-4}$ for PenCPG and LSALM. The maximum numbers of iterations are set to $k_{\rm max}=5000$ for iPLNEP, $k_{\rm max}=3000$ for PenCPG, and $k_{\rm max}=30000$ for LSALM. In addition, iPLNEP terminates if  $\max\{\|(X^k)^{\top}X^k-I_p\|_F,\varepsilon_k\}<10^{-5}$. Since RADMM preserves feasibility, it terminates when either the maximum number of iterations, $40000$, is reached or  $\min\big\{\frac{\|X^k-Z^k\|_{F}}{\max\{1,\|X^k\|_F,\|Z^k\|_F\}},\|X^k-X^{k-1}\|_F\big\}< 10^{-4}$. The same stopping criterion was used for RADMM by Zhu et al. \cite{Zhu2026}, who define $Z^k$ in their paper.

\begin{table}[h]
\renewcommand{\arraystretch}{1.2}
\centering
\caption{Numerical comparison of five methods for solving row-sparse PCA problems}
\label{rSPCA-tab}
\scalebox{0.68}[0.68]{
\begin{tabular}{c|c|cccccc} 
\hline    
\multicolumn{2}{c|}{$(n,p,\mu)$}& (2000,6,0.035)& (2000,9,0.035)&(2000,12,0.035)&(2500,10,0.030)&(2500,10,0.034)&(2500,10,0.038)\\[0.5mm]
\hline
\multirow{6}*{iPLNEPsp} 
&obj& 0.0788& 0.0804& -0.4348& -0.2645& 0.2674& 0.1980\\ 
&rspar& 0.9970& 0.1135& 0.0181& 0.0383& 0.1106& 0.9960\\ 
&time (s)& 0.43& 1.20& 1.78& 1.43& 1.96& 0.80\\ 
&iter& 426.6& 763.0& 952.0& 814.9& 980.0& 551.5\\ 
&initer& 1.2& 1.7& 1.8& 1.6& 2.0& 1.3\\ 
&infeas& 2.31e-6& 6.79e-7& 7.20e-7& 7.57e-7& 6.48e-7& 2.16e-6\\ 
\hline  
\multirow{6}*{iPLNEPac} 
&obj& 0.0819& 0.0806& -0.4348& -0.2645& 0.2675& 0.1983\\ 
&rspar& 0.9970& 0.1148& 0.0179& 0.0384& 0.1118& 0.9960\\ 
&time (s)& 0.40& 1.24& 1.77& 1.47& 2.18& 0.92\\ 
&iter& 389.4& 709.5& 828.5& 709.1& 918.5& 595.2\\ 
&initer& 1.4& 2.2& 2.5& 2.1& 2.6& 1.6\\ 
&infeas& 3.83e-6& 2.64e-8& 3.31e-8& 6.94e-8& 5.17e-8& 2.08e-6\\ 
\hline  
\multirow{5}*{PenCPG} 
&obj& 0.0882& 0.0805& -0.4348& -0.2645& 0.2678& 0.2081\\ 
&rspar& 0.9970& 0.1150& 0.0180& 0.0384& 0.1120& 0.9960\\ 
&time (s)& 0.03& 0.46& 0.49& 0.56& 0.87& 0.05\\ 
&iter& 43.1& 851.2& 758.4& 788.9& 1212.7& 66.2\\ 
&infeas& 1.21e-6& 7.02e-6& 3.35e-6& 3.59e-6& 7.24e-6& 1.85e-6\\ 

\hline
\multirow{5}*{LSALM} 
&obj& 0.0871& 0.0808& -0.4348& -0.2645& 0.2675& 0.2032\\ 
&rspar& 0.9970& 0.1154& 0.0181& 0.0384& 0.1103& 0.9960\\ 
&time (s)& 0.44& 1.52& 1.48& 1.54& 1.92& 0.86\\ 
&iter& 583.9& 1818.5& 1642.3& 1610.4& 1975.2& 891.8\\ 
&infeas& 7.15e-6& 5.31e-6& 4.37e-6& 3.80e-6& 5.09e-6& 6.57e-6\\ 

\hline
\multirow{4}*{RADMM} 
&obj& 0.0813& 0.0823& -0.4334& -0.2618& 0.2705& 0.4372\\ 
&rspar& 0.9962& 0.1112& 0.0183& 0.0383& 0.1067&  0.7145\\ 
&time (s)& 9.03& 9.40& 9.57& 9.61& 12.50& 24.84\\ 
&iter& 10623.5& 9255.9& 8576.0& 7982.2& 10312.5& 20160.3\\ 
\hline
\end{tabular}}
\end{table}

Table \ref{rSPCA-tab} reports the average results over $\textbf{20}$ independent runs for each instance. The reported quantities include the final objective value $(\texttt{obj})$, row sparsity $(\texttt{rspar})$, CPU time $(\texttt{time})$, number of outer iterations $(\texttt{iter})$, constraint violation $(\texttt{infeas})$, and average number of inner iterations per subproblem $(\texttt{initer})$. Here, row sparsity is defined as the proportion of rows whose $\ell_2$-norms are below $10^{-5}$. As shown in Table \ref{rSPCA-tab}, iPLNEP achieves objective values comparable to those of the benchmark methods, with substantially lower CPU time than RADMM and consistently smaller constraint violations than LSALM. Its computational cost is comparable to that of LSALM, although it is generally higher than that of PenCPG. The two iPLNEP variants also achieve similar row sparsity levels, while Algorithm \ref{sPADMM} requires at most $\textbf{2}$ inner iterations on average for every tested instance. Overall, the accelerated variant offers no noticeable computational advantage over the standard semi-proximal ADMM in this experiment.

\subsection{Experiments on community detection}\label{sec6.4}

We apply iPLNEP to solve the problem in Example \ref{Example1}, where $A$ is the modularity matrix, $\varphi(X)=\Vert{}X\Vert{}_{1}$, and $u=e/\sqrt{n}$ with $e\in\mathbb{R}^n$ denoting the all-ones vector. Since the equality constraint is unchanged by a positive scaling, we equivalently set $g(X)=XX^{\top}e-e$, so that \eqref{Rcprob} is specified by $h\equiv 0,\mathcal{M}={\rm St}(n,p)$ and $K=\{0\}$. In this case, iPLNEP handles a nonsmooth optimization problem over the Stiefel manifold subject to an  equality constraint. The linear operator $\mathcal{A}_k$ and its adjoint in Remark \ref{remark-admm}(a) are given by $\mathcal{A}_kH=(H,X^{k}H^{\top}e+H(X^k)^{\top}e)$ and $\mathcal{A}_k^*(\Gamma,\xi)=\Gamma+(e\xi^{\top}+\xi e^{\top})X^k$. 

We compare iPLNEP with I-AManPG, the inexact manifold PG method proposed by Huang et al. \cite{Huang2025} by viewing
$\mathcal{F}_{u}:=\{X\in{\rm St}(n,p)\mid XX^{\top}u=u\}$ as a manifold. The modularity matrix $A$ is generated from synthetic LFR benchmark networks for community detection (see Lancichinetti et al. \cite{Lancichinetti2008}). The LFR networks are generated using the implementation of I-AManPG with $\tau_1=-2$, $\tau_2=-1$, $N=1000$, $d_{\rm ave}=20$, $d_{\rm max}=40$, $N_c=50$, and $n_c=20$; see Huang et al. \cite[Section 5.2]{Huang2025} for details. The regularization parameter $\mu$ is set to $0.15$. The two iPLNEP variants and I-AManPG are initialized at the same point 
$X^0\in\mathcal{F}_u$, generated according to the implementation of I-AManPG. I-AManPG uses the default parameter settings in its source code. We terminate iPLNEP when $\|X^k(X^k)^{\top}u-u\|_2<10^{-5}$ and either
$\|V^k\|_F<10^{-3}$ or $\varepsilon_k<10^{-5}$, or when the iteration count reaches $5000$. I-AManPG is terminated using its default stopping criterion.

\begin{table*}[h]
\renewcommand{\arraystretch}{1.2}
\centering
\caption{Numerical comparison of iPLNEP and I-AManPG for community detection}
\label{CD-tab}
\scalebox{0.685}[0.685]{
		\begin{tabular}{c|c|ccccccccc} 
			\hline    
			\multicolumn{2}{c|}{$\mu_{\rm LFR}$}& 0.0& 0.1&0.2&0.3&0.4&0.5&0.6&0.7&0.8\\
			\hline
			\multirow{6}*{iPLNEPsp} 
            &AMI&1.0000& 1.0000& 1.0000& 1.0000&1.0000&0.9998&0.9704&0.4193&0.0589\\ 
            &NMI&1.0000& 1.0000& 1.0000& 1.0000&1.0000&0.9998&0.9742&0.4656&0.1319\\ 
            &pur&1.0000& 1.0000& 1.0000& 1.0000&1.0000&0.9999&0.9812&0.5739&0.3055\\ 
            &obj&-189.9896& -168.7026& -147.2721& -125.7580&-104.3158&-83.0911&-62.1011&-44.7245&-40.9447\\ 
            &time (s)&4.47& 1.57& 1.13& 0.87&0.98&0.70&0.96&4.87&4.31\\ 
            &infeas&9.10e-6& 8.53e-6& 8.59e-6& 8.09e-6&8.44e-6&7.87e-6&8.24e-6&4.57e-6&5.91e-6\\ 
			\hline
			\multirow{6}*{iPLNEPac} 
            &AMI&1.0000& 1.0000& 1.0000& 1.0000&1.0000&0.9998&0.9618&0.4004&0.0548\\ 
            &NMI&1.0000& 1.0000& 1.0000& 1.0000&1.0000&0.9998&0.9670&0.4487&0.1275\\ 
            &pur&1.0000& 1.0000& 1.0000& 1.0000&1.0000&0.9999&0.9754&0.5592&0.2968\\ 
            &obj&-189.9896& -168.7026& -147.2722& -125.7577&-104.3156&-83.0912&-61.8683&-44.6821&-40.9393\\ 
            &time (s)&6.10& 2.03& 1.26& 1.19&1.04&1.15&1.65&7.48&3.56\\ 
            &infeas&8.64e-6& 8.87e-6& 8.47e-6& 8.38e-6&8.18e-6&7.50e-6&8.30e-6&4.32e-6&6.43e-6\\ 
			\hline  
            \multirow{5}*{I-AManPG} 
            &AMI&1.0000& 1.0000& 1.0000& 1.0000&1.0000&0.9998&0.9541&0.4031&0.0561\\ 
            &NMI&1.0000& 1.0000& 1.0000& 1.0000&1.0000&0.9998&0.9601&0.4511&0.1292\\ 
            &pur&1.0000& 1.0000& 1.0000& 1.0000&1.0000&0.9999&0.9680&0.5591&0.3048\\ 
            &obj&-189.8882& -168.3621& -147.1416& -125.6683&-104.2121&-83.0904&-61.6879&-44.4846&-40.7626\\ 
            &time (s)&0.68& 0.49& 0.53& 0.65&0.42&0.46&1.03&1.51&1.57\\
			\hline
		\end{tabular}
	}
\end{table*}

Table \ref{CD-tab} reports the average results of the three methods over the \textbf{10} available data instances for each value of $\mu_{\rm LFR}$. AMI, NMI, and purity are clustering quality metrics, whose definitions are given in Huang et al. \cite[Section 5.2.3]{Huang2025}. Higher values indicate better agreement with the ground-truth partition. As shown in Table \ref{CD-tab}, all three methods achieve identical values of these metrics, to the reported precision, for $\mu_{\rm LFR}=0.0,0.1,\ldots,0.5$. 
For $\mu_{\rm LFR}=0.6,0.7,0.8$, iPLNEPsp achieves the highest values for all three measures. I-AManPG slightly outperforms iPLNEPac in AMI and NMI for $\mu_{\rm LFR}=0.7,0.8$, while their purity values are nearly identical.
The two iPLNEP variants produce comparable objective values, both of which are lower than those obtained by I-AManPG. Although I-AManPG runs faster in most cases, all three methods solve each test in less than $\textbf{8}$ seconds. These results demonstrate that iPLNEPsp is competitive with I-AManPG in clustering quality and objective value on the tested instances.
\subsection{Experiments on $k$-means clustering}\label{sec6.5}

We apply iPLNEP to solve the problem in Example \ref{Example2} with $f(X)=-\frac{1}{2}{\rm tr}(DXX^{\top})$ and $u=e$, where $D=YY^{\top}$ with $Y\in\mathbb{R}^{N_s\times N_d}$ denoting the standardized data matrix, and $e$ denotes the all-ones vector. Here, $N_s$ and $N_d$ denote the numbers of samples and features, respectively. As shown in Carson et al. \cite{carson2017}, this problem provides an equivalent manifold reformulation of $k$-means clustering. It is a special case of \eqref{Rcprob} with $\vartheta\equiv 0=h,g(X)=(X,XX^{\top}e-e),\mathcal{M}={\rm St}(N_{s},N_{c})$ and $K=\mathbb{R}_{+}^{N_s\times N_{c}}\times\{0\}$, where $N_c$ denotes the number of clusters. In this setting, iPLNEP provides a Riemannian nonsmooth penalty framework for optimization problems with nonnegativity and orthogonality constraints. The mapping $\mathcal{A}_k$ in Remark \ref{remark-admm}(a) is given by $\mathcal{A}_kH=\big(H,\,
X^kH^{\top}e+H(X^k)^{\top}e\big)$ for $H\in\mathbb{R}^{N_s\times N_{c}}$. 

We evaluate the performance of iPLNEP against RALM proposed in Andreani et al. \cite{Andreani2026} using $Y$ constructed from a dataset in the UCI Machine Learning Repository (see Kelly et al. \cite{Kelly2023uci}). For RALM, we use the default parameter settings and initial point provided in its source code. We initialize iPLNEP at $X^0$ using \texttt{[X0,\textasciitilde{}] = eigs(D,p)}, and test RALM with both its default initialization and the same initial point as iPLNEP. As shown in Table~\ref{kmeans-tab}, using the same initial point as iPLNEP generally increases the CPU time of RALM, although exceptions occur on some datasets. In particular, RALM with this initialization exceeds the time limit of $\textbf{2000}$ seconds on Ecoli. We terminate iPLNEP when
$(\|\min(X^k,0)\|_F^2+\|X^k(X^k)^{\top}e-e\|_2^2)^{1/2}<10^{-5}$
and either $\|V^k\|_F<10^{-4}$ or $\varepsilon_k<10^{-5}$, or when the iteration count reaches $5000$. RALM is terminated according to its default stopping condition.

Table~\ref{kmeans-tab} shows that, compared with RALM$_{\rm def}$, iPLNEP${\rm sp}$ achieves higher NMI and ACC values on six datasets, identical values at the reported precision on five datasets, and lower values on the remaining two. The improvements are particularly pronounced on the Iris, Pima diabetes, and Transfusion datasets. In terms of CPU time, iPLNEP${\rm sp}$ is faster than RALM$_{\rm def}$ on eight of the thirteen datasets, with substantial savings on Ecoli, Ionosphere, Iris, Pima diabetes, and Wine. The two iPLNEP variants yield comparable clustering quality on most datasets, while iPLNEP${\rm sp}$ is faster than iPLNEP${\rm ac}$ on twelve of the thirteen datasets. Both iPLNEP variants and RALM$_{\rm def}$ attain constraint violations below $3\times10^{-5}$, whereas RALM$_{\rm eig}$ exceeds this threshold on four datasets. Overall, iPLNEP${\rm sp}$ exhibits competitive clustering performance and computational efficiency on the tested datasets.
\begin{table}[H]
\centering
\caption{Numerical comparison of iPLNEP and RALM variants for $k$-means clustering. Each cell reports NMI/ACC (first line) and constraint violation/CPU time (s) (second line); ``--'' indicates that the 2000-second time limit was exceeded.}
\label{kmeans-tab}
\footnotesize
\setlength{\tabcolsep}{7pt}
\renewcommand{\arraystretch}{2.5}

\newcommand{\twoline}[2]{\makecell{#1\\[1.5pt]#2}}

\begin{tabular}{@{}c|c|c|c|c@{}}
\hline
\makecell[c]{Dataset\\[-1mm]\scriptsize$(N_c,N_d,N_s)$}
& iPLNEPsp & iPLNEPac
& RALM$_{\rm def}$ & RALM$_{\rm eig}$ \\
\hline

\makecell[c]{Breast cancer\\[-1mm]\scriptsize$(2,30,569)$}
& \twoline{0.5858 / 91.74}{5.66e-7 / 76.48}
& \twoline{0.5858 / 91.74}{4.96e-6 / 109.79}
& \twoline{0.5547 / 91.04}{2.20e-6 / 64.59}
& \twoline{0.5547 / 91.04}{7.90e-5 / 358.77} \\

\makecell[c]{Cloud\\[-1mm]\scriptsize$(2,10,2048)$}
& \twoline{1.0000 / 100.00}{6.21e-6 / 279.86}
& \twoline{1.0000 / 100.00}{1.91e-6 / 323.75}
& \twoline{1.0000 / 100.00}{8.38e-6 / 213.77}
& \twoline{1.0000 / 100.00}{1.91e-4 / 1545.55} \\

\makecell[c]{Ecoli\\[-1mm]\scriptsize$(8,7,336)$}
& \twoline{0.6645 / 66.37}{4.51e-6 / 7.91}
& \twoline{0.6236 / 70.24}{3.94e-6 / 13.37}
& \twoline{0.6276 / 65.77}{1.76e-5 / 477.99}
& \twoline{-- / --}{-- / --} \\

\makecell[c]{Ionosphere\\[-1mm]\scriptsize$(2,34,351)$}
& \twoline{0.1349 / 71.23}{2.80e-6 / 1.12}
& \twoline{0.1349 / 71.23}{2.86e-6 / 1.25}
& \twoline{0.1349 / 71.23}{4.07e-7 / 30.01}
& \twoline{0.1349 / 71.23}{6.37e-7 / 24.72} \\

\makecell[c]{Iris\\[-1mm]\scriptsize$(3,4,150)$}
& \twoline{0.6838 / 85.33}{9.35e-6 / 3.19}
& \twoline{0.6595 / 83.33}{2.20e-6 / 4.62}
& \twoline{0.6588 / 83.33}{1.07e-5 / 33.89}
& \twoline{0.6526 / 82.67}{1.12e-5 / 176.11} \\

\makecell[c]{Parkinsons\\[-1mm]\scriptsize$(2,22,195)$}
& \twoline{0.0011 / 70.77}{8.91e-6 / 48.80}
& \twoline{0.0133 / 75.38}{8.81e-6 / 180.69}
& \twoline{0.1367 / 75.38}{9.55e-7 / 39.49}
& \twoline{0.0047 / 54.36}{3.42e-6 / 11.52} \\

\makecell[c]{Pima diabetes\\[-1mm]\scriptsize$(2,8,768)$}
& \twoline{0.1203 / 70.96}{2.43e-6 / 3.43}
& \twoline{0.0866 / 68.88}{5.14e-6 / 14.86}
& \twoline{0.0617 / 67.45}{1.06e-6 / 133.77}
& \twoline{0.0622 / 67.45}{1.33e-5 / 1138.05} \\

\makecell[c]{Raisin\\[-1mm]\scriptsize$(2,7,900)$}
& \twoline{0.3305 / 76.22}{7.08e-6 / 62.58}
& \twoline{0.3305 / 76.22}{1.98e-6 / 87.10}
& \twoline{0.3388 / 76.78}{2.84e-6 / 93.70}
& \twoline{0.3388 / 76.78}{1.47e-5 / 693.45} \\

\makecell[c]{Seeds\\[-1mm]\scriptsize$(3,7,210)$}
& \twoline{0.7446 / 92.38}{4.23e-7 / 20.53}
& \twoline{0.7307 / 91.90}{2.86e-6 / 22.10}
& \twoline{0.7172 / 91.43}{1.42e-5 / 59.09}
& \twoline{0.7332 / 91.90}{9.09e-5 / 292.05} \\

\makecell[c]{SPECTF\\[-1mm]\scriptsize$(2,44,267)$}
& \twoline{0.0736 / 66.29}{2.51e-6 / 22.42}
& \twoline{0.0736 / 66.29}{6.74e-6 / 39.14}
& \twoline{0.0736 / 66.29}{4.40e-6 / 21.32}
& \twoline{0.0736 / 66.29}{7.41e-6 / 89.34} \\

\makecell[c]{Thyroid\\[-1mm]\scriptsize$(3,5,215)$}
& \twoline{0.5669 / 87.44}{9.14e-7 / 20.86}
& \twoline{0.5669 / 87.44}{9.73e-6 / 24.95}
& \twoline{0.5669 / 87.44}{2.46e-6 / 42.24}
& \twoline{0.5669 / 87.44}{1.54e-5 / 251.58} \\

\makecell[c]{Transfusion\\[-1mm]\scriptsize$(2,4,748)$}
& \twoline{0.0350 / 76.74}{9.95e-6 / 62.73}
& \twoline{0.0486 / 67.25}{2.67e-6 / 0.97}
& \twoline{0.0172 / 73.93}{2.83e-5 / 21.66}
& \twoline{0.0157 / 72.46}{6.94e-6 / 17.32} \\

\makecell[c]{Wine\\[-1mm]\scriptsize$(3,13,178)$}
& \twoline{0.8759 / 96.63}{3.71e-6 / 1.86}
& \twoline{0.8759 / 96.63}{6.25e-6 / 3.15}
& \twoline{0.8759 / 96.63}{3.16e-6 / 38.33}
& \twoline{0.8759 / 96.63}{4.71e-5 / 190.35} \\

\hline
\end{tabular}
\end{table}

\subsection{Experiments on nonconvex SOCPs}\label{sec6.6}

\subsubsection{DC regularized SOCPs for robust SVMs}\label{sec6.5.2}

Following Nath and Bhattacharyya \cite{Nath2007}, let $X_l$ ($l=1,2$) denote the $n$-variate random vector for class $l$, with mean $s_l$ and covariance matrix $S_l$. We seek a robust classifier that correctly classifies an observation from class $l$ with probability at least $p_l\in(0,1)$, for $l=1,2$. This leads to the following chance-constrained problem:
\[
 \min_{w\in\mathbb{R}^n,t\in\mathbb{R}}\Big\{\frac{1}{2} \|w\|_2^2\ \ {\rm s.t.}\ \ \Pr\{w^{\top}X_1 + t\ge 0\} \ge p_1, \Pr\{w^{\top}X_2 + t \le 0\} \ge p_2\Big\}. 
\]
Using the means and covariance matrices as distributional information and applying the multivariate Chebyshev inequality of Lanckriet et al. \cite[Lemma 1]{Lanckriet2002}, Nath and Bhattacharyya \cite{Nath2007} derived the following conservative SOCP formulation, where  $\kappa_l=\sqrt{\frac{p_l}{1-p_l}}$ for $l=1,2$:
\begin{equation}\label{robust-SOCP}
 \min_{w\in\mathbb{R}^n,t\in\mathbb{R}}\Big\{\frac{1}{2} \|w\|_2^2\ \ {\rm s.t.}\ \ w^{\top}s_1 + t \ge 1 + \kappa_1 \sqrt{w^{\top}S_1 w},\, -(w^{\top}s_2 + t) \ge 1 + \kappa_2 \sqrt{w^{\top}S_2 w}\Big\}.
 \end{equation}
L\'{o}pez et al. \cite{Lopez2018} further incorporated $\ell_1$- or $\ell_0$-norm regularization into \eqref{robust-SOCP} for simultaneous feature selection and classification. We instead consider a nonconvex SOCP with an $\ell_1$-$\ell_2$ DC regularizer: 
 \begin{equation}\label{dc-rSOCP}
 \begin{aligned}
 &\min_{(x_1,x_2)\in \mathbb{R}\times\mathbb{R}^{n}} \frac{1}{2}\|x_2\|_2^2 + \mu(\|x_2\|_1 - \|x_2\|_2)\\
 &\qquad\ {\rm s.t.}\ 1+\kappa_1\|D^{\top}_1x_2\|_2\le s_1^{\top}x_2+x_1\\
  &\qquad\qquad\ 1+\kappa_2\|D^{\top}_2x_2\|_2\le -s_2^{\top}x_2-x_1,
\end{aligned}
\end{equation}
where $D_1\in\mathbb{R}^{n\times m_1}$ and $D_2\in\mathbb{R}^{n\times m_2}$ are such that $S_1=D_1D_1^{\top}$ and $S_2=D_2D_2^{\top}$. Note that \eqref{dc-rSOCP} is a special case of Example \ref{Example3} with $l=2,m_1=n+1,m_2=n+1$ and
$g(x):=(s_1^{\top}x_2+x_1-1;\kappa_1D_1^{\top}x_2;-s_2^{\top}x_2-x_1-1;\kappa_2D_2^{\top}x_2)$ for $x=(x_1,x_2)\in\mathbb{R}\times\mathbb{R}^{n}$, but without the manifold constraint. In this setting, iPLNEP provides a retraction-free method for solving DC programs with second-order cone (SOC) constraints. 

We evaluate iPLNEP against the DCA proposed by L\'{o}pez et al. \cite{Lopez2018} on the datasets for SVM classification, including those from the UCI Machine Learning Repository and high-dimensional gene-expression
datasets used in L\'{o}pez et al. \cite{Lopez2018}. The DCA
solves the following SOCP at each iteration using SeDuMi \cite{Sturm1999}: 
\begin{equation}\label{dcsub-rSOCP}
\begin{aligned}
 &\min_{(x_1,x_2)\in \mathbb{R}\times\mathbb{R}^{n}, y\in\mathbb{R}^n} \frac{1}{2}\|x_2\|_2^2 + \mu\langle e,y\rangle - \mu\langle\zeta_2^k,x_2-x_2^k\rangle -\mu\|x_2^k\|_2\\
 &\qquad\quad\ {\rm s.t.}\ \ 1+\kappa_1\|D^{\top}_1x_2\|_2\le s_1^{\top}x_2+x_1 \\
 &\qquad\qquad\quad  1+\kappa_2\|D^{\top}_2x_2\|_2\le -s_2^{\top}x_2-x_1,\\
 &\qquad\qquad\quad\ -y\le x_{2}\le y.
 \end{aligned}
\end{equation}
We terminate iPLNEP if
$\big(\sum_{i=1}^l{\rm dist}^2(g_i(x^k),K^{m_i})\big)^{1/2}<10^{-5}$ 
and either $\|v^k\|_2<10^{-4}$ or $\varepsilon_k<10^{-5}$, or if the iteration count reaches $5000$. For DCA, we use SeDuMi
with its default settings. DCA is terminated if
$\|x^{k+1}\!-\!x^k\|_2<10^{-4}$ or the maximum of $500$ iterations is
reached.
\begin{table}[H]
\centering
\caption{Numerical comparison of iPLNEP and DCA for SVM classification}
\label{svm-tab}
\small
\setlength{\tabcolsep}{5pt}
\renewcommand{\arraystretch}{1.02}
\begin{tabular}{c|c|c|ccccc}
\hline
Dataset & $(N_s,N_d)$ & Solver & ACC (\%) & obj & spar & infeas & time (s) \\
\hline
\multirow{3}*{Breast cancer} & \multirow{3}*{(569,30)}
 & iPLNEPsp & 94.73 & 9.5823 & 0.3667 & 3.56e-6 & 0.18 \\
 & & iPLNEPac & 94.73 & 9.5824 & 0.3667 & 2.09e-6 & 0.14 \\
 & & DCA      & 94.73 & 9.5823 & 0.3667 & 2.07e-9 & 0.20 \\
 \hline
\multirow{3}*{Pima diabetes} & \multirow{3}*{(768,8)}
 & iPLNEPsp & 75.39 & 176.5721 & 0.1250 & 1.89e-6 & 0.52 \\
 & & iPLNEPac & 75.39 & 176.5722 & 0.1250 & 1.76e-6 & 0.63 \\
 & & DCA      & 75.39 & 176.5728 & 0.0000 & 3.46e-9 & 0.07 \\
\hline
\multirow{3}*{German Credit} & \multirow{3}*{(1000,24)}
 & iPLNEPsp & 73.40 & 76.3063 & 0.0000 & 1.75e-6 & 0.32 \\
 & & iPLNEPac & 73.40 & 76.3063 & 0.0000 & 1.99e-6 & 0.31 \\
 & & DCA      & 73.40 & 76.3065 & 0.0000 & 4.15e-9 & 0.09 \\
\hline
\multirow{3}*{Splice} & \multirow{3}*{(1000,60)}
 & iPLNEPsp & 82.40 & 9.5355 & 0.4667 & 3.42e-7 & 0.15 \\
 & & iPLNEPac & 82.40 & 9.5355 & 0.4667 & 2.02e-7 & 0.27 \\
 & & DCA      & 82.40 & 9.5355 & 0.4667 & 3.12e-9 & 0.18 \\
\hline
\multirow{3}*{Colorectal} & \multirow{3}*{(62,2000)}
 & iPLNEPsp & 90.32 & 5.2235 & 0.9570 & 5.59e-17 & 3.06 \\
 & & iPLNEPac & 90.32 & 5.2235 & 0.9570 & 6.08e-17 & 3.78 \\
 & & DCA      & 90.32 & 5.2235 & 0.9570 & 6.13e-9 & 5.88 \\
\hline
\multirow{3}*{Gravier} & \multirow{3}*{(168,2905)}
 & iPLNEPsp & 94.05 & 20.0145 & 0.9091 & 8.43e-11 & 6.68 \\
 & & iPLNEPac & 94.05 & 20.0145 & 0.9091 & 1.01e-10 & 10.01 \\
 & & DCA      & 94.05 & 20.0145 & 0.9091 & 5.52e-9 & 32.21 \\
\hline
\multirow{3}*{Lymphoma} & \multirow{3}*{(96,4026)}
 & iPLNEPsp & 96.88 & 0.9305 & 0.9945 & 5.79e-17 & 4.93 \\
 & & iPLNEPac & 96.88 & 0.9305 & 0.9945 & 4.29e-12 & 6.29 \\
 & & DCA      & 96.88 & 0.9305 & 0.9945 & 2.02e-9 & 62.24 \\
\hline
\multirow{3}*{Pomeroy} & \multirow{3}*{(66,7128)}
 & iPLNEPsp & 100.00 & 8.6054 & 0.9773 & 1.09e-9 & 53.26 \\
 & & iPLNEPac & 100.00 & 8.6054 & 0.9773 & 6.14e-9 & 80.88 \\
 & & DCA      & 100.00 & 8.6054 & 0.9771 & 4.60e-9 & 19.18 \\
\hline
\multirow{3}*{Shipp} & \multirow{3}*{(77,7219)}
 & iPLNEPsp & 97.40 & 4.7117 & 0.9892 & 4.82e-17 & 15.93 \\
 & & iPLNEPac & 97.40 & 4.7117 & 0.9892 & 4.55e-17 & 19.84 \\
 & & DCA      & 97.40 & 4.7117 & 0.9892 & 2.56e-9 & 28.25 \\
\hline
\multirow{3}*{West} & \multirow{3}*{(49,7219)}
 & iPLNEPsp & 100.00 & 5.5751 & 0.9870 & 2.30e-17 & 13.51 \\
 & & iPLNEPac & 100.00 & 5.5751 & 0.9870 & 2.52e-17 & 16.71 \\
 & & DCA      & 100.00 & 5.5751 & 0.9870 & 4.97e-9 & 14.86 \\
\hline
\end{tabular}
\end{table}

Table~\ref{svm-tab} reports results averaged over \textbf{20} random starting points for each dataset with $\mu=0.5$. Here, $N_s$ and $N_d$ denote the numbers of samples
and features, respectively, and classification performance is
measured by accuracy (ACC). All three solvers achieve identical ACC
values at the reported precision and nearly identical sparsity
levels, except on Pima diabetes (Data~2), where DCA produces no zero
components. iPLNEPsp requires less CPU time than DCA on Data~1,
4--7, and 9--10, but more on Data~2--3 and~8. The two iPLNEP
variants achieve comparable objective values and sparsity levels,
while iPLNEPsp is faster on eight of the ten datasets. As in
Section~\ref{sec6.3}, iPLNEPac shows no consistent computational
advantage over iPLNEPsp.
\subsubsection{Synthetic nonconvex SOCPs}\label{sec6.5.1}

All the preceding test instances involve a quadratic $f$. In this section, we consider a synthetic nonconvex SOCP with an $\ell_1$-norm regularized quartic polynomial objective function, subject to a generalized sphere constraint and an SOC constraint, formulated as
\begin{equation}\label{test-socp1}
\begin{aligned}
&\min_{(x_1,x_2)\in\mathbb{R}\times\mathbb{R}^{n}}\underbrace{\sum_{i=1}^n\big(\frac{1}{4}a_ix_{2i}^4+\frac{1}{3}b_ix_{2i}^3\big)+\frac{1}{2}x^{\top}_2A_0x_2+c_0x_1+c_1^{\top}x_2}_{f(x)}+\mu\|x_2\|_1\\
&\qquad{\rm s.t.}\ \ \|x_2\|_2\le t_0,\, \|x_2\|_2\le x_1,\,x_2^{\top}Bx_2= 1,
\end{aligned}
\end{equation}
where the vectors \(a,b,c_1\in\mathbb{R}^n\) are generated using \lstinline|randn(n,1)|, and the scalar \(c_0\) is generated using \lstinline|abs(randn(1))|. The matrix $A_0$ is generated by \lstinline|A0 = randn(n); A0 = (A0+A0')/2|, and the matrix $B$ is generated by \lstinline|B0 = 0.05*randn(n,m); B = B0*B0'+eye(n)| with $m=\lceil 0.2n\rceil$. We set $t_0=0.1/\lambda_{\rm max}(B)+0.9/\lambda_{\rm min}(B)$. Problem \eqref{test-socp1} is a special case of \eqref{Rcprob} with $\mathcal{M}=\mathbb{R}\times\big\{z\in\mathbb{R}^n\mid z^{\top}Bz=1\big\}$, $K=\mathbb{R}_{+}\times K^{n+1},h\equiv 0$, and $g(x)=(t_0^2-\!\|x_2\|_2^2,x_1,x_2)$. In this setting, iPLNEP provides a Riemannian NEP framework for solving nonsmooth optimization problems involving SOC and generalized sphere constraints. 

We evaluate the performance of iPLNEP for solving  \eqref{test-socp1} against the open-source solver IPOPT developed in W{\"a}chter et al. \cite{Wachter2006}, applied to the following nonlinear programming reformulation: 
\begin{equation*}
 \min_{(x_1,x_2)\in\mathbb{R}\times\mathbb{R}^{n}\atop y\in\mathbb{R}^n}\Big\{f(x)+\mu\langle e,y\rangle\ \ {\rm s.t.}\ \ \|x_2\|_2^2\le t_0^2,\|x_2\|_2^2\le x_1^2,x_1\ge 0,-y\le x_2\le y,x_2^{\top}Bx_2 = 1\Big\}.
\end{equation*}
For IPOPT, we use its default algorithmic settings and impose a CPU-time limit of $2000$ seconds per run. We terminate iPLNEP when $(\sum_{i=1}^l{\rm dist}^2(g_i(x^k),K^{m_i}))^{1/2}< 10^{-5}$ and either $\|v^k\|_2<10^{-4}$ or $\varepsilon_k< 10^{-5}$, or when the iteration count reaches $15000$. 

\begin{table*}[h]
\renewcommand{\arraystretch}{1.0}
\centering
\caption{Numerical comparison of iPLNEP and IPOPT for \eqref{test-socp1}}
\label{socp1-tab}
\scalebox{0.8}[0.8]{
\begin{tabular}{c|c|cccccc} 
\hline    
\multicolumn{2}{c|}{$(n, \mu)$}& (100, 0.5)& (500, 0.5)&(3000, 0.5)& (100, 1.0)& (500, 1.0)&(3000, 1.0)\\
\hline
\multirow{6}*{iPLNEPsp} 
&obj& -7.3252& -16.2634& -37.5605& -4.5390& -10.3857& -23.3649\\ 
&spar& 0.3070& 0.3209& 0.3142& 0.5740& 0.5948& 0.5797\\ 
&time (s)& 0.03& 0.08& 18.34& 0.03& 0.13& 12.24\\ 
&iter& 228.8& 243.8& 1894.7& 185.8& 287.5& 1215.7\\ 
&initer& 2.3& 2.9& 9.4& 3.6& 4.8& 14.2\\ 
&infeas& 2.60e-6& 1.19e-6& 5.01e-8& 3.03e-6& 8.28e-7& 7.66e-8\\ 
\hline  
\multirow{6}*{iPLNEPac} 
&obj& -7.3253& -16.2636& -37.5605& -4.5352& -10.3857& -23.3649\\ 
&spar& 0.3130& 0.3222& 0.3142& 0.5735& 0.5948& 0.5797\\ 
&time (s)& 0.06& 0.12& 20.29& 0.07& 0.18& 13.08\\ 
&iter& 289.4& 326.6& 2024.4& 281.4& 350.1& 1215.2\\ 
&initer& 4.1& 3.5& 10.1& 5.7& 6.1& 15.0\\ 
&infeas& 1.80e-6& 2.33e-6& 4.54e-7& 2.29e-6& 1.70e-6& 7.18e-8\\ 
\hline  
\multirow{5}*{IPOPT} 
&obj& -7.3253& -16.2637&\quad --& -4.5390& -10.3858&\quad --\\ 
&spar& 0.3125& 0.3205&\quad--& 0.5730& 0.5942&\quad--\\ 
&time (s)& 1.72& 193.51&\quad--& 1.71& 193.33&\quad --\\   
&iter& 18.6& 19.2&\quad--& 19.5& 19.6&\quad--\\
&infeas& 2.38e-8& 2.01e-8&\quad--& 3.19e-8& 3.09e-8&\quad--\\ 			
\hline
\end{tabular}
}
\end{table*}

Table~\ref{socp1-tab} reports the average results over
\textbf{20} independent runs for each instance, where ``--'' indicates
that IPOPT exceeded the CPU-time limit of $\textbf{2000}$ seconds in
all runs. On instances solved by IPOPT within the time limit, both iPLNEP variants achieve nearly identical
objective values and similar sparsity levels, while requiring
substantially less CPU time. For $n=3000$, the average CPU times of the iPLNEP variants range
from approximately $12$ to $21$ seconds, whereas IPOPT
exceeds the time limit. These results indicate that iPLNEP scales well on the tested instances. Although IPOPT attains smaller constraint violations on the instances it solves within the time limit, all reported constraint
violations are below $5\times10^{-6}$.
\section{Conclusion}\label{sec7}

We have developed an inexact proximal-linearized NEP method for problem \eqref{Rcprob} based on its partial NEP formulation \eqref{penprob}. The method incorporates computable inexactness criteria for strongly convex tangent-space subproblems, adaptive proximal and penalty parameter updates, and a retraction-based acceptance mechanism. Under boundedness of the iterates and penalty parameters, we establish an $O(\epsilon^{-2})$ worst-case outer-iteration complexity bound for finding an $\epsilon$-approximate KKT point. Building on this result, we develop a complexity-transfer framework that translates uniform decay estimates for subproblem objective gaps and KKT residuals into total inner-iteration complexity bounds. For dual inner solvers, the required primal objective-gap and KKT-residual estimates are derived from dual objective-gap decay. Applying this framework yields an $O(\epsilon^{-4})$ overall oracle complexity bound with accelerated semi-proximal ADMM for the subproblems or FISTA for their duals, and the improved  $O(\epsilon^{-2}\log(\epsilon^{-1}))$ bound with semi-proximal ADMM under an additional uniform error-bound condition. This condition is verified for a nonsmooth Stiefel-manifold example in Appendix A. Under the same boundedness assumptions and the KL property of the potential function \(\Xi_{\overline c}\), we prove that the entire iterate sequence converges to a KKT point of \eqref{Rcprob}. Numerical experiments reported in the electronic companion further illustrate the practical implementability of iPLNEP using semi-proximal ADMM and its accelerated variant.

Future research may investigate sharper overall oracle complexity bounds without additional uniform error-bound assumptions, as well as extensions to unbounded penalty parameters.

\medskip

\textbf{Acknowledgments.} The authors gratefully acknowledge Prof. Nachuan Xiao from The Chinese University of Hong Kong, Shenzhen, Prof. Wen Huang from Xiamen University, and Dr. Linglingzhi Zhu from Georgia Institute of Technology for kindly providing the implementations of PenCPG, I-AManPG, and LSALM, respectively, which were helpful for the numerical evaluation of the proposed method.

\textbf{Declarations.} The authors used ChatGPT 6.0 to assist with the construction of Example~\ref{example-UEB} in Appendix~A  and English editing. All AI-assisted content has been verified by the authors, who take full responsibility for the correctness and final content of the paper.

\bibliography{reference}
\bibliographystyle{plain}

\medskip
\noindent
{\bf\large Appendix A.} We present a nonsmooth optimization problem over the Stiefel manifold with a nonnegativity constraint for which the uniform error-bound condition required by Corollary \ref{sPADMM-oracle} holds. We first establish a sufficient condition for the uniform
KKT error bound in terms of uniform quadratic growth
of the dual objective functions.
\begin{proposition}\label{dtwice-grow-kkt-bound}
 Let $D_k^*$ and $\mathcal{D}^*_k$ be the optimal value and solution set of \eqref{dsubprob}, respectively. Suppose that there exists a constant $e_D>0$ such that, for every $k\in\mathbb{N}$ and $\lambda\in {\rm dom}\,\psi_{\rho_k}^*$,
 \begin{equation}\label{Xi-tgrow}
 D_k(\lambda)-D_k^*\ge ({e_{D}}/{2}){\rm dist}(\lambda,\mathcal{D}^*_k)^2.
 \end{equation}
 Then, under Assumption \ref{ass1}, there exists a constant $c_{\rm err}>0$ such that ${\rm dist}\bigl(w,\mathcal R_k^{-1}(0)\bigr)
 \le c_{\rm err}\|\mathcal R_k(w)\|$ for all $k\in\mathbb{N}$ and all $w=(v,z,\lambda)\in T_{x^k}\mathcal{M}\times\mathbb{Z}\times\mathbb{Z}$ satisfying $\lambda\in\partial\psi_{\rho_k}(z)$.
\end{proposition} 
\begin{proof}
Fix any $k\in\mathbb{N}$ and any $w=(v,z,\lambda)\in T_{x^k}\mathcal{M}\times\mathbb{Z}\times\mathbb{Z}$ satisfying $\lambda\in\partial\psi_{\rho_k}(z)$. It follows from \eqref{KKT-res} that $[\mathcal{R}_k(w)]_1=\beta_kv+\mathcal{P}_{T_{x^k}\mathcal{M}}(\nabla \ell_k(0)+\mathcal{A}^*_k\lambda)$, $[\mathcal{R}_k(w)]_2=z-\mathcal{P}_{\psi_{\rho_k}}(z+\lambda)=0$, and $[\mathcal{R}_k(w)]_3=z-\mathcal{A}_kv-G(x^k,0)$. The inclusion $z\in\partial\psi_{\rho_k}^*(\lambda)$ then implies that
\[ [\mathcal{R}_k(w)]_3+\beta_k^{-1}\mathcal{A}_k[\mathcal{R}_k(w)]_1=z-G(x^k,0)+\beta_k^{-1}\mathcal{A}_k\mathcal{P}_{T_{x^k}\mathcal{M}}\big[\nabla \ell_k(0)+\mathcal{A}^*_k\lambda\big]\in\partial D_{k}(\lambda).
\]
Let $\lambda^*\in\mathcal{D}^*_k$ be such that $\|\lambda-\lambda^*\|={\rm dist}(\lambda,\mathcal{D}^*_k)$. The above inclusion and \eqref{Xi-tgrow} imply that 
\(
 \|\lambda-\lambda^*\|^2
 \le \frac{2}{e_D}(D_k(\lambda)-D_k^*\bigr)\le \frac{2}{e_D}\big\langle [\mathcal{R}_k(w)]_3+\beta_k^{-1}\mathcal{A}_k[\mathcal{R}_k(w)]_1,\lambda-\lambda^*\big\rangle.
\)
Consequently, we obtain
\begin{equation}\label{lambda-bound}
 \|\lambda-\lambda^*\|\le ({2}/e_D)\big[\|[\mathcal{R}_k(w)]_3\|+\beta_{\rm min}^{-1}\|\mathcal{A}_k\|\|[\mathcal{R}_k(w)]_1\|\big].
\end{equation}
Let $v^*:=-\beta_k^{-1}\mathcal{P}_{T_{x^k}\mathcal{M}}(\nabla \ell_k(0)+\mathcal{A}_k^*\lambda^*)$ and $z^*:=\mathcal{A}_kv^*+G(x^k,0)$. Then $(v^*,z^*,\lambda^*)\in\mathcal R_k^{-1}(0)$. Indeed, by the definition of $\mathcal{R}_k$, it suffices to show that $z^*\in\partial\psi_{\rho_k}^*(\lambda^*)$. Since $\lambda^*\in\mathcal{D}^*_k$, the optimality condition of
\eqref{dsubprob} at $\lambda^*$ yields $0\in\partial\psi_{\rho_k}^*(\lambda^*)+\beta_k^{-1}\mathcal{A}_k\mathcal{P}_{T_{x^k}\mathcal{M}}(\mathcal{A}_k^*\lambda^*+\nabla \ell_k(0))-G(x^k,0)$. Then $z^*=\mathcal{A}_kv^*+G(x^k,0)\in\partial\psi_{\rho_k}^*(\lambda^*)$. Moreover, the expressions of $[\mathcal{R}_k(w)]_1$ and $[\mathcal{R}_k(w)]_3$ imply that
\begin{equation*}
 \|v-v^*\|\le\beta_{\min}^{-1}
\bigl(\|[\mathcal{R}_k(w)]_1\|+\|\mathcal{A}_k\|\|\lambda-\lambda^*\|\bigr),\, \|z-z^*\|\le\|[\mathcal{R}_k(w)]_3\|+\|\mathcal{A}_k\|\|v-v^*\|.
\end{equation*} 
Combining the two inequalities with \eqref{lambda-bound} and
using $(v^*,z^*,\lambda^*)\in\mathcal R_k^{-1}(0)$ and $[\mathcal R_k(w)]_2=0$, we get 
\[
 {\rm dist}\bigl(w,\mathcal R_k^{-1}(0)\bigr)
 \le \big[L_{\lambda}^2+L_v^2+(1+c_{\mathcal{A}}L_v)^2\big]^{1/2}\|\mathcal{R}_k(w)\|,
\]
where $c_{\mathcal{A}}\!:=\sup_{k\in\mathbb{N}}\|\mathcal{A}_k\|<\infty$, $L_{\lambda}:=({2}/{e_D})(1+\beta_{\min}^{-2}c_{\mathcal{A}}^2)^{1/2}$ and $L_v:=\beta_{\min}^{-1}(1+c_{\mathcal{A}}L_{\lambda})$. 
\end{proof}
\begin{example}\label{example-UEB}
Let $A\in\mathbb{R}^{n\times n}$ be a symmetric matrix, $C\in\mathbb{R}^{n\times p}$, and $d\in \mathbb{R}^n$ with $\|d\|=1$. Consider 
\begin{equation}\label{UB-example}
 \min_{X\in{\rm St}(n,p)} \frac{1}{2}{\rm tr}(X^{\top}AX)-\langle C, X\rangle+\mu\|X^{\top}d\|_1,\ \ {\rm s.t.}\ \ g(X):=X^{\top}d\ge 0.
\end{equation}
where $\mu>0$ is a regularization parameter. If the norm $\vert\!\|\cdot\|\!\vert$ is chosen as the entrywise $\ell_1$-norm, under Assumption~\ref{ass1}, there exists a constant $c_{\rm err}>0$ such that ${\rm dist}\bigl(W,\mathcal R_k^{-1}(0)\bigr)\le c_{\rm err}\|\mathcal R_k(W)\|$ for every $k\in\mathbb{N}$ and every $W=(V,z,\lambda)\in T_{X^k}{\rm St}(n,p)\times\mathbb{R}^{2p}\times\mathbb{R}^{2p}$ satisfying $\lambda\in\partial\psi_{\rho_k}(z)$.
\end{example}
\begin{proof}
Fix any $k\in \mathbb{N}$. Let $d^k:=(X^k)^{\top}d,\,C^k:=AX^k-C$ and $\mathcal{P}_k:=\mathcal{P}_{T_{X^k}{\rm St}(n,p)}$. For this example, 
$\psi_{\rho_k}(Z,y)
=\mu\|Z^\top d\|_1+\rho_k\sum_{i=1}^p\max\{-y_i,0\}$ for $(Z,y)\in\mathbb{R}^{n\times p}\times\mathbb{R}^p$. Its conjugate is
$\psi_{\rho_k}^*(U,\eta)=\delta_{\{du^\top:\,u\in[-\mu,\mu]^p\}}(U)
+\delta_{[-\rho_k,0]^p}(\eta)$. Consequently, every dual variable in the effective domain can be represented as $(U,\eta)=(d\lambda_1^\top,\lambda_2)$, where $\lambda_1,\lambda_2\in\mathbb{R}^p$.
Under this parametrization, the dual subproblem \eqref{dsubprob} reduces to
$\min_{\lambda\in\mathbb{R}^{2p}}D_k(\lambda)$, where
\[
D_k(\lambda)
=\frac{1}{2\beta_k}
\|\mathcal{P}_k(C^k+d(\lambda_1+\lambda_2)^\top)\|_F^2
-(d^k)^\top(\lambda_1+\lambda_2)
+\delta_{[-\mu,\mu]^p}(\lambda_1)
+\delta_{[-\rho_k,0]^p}(\lambda_2).
\]
 Let $B_k\!:=[-\mu-\rho_k,\mu]^p$. The dual subproblem can equivalently be reformulated as
 \begin{equation}
 \min_{s\in B_k}\widetilde{D}_k(s)\!:=\frac{1}{2\beta_k}\|\mathcal{P}_k(C^k\!+\!ds^{\top})\|^2_F-(d^k)^{\top}s. 
 \end{equation}
 We first establish a uniform quadratic-growth bound
for $\widetilde{D}_k$ and then transfer it to $D_k$.
   
 \noindent
 {\bf Step 1: Uniform quadratic growth of $\widetilde{D}_k$.} Since $X^k\in{\rm St}(n,p)$, combining $\|d\|=1$ with the expression of $\mathcal{P}_k$ (see Absil et al. \cite[Example 3.6.2]{Absil2008}) yields, for any $s\in\mathbb{R}^p$, 
 \begin{equation}\label{appendix-equa1}  
 \|\mathcal{P}_k(ds^{\top})\|^2_F=\|ds^{\top}\|^2_F-\|X^k{\rm sym}((X^k)^{\top}(ds^{\top}))\|^2_F=\big(1-{\|d^k\|^2}/{2}\big)\|s\|^2-\frac{1}{2}|(d^k)^{\top}s|^2.
 \end{equation}
 Fix any $s\in B_k$. Pick any $s^*\in S_k^*:=\arg\min_{s\in B_k}\widetilde{D}_k(s)$. Define $\Delta_{\rm dobj}:=\widetilde{D}_k(s)-\widetilde{D}_k(s^*)\ge 0$ and $\Delta_{s}:=s-s^*$. The optimal condition of the problem $\min_{s\in B_k}\widetilde{D}_k(s)$ yields 
\begin{equation}\label{s_kge0}
 t_k\!:=\langle\nabla \widetilde{D}_k(s^*),\Delta_{s}\rangle=\frac{1}{\beta_k}\langle \mathcal{P}_k(C^k+d(s^*)^{\top}),\mathcal{P}_k(d\Delta_{s}^{\top})\rangle-(d^k)^{\top}\Delta_{s}\ge 0.
\end{equation}
Moreover, by the definitions of $\Delta_{\rm dobj}$ and $\widetilde{D}_k$, we calculate that $\Delta_{\rm dobj}=t_k+\frac{1}{2\beta_k}\|\mathcal{P}_k(d\Delta_{s}^{\top})\|^2_F$. Then 
\begin{equation}\label{s_kDeltaD}
 0\le t_k\le \Delta_{\rm dobj}\ \ {\rm and}\ \  \|\mathcal{P}_k(d\Delta_{s}^{\top})\|^2_F\le 2\beta_k\Delta_{\rm dobj}\le 2\overline{\beta}\Delta_{\rm dobj}.
\end{equation}
 Let $e_{A}\!:=\max_{X\in{\rm St}(n,p)}\|AX-C\|_F$, $e_{\mu}\!:=\sqrt{p}(\mu+\overline{\rho})$ and $e_P\!:=(e_{\mu}+e_A)/\beta_{\min}$. Since $\|s^*\|\le e_{\mu}$ and $\|\mathcal{P}_k\|\le 1$, combining inequality \eqref{s_kge0} with the first inequality of \eqref{s_kDeltaD} and using $\|C^k\|_F\le e_{A}$ yields 
$\big|(d^k)^{\top}\Delta_{s}\big|\le \Delta_{\rm dobj}+e_P\|\mathcal{P}_k(d\Delta_{s}^{\top})\|_F$. In addition, substituting \(s=\Delta_s\) into \eqref{appendix-equa1} and noting that $\|d^k\|\le 1$, we have $\|\Delta_{s}\|^2\le 2\|\mathcal{P}_k(d\Delta_{s}^{\top})\|^2_F+|(d^k)^{\top}\Delta_{s}|^2$. From the two inequalities, it follows that
\begin{equation}\label{bound-lamtilde1}
\|\Delta_{s}\|^2\le 2\|\mathcal{P}_k(d\Delta_{s}^{\top})\|^2_F+2\Delta_{\rm dobj}^2+2e_P^2\|\mathcal{P}_k(d\Delta_{s}^{\top})\|_F^2\overset{\eqref{s_kDeltaD}}{\le} 4\overline{\beta}(1+e^2_P)\Delta_{\rm dobj}+2\Delta_{\rm dobj}^2.
\end{equation}
On the other hand, the definition $\widetilde{D}_k$ implies that
$-e_{\mu}\le\widetilde{D}_k(s)\le e_{\mu}+(e_{\mu}+e_A)^2/(2\beta_{\min})$ for all $s\in B_k$. The definition of $\Delta_{\rm dobj}$ implies that 
$0\le\Delta_{\rm dobj}\le e_{\Delta}\!:=2e_{\mu}+\frac{(e_{\mu}+e_A)^2}{2\beta_{\min}}$. It follows from \eqref{bound-lamtilde1} that 
$\|s-s^*\|^2\le e_{\widetilde{D}}(\widetilde{D}_k(s)-\widetilde{D}_k^*)$, where $e_{\widetilde{D}}\!:=4\overline{\beta}(1+e^2_P)+2e_{\Delta}$ and $\widetilde{D}_k^*\!:=\min_{s\in B_k}\widetilde{D}_k(s)$.
This inequality, by the arbitrariness of $s\in B_k$ and $s^*\in S_k^*$ and the convexity of $\widetilde{D}_k$, implies that
\begin{equation}\label{Dtilde-QG}
 \widetilde{D}_k(s)-\widetilde{D}_k^*\ge e_{\widetilde{D}}^{-1}{\rm dist}^2(s,S_k^*)\quad{\rm for\ all}\ s\in B_k.
\end{equation}

\noindent
{\bf Step 2: Uniform quadratic growth of $D_k$.} Fix any $\lambda=(\lambda_1,\lambda_2)\in [-\mu,\mu]^p\times[-\rho_k,0]^p$ and let $s:=\lambda_1+\lambda_2$. Let $s^*\in S_k^*$ be such that $\|s-s^*\|={\rm dist}(s,S_k^*)$. We claim that there exists a vector $\lambda^*=(\lambda_1^*,\lambda_2^*)\in [-\mu,\mu]^p\times[-\rho_k,0]^p$ such that
\begin{equation}\label{lam1-lam2}
 s^*=\lambda_1^*+\lambda_2^*\ \ {\rm and}\ \ \|\lambda_1-\lambda_1^*\|^2+\|\lambda_2-\lambda_2^*\|^2\le \|s-s^*\|^2={\rm dist}^2(s,S_k^*).
\end{equation}
For each $i\in[p]_+$, let $\delta_i\!:=s^*_i-s_i,\lambda^*_{1,i}\!:=\mathcal{P}_{[-\mu,\mu]}(\lambda_{1,i}+\delta_i)$ and $\lambda_{2,i}^*:=s^*_i-\lambda^*_{1,i}$. From $\lambda_{1,i}+\delta_i=s_i^*-\lambda_{2,i}$ and $s^*\in [-\mu-\rho_k,\mu]^p$, it follows that $\lambda^*=(\lambda_1^*,\lambda_2^*)\in [-\mu,\mu]^p\times[-\rho_k,0]^p$ and $s^*=\lambda_1^*+\lambda_2^*$. Since $\lambda_1\in [-\mu,\mu]^p$, if $\delta_i\ge 0$, then $\lambda^*_{1,i}:=\mathcal{P}_{[-\mu,\mu]}(\lambda_{1,i}\!+\!\delta_i)=\min\{\lambda_{1,i}\!+\!\delta_i,\mu\}$, which implies that $\lambda^*_{1,i}-\lambda_{1,i}=\min\{\delta_i,\mu-\lambda_{1,i}\}\in[0,\delta_i]$. Therefore, $\lambda^*_{2,i}-\lambda_{2,i}=\delta_i-(\lambda^*_{1,i}-\lambda_{1,i})\ge 0$, which implies that $(\lambda^*_{1,i}-\lambda_{1,i})(\lambda^*_{2,i}-\lambda_{2,i})\ge 0$. Similarly, if $\delta_i< 0$, we also have $(\lambda^*_{1,i}-\lambda_{1,i})(\lambda^*_{2,i}-\lambda_{2,i})\ge 0$. Then
\[
\|\lambda_1-\lambda_1^*\|^2+\|\lambda_2-\lambda_2^*\|^2=\sum_{i=1}^p\left[\delta_i^2-2(\lambda^*_{1,i}-\lambda_{1,i})(\lambda^*_{2,i}-\lambda_{2,i})\right]\le\sum_{i=1}^p\delta_i^2=\|s-s^*\|^2.
\]
Thus, we prove the claim. Let $D_k^*:=\min_{\lambda\in\mathbb{R}^{2p}}D_k(\lambda)$. Then  $D_k^*=D_k(\lambda^*)=\widetilde{D}_k(s^*)=\widetilde{D}_k^*$, which implies that $\lambda^*\in \Lambda^*_k:=\arg\min_{\lambda\in\mathbb{R}^{2p}}D_k(\lambda)$. From the second inequality of \eqref{lam1-lam2} and the arbitrariness of $\lambda\in [-\mu,\mu]^p\times[-\rho_k,0]^p$, we have ${\rm dist}(\lambda,\Lambda^*_k)\le {\rm dist}(s,S_k^*)$ for all $\lambda\in [-\mu,\mu]^p\times[-\rho_k,0]^p$. Since $D_k(\lambda)=\widetilde{D}_k(s)$ for any $\lambda\in [-\mu,\mu]^p\times[-\rho_k,0]^p$ and $s\in B_k$ with $s=\lambda_1+\lambda_2$, it follows from \eqref{Dtilde-QG} and $D_k^*=\widetilde{D}_k^*$ that
$D_k(\lambda)-D_k^*\ge e_{\widetilde{D}}^{-1}{\rm dist}(\lambda,\Lambda_k^*)^2$ for all $\lambda\in [-\mu,\mu]^p\times[-\rho_k,0]^p$.

Since $\|d\|=1$, the parametrization $(\lambda_1,\lambda_2)\mapsto
(d\lambda_1^\top,\lambda_2)$ preserves Euclidean distances. Hence, the uniform quadratic-growth bound established above also holds
for the original dual objective on its effective domain.
The desired uniform KKT error bound then follows from
Proposition~\ref{dtwice-grow-kkt-bound}.
\end{proof}

\end{document}